\documentclass[10pt]{amsart}
\usepackage{amssymb}
\usepackage{dsfont}
\usepackage{amscd}
\usepackage[mathscr]{euscript} % for the power-set P
\usepackage[all]{xy}
\usepackage{url}
\usepackage{stmaryrd}
\usepackage{comment}
\usepackage[retainorgcmds]{IEEEtrantools}
\usepackage{hyperref}
\usepackage{graphicx}
\usepackage{mathtools}
\usepackage{color}
\usepackage{centernot}
\usepackage{marvosym}

\title{Model theoretic dynamics in Galois fashion}
\author[D. M. HOFFMANN]{Daniel Max Hoffmann$^{\dagger}$}
\thanks{2010 \textit{Mathematics Subject Classification}. Primary 03C95; Secondary 03C50, 03C52}
\thanks{\textit{Key words and phrases}. model companions, automorphisms group, group actions.}
\thanks{$^{\dagger}$SDG. Part of the work on this paper was conducted during the author's internship at the Warsaw Center of Mathematics and Computer Science.
Supported by NCN grants 2016/20/T/ST1/00482 and 2016/21/N/ST1/01465,
and partially supported by NCN grant 2015/19/B/ST1/01150.}
\address{$^{\dagger}$Instytut Matematyki\\
Uniwersytet Warszawski\\
Warszawa\\
Poland}
\email{daniel.max.hoffmann@gmail.com}
\urladdr{https://www.researchgate.net/profile/Daniel\_Hoffmann8}

\DeclareMathOperator{\acl}{acl} \DeclareMathOperator{\dcl}{dcl} 
 \DeclareMathOperator{\aut}{Aut} \DeclareMathOperator{\id}{id}
\DeclareMathOperator{\cl}{cl}

\DeclareMathOperator{\im}{im}  
  
\DeclareMathOperator{\td}{trdeg} 
 \DeclareMathOperator{\theo}{Th}\DeclareMathOperator{\alg}{alg}

 \DeclareMathOperator{\eq}{eq}
 
\DeclareMathOperator{\tp}{tp}

\DeclareMathOperator{\ev}{ev}

\DeclareMathOperator{\ddf}{DF}\DeclareMathOperator{\dcf}{DCF}\DeclareMathOperator{\scf}{SCF}
\DeclareMathOperator{\acf}{ACF}

\DeclareMathOperator{\tcf}{TCF}

\DeclareMathOperator{\perf}{perf}
\DeclareMathOperator{\Dgat}{Diag^{\text{at}}}

\DeclareMathOperator{\mc}{mc}

\DeclareMathOperator{\Cn}{Cn}
\DeclareMathOperator{\qftp}{qftp}
\DeclareMathOperator{\cb}{Cb}
\DeclareMathOperator{\mlt}{mlt}
\DeclareMathOperator{\fo}{fo}

\newtheorem{theorem}{Theorem}[section]
\newtheorem{prop}[theorem]{Proposition}
\newtheorem{lemma}[theorem]{Lemma}
\newtheorem{cor}[theorem]{Corollary}
\newtheorem{fact}[theorem]{Fact}
\theoremstyle{definition}
\newtheorem{definition}[theorem]{Definition}
\newtheorem{example}[theorem]{Example}
\newtheorem{remark}[theorem]{Remark}
\newtheorem{question}[theorem]{Question}
\newtheorem{conj}[theorem]{Conjecture}

\theoremstyle{remark}
\newtheorem*{theorem*}{Theorem}
\newtheorem*{cor*}{Corollary}

\theoremstyle{definition}
\theoremstyle{definition}

\theoremstyle{definition}

\theoremstyle{remark}

\makeatletter
\providecommand*{\cupdot}{%
  \mathbin{%
    \mathpalette\@cupdot{}%
  }%
}
\newcommand*{\@cupdot}[2]{%
  \ooalign{%
    $\m@th#1\cup$\cr
    \sbox0{$#1\cup$}%
    \dimen@=\ht0 %
    \sbox0{$\m@th#1\cdot$}%
    \advance\dimen@ by -\ht0 %
    \dimen@=.5\dimen@
    \hidewidth\raise\dimen@\box0\hidewidth
  }%
}

\providecommand*{\bigcupdot}{%
  \mathop{%
    \vphantom{\bigcup}%
    \mathpalette\@bigcupdot{}%
  }%
}
\newcommand*{\@bigcupdot}[2]{%
  \ooalign{%
    $\m@th#1\bigcup$\cr
    \sbox0{$#1\bigcup$}%
    \dimen@=\ht0 %
    \advance\dimen@ by -\dp0 %
    \sbox0{\scalebox{2}{$\m@th#1\cdot$}}%
    \advance\dimen@ by -\ht0 %
    \dimen@=.5\dimen@
    \hidewidth\raise\dimen@\box0\hidewidth
  }%
}
\makeatother

\def\Ind#1#2{#1\setbox0=\hbox{$#1x$}\kern\wd0\hbox to 0pt{\hss$#1\mid$\hss}
\lower.9\ht0\hbox to 0pt{\hss$#1\smile$\hss}\kern\wd0}

\def\ind{\mathop{\mathpalette\Ind{}}}

\def\notind#1#2{#1\setbox0=\hbox{$#1x$}\kern\wd0
\hbox to 0pt{\mathchardef\nn=12854\hss$#1\nn$\kern1.4\wd0\hss}
\hbox to 0pt{\hss$#1\mid$\hss}\lower.9\ht0 \hbox to 0pt{\hss$#1\smile$\hss}\kern\wd0}

\def\nind{\mathop{\mathpalette\notind{}}}

\begin{document}

\newcommand{\twoc}[3]{ {#1} \choose {{#2}|{#3}}}
\newcommand{\thrc}[4]{ {#1} \choose {{#2}|{#3}|{#4}}}
\newcommand{\Rr}{{\mathds{R}}}
\newcommand{\Kk}{{\mathds{K}}}

\newcommand{\dlog}{\mathrm{ld}}
\newcommand{\ga}{\mathbb{G}_{\rm{a}}}
\newcommand{\gm}{\mathbb{G}_{\rm{m}}}
\newcommand{\gaf}{\widehat{\mathbb{G}}_{\rm{a}}}
\newcommand{\gmf}{\widehat{\mathbb{G}}_{\rm{m}}}
\newcommand{\gdf}{\mathfrak{g}-\ddf}
\newcommand{\gdcf}{\mathfrak{g}-\dcf}
\newcommand{\fdf}{F-\ddf}
\newcommand{\fdcf}{F-\dcf}
\newcommand{\mw}{\scf_{\text{MW},e}}

\begin{abstract}
We fix a monster model $\mathfrak{D}$ of some stable theory and investigate substructures of $\mathfrak{D}$ which are existentially closed as structures additionally equipped with an action of a fixed group $G$.
We describe them as PAC substructures of $\mathfrak{D}$ and obtain results related to Galois theory.

Assuming that some
class of these existentially closed substructures is elementary,
we show that, under the assumption of having bounded models, 
its theory is simple and eliminates quantifiers up to some existential formulas. Moreover, this theory codes finite sets and allows a geometric elimination of imaginaries, but not always a weak elimination of imaginaries.
\end{abstract}
\maketitle

\tableofcontents
\section{Introduction}
\subsection{Motivation}
We start in a very unusual way. Willem de Sitter gave a solution for Einstein's general relativity, which describes a space without any matter, but in the movement. It was completely opposite to the original Einstein's solution, i.e. to a static space, but containing the matter. Further studies allowed to join both concepts into one, which better or worse describes our perception of the Universe. The following paper shares something with de Sitter's solution: we describe movement without matter, i.e. we do not pay enough attention to prove existence of considered theories on a general level.
There is a lot of important theories,
but many of them have static models.
% which have models, but many of them are static. 
Our general aim for this and future works is to develop model theoretic dynamics in a way similar to that what happened in algebra and algebraic dynamics.

Our main motivation arises from two papers: \cite{ChaPil} and \cite{nacfa}. The first one describes the model companion of substructures of a monster model of a stable theory, equipped with an additional automorphism. The second one is about the model companion of a theory of fields equipped with an action of a fixed finite group. In this paper (Subsection \ref{sec:stable.to.simple}), we achieve a generalization of both situations: we provide a description of a model companion of substructures of a monster model of an arbitrary stable theory, equipped with an action of a fixed group. Of course, there are some additional assumptions, but they are almost harmless and there are counterexamples for the corresponding statements in the full generality, i.e. without our almost harmless assumptions.

Our second, but also very important, motivation is to prepare background for a more sophisticated research. It was shown that there is a link between
existence of the model companion of a theory with an additional automorphism and the \emph{finite cover property} (\cite{balshe}).
\begin{comment}
The theory $T$ does not have the \emph{finite cover property} if and only if  for every $\mathcal{L}$-formula $\varphi(x,y)$,	there exists a natural number $k_{\varphi}$ such that for every $A\subseteq\mathfrak{C}$ and every
$$p\subseteq\lbrace \varphi(x,a),\;\neg\varphi(x,a)\;|\;a\in A\rbrace,$$
it follows
$$(\forall p_0\subseteq p,\; |p_0|<k_{\varphi})\big(\,p_0\text{ is consistent}\,\big)\quad\Rightarrow\quad p\text{ is consistent}.$$
\end{comment}
The aforementioned result closes, in some sense, ``the storyline of existence" of model companions of theories with an additional automorphism, which was originated by Kikyo in \cite{kikyo1} (and continued in \cite{kikpil} and \cite{kishe}). What is important for us is that the main counterexamples for existence of the model companion of a theory with an additional automorphism from \cite{kikyo1} are not counterexamples anymore in a more general context. 
More precisely, a theory of structures with extra automorphisms can be treated as a theory of structures with a group action. 
For example, we are interested in structures with a fixed number $n$ (which may be a cardinal) of automorphisms which satisfy some equalities. Such structures can be viewed as structures with an action of group $G$, where $G$ is given by the following presentation:
$$\langle \sigma_1,\ldots,\sigma_n\;\;|\;\; r_1,\ldots,r_m\rangle,$$
where $r_1,\ldots,r_m$ are relations corresponding to the equalities which automorphisms $\sigma_1,\ldots,\sigma_n$ should satisfy.
Therefore, we can treat a theory of structures with an additional automorphism as a theory of structures with a group action of the group $G=\mathbb{Z}$. Hence \cite{kikyo1} was motivated by non-existence of model companions for actions of $\mathbb{Z}$, but we show that model companions of those counterexample theories exist, if we consider actions of a finite group $G$ instead of actions of $\mathbb{Z}$ (Example \ref{rand.graph}, Example \ref{orders}, Remark \ref{Boole2}). 

Another notable phenomena is Hrushovski's proof of non-existence of the model companion of the theory of fields with two commuting automorphisms (\cite[Theorem 3.2]{kikyo2}).
Such fields can be viewed as fields equipped with a group action of $\mathbb{Z}\times\mathbb{Z}$. At the other hand, the model companion of the theory of fields equipped with a group action of $\mathbb{Z}_2\times\mathbb{Z}_2$ (or any other product $G_1\times G_2$ of finite groups $G_1$ and $G_2$) exists (Theorem 2.10 in \cite{nacfa}). The general case (group actions on arbitrary stuctures instead of fields) is even more difficult: the author of \cite{kikyo2} admits at the beginning of Section 3. (in \cite{kikyo2}) that he could not show in general that a model companion of the theory of models with two commuting automorphisms does not exist. However, he additionally assumed existence of $a$, $b$ and $M\models T$ such that $a\ind_M b$ and $\acl(Mab)\neq\dcl(\acl(Ma),\acl(Mb))$, consult the last lines on the page 6. in \cite{kikyo2}. Example \ref{example.empty} shows that without any assumptions, model companions may exist for actions of any group $G$.

Therefore, one could ask about existence of model companions in a wider, than the one considered in \cite{balshe}, context. Any reasonable answer to the last question will put more light on the concepts related to the finite cover property. As a side remark, we note here that even without the assumption about existence of model companion there is some interest in the class of existentially closed models with a group action (\cite{Pilfork}). Therefore we provide a description of such structures (without assumption that the class of existentially closed structures with a group action is an elementary class) in Section \ref{sec:galois}.

During a conversation with Amador Martin-Pizzaro, we found out that models studied by us provide an example for considerations in \cite[p.11]{criteresimple}, where the authors describe a criterion for simplicity of a theory in the shape of ACFA. More precisely, if a theory satisfies the five conditions from \cite{criteresimple}, then it is simple. Still, those five conditions need to be checked, which is done in the next sections of this paper. This means that our theory is related to the present research in the field of simple theories.

The last reason why we provide the following paper is more practical. It is quite common to use results from \cite{ChaPil} during studying a new theory equipped with an automorphism. We give an example about this: in \cite{hils0}, the authors introduce a new theory which describes compact complex manifolds with an automorphism. Its model companion, denoted by CCMA, exists and automatically inherits simplicity and other tameness properties due to the main results of \cite{ChaPil}. We hope that our generalizations of results from \cite{ChaPil}
will be useful in a similar way.

\subsection{Summary of results}
We summarize now the results of this paper.
Let $T$ be an $\mathcal{L}$-theory and let $G$ be a group.
In the crucial moment of this paper (Section \ref{sec:forking}), we assume that the model companion of $T$, say $T^{\mc}$, exists, is stable, allows to eliminate quantifiers and has elimination of imaginaries. 
We introduce the language $\mathcal{L}^G$ given by adding to $\mathcal{L}$ a unary function symbol for each $g\in G$. By $T_G$, we denote the $\mathcal{L}^G$-theory which models are models of $T$ equipped with a group action of $G$ (by $\mathcal{L}$-automorphisms). We assume that the model companion of $T_G$, say $T_G^{\mc}$, exists.

The main difficulty is to understand how models of $T_G^{\mc}$ embed into models of $T^{\mc}$. It was not the case in \cite{ChaPil}, because the authors of \cite{ChaPil} worked with a model complete theory $T$, hence under the assumption $T=T^{\mc}$ (they assumed that $T$ is stable, allows quantifier elimination and elimination of imaginaries), and so every model of $T_G^{\mc}$ is also a model of $T^{\mc}$ (if $G=\mathbb{Z}$, then the $G$-action always extends to saturated over-structures). After a lesson learned during studying the theory $G-\tcf$ (Example \ref{gtcf1}), where models of $T_G^{\mc}$ may be not even separably closed fields (see Theorem 3.6 in \cite{nacfa}) and $T^{\mc}=$ACF, we know that it is more natural to distinguish $T$ from $T^{\mc}$ (or to do not assume that $T$ is stable in the situation of \cite{ChaPil}). However, a lot of new phenomena occur, we still have some nice behaviour (e.g. Proposition \ref{A_Galois}, which says that the relative algebraic closure is normal in the full algebraic closure, which means that a $G$-structure on a model treats equally all ``roots" of an algebraic formula).

The main result of this paper is Theorem \ref{ind_thm_model}, which contains (under one additional assumption) Independence Theorem over a model for a special ternary relation in $T_G^{\mc}$. We use it in Theorem \ref{TG.simple}, which provides reasonable assumptions for simplicity of $T_G^{\mc}$. We have the following sequence of implications, where $(\mathfrak{C},(\sigma_g)_{g\in G})$ is a monster model of $T_G^{\mc}$,
$$\xymatrixcolsep{1.5pc}\xymatrix{G\text{ is finite}\ar@{=>}[r]^-*+++[u]{\text{
\footnotesize Proposition \ref{finite.bounded}}} & \mathfrak{C}\text{ is bounded}
{}\save[]+<0cm,-0.5cm>*\txt<8pc>{%
(Definition \ref{PP.bounded})} \restore
\ar@{=>}[r]^-*+++[u]{\text{\footnotesize Proposition \ref{bounded.split}}} &
\text{alg. closures split (Definition \ref{alg.cls.def})}\ar@{=>}[d]^-*+++[u]{\text{\footnotesize Theorem \ref{ind_thm_model}}} \\ 
& & \text{Ind. Theorem over model for } T_G^{\mc}\ar@{=>}[d]^-*+++[u]{\text{\footnotesize Theorem \ref{TG.simple}}}\\
& & \txt{$T_G^{\mc}$ is simple \\+ description of forking ind. in $T_G^{\mc}$.}}$$
One can wonder whether the assumption about algebraic closures is important for the simplicity of $T_G^{\mc}$. It is indeed, and we discuss this problem in Remark \ref{alg.cls.split.important}. In short: if algebraic closures do not split, then $\mathfrak{C}$ is not bounded, and - after specifying to the theory of fields  - the theory of a PAC field is simple just if the field is bounded.

The second important result is about elimination of imaginaries in $T_G^{\mc}$. It is easy to prove that $T_G^{\mc}$ codes finite tuples (Lemma \ref{finite.codes}). On the other hand, there are examples of theories in the shape of $T_G^{\mc}$ which do not eliminate imaginaries (Remark \ref{no.full.EI}), hence they also do not have  the weak elimination of imaginaries. However, $T_G^{\mc}$ always has the geometric elimination of imaginaries (Theorem \ref{imaginaries}).

We provide also ``semi" quantifier elimination result (Remark \ref{almost.qe}), which is analogous to the similar one for ACFA, but we do not consider it as important.

To obtain the above results, we need a lot of algebraic analysis of the $G$-actions on substructures of a monster model of $T^{\mc}$. 
We transferred the notion of regularity from the ground of fields to the general model theory (Definition \ref{regular.def}, however it turned out that our ``regularity" is in fact Hrushovski's ``stationarity", check Remark \ref{regular.def}).
Regularity is beneficial in the stable context (Corollary \ref{regular.PAPA}). For example, elements in a regular extension over $A$ have stationary types over $A$ (Corollary \ref{reg.stationary}). In Section \ref{sec:galois}, we invoke generalizations of the notions from Galois theory and use them to describe the Galois groups of structures with a $G$-action.
The invariants of a $G$-action play an important role, therefore we provide a description of the substructure of invariants. It turns out that the substructure of invariants, for a finitely generated group $G$, is PAC (Proposition \ref{inv.pac}) and bounded (Lemma \ref{MG.bounded}), hence the theory of invariants, for a finitely generated group $G$, is simple (Theorem \ref{MG.simple}).

Moreover, we use techniques involving Galois groups (e.g. Lemma \ref{N_Galois}) to describe types (Fact \ref{fact_type_descr}), algebraic closure (Corollary \ref{acl_all} and Proposition \ref{A_Galois}), which are used in the main proof of this paper (proof of the Theorem \ref{ind_thm_model}).

We also give an alternative to \cite{PilPol} definition of a PAC substructure in the stable context. We compare our definition of a PAC substructure to the well known definition of a PAC substructure from \cite{PilPol} and to Hrushovski's definition of a PAC substructure in the strongly minimal context (Proposition \ref{PACvsPACPP} and Proposition \ref{PAC.Hru}). We use our definition of a PAC substructure to show that models of the theory $T_G^{\mc}$ are PAC (Proposition \ref{M.is.PAC}).

\subsection{Acknowledgements}
I thank my supervisor, Piotr Kowalski, for his guidance to this point of my mathematical studies, which allowed me to undertake my first serious research.
I thank also Thomas Scanlon, which offered his knowledge and time during my visit in Berkeley. 
Also Martin Hils has significantly contributed to the final draft of this paper remarking that large parts of Section \ref{sec:galois} can be obtained without assumption about existence of model companions of theories with a group action.
Besides this, I am grateful to my wife, Agata, for being proud of me and supporting me during my work.
Last but not least: 
\begin{center}
\textit{Deo gratias!}
\end{center}

\section{Prelude}
\subsection{Preliminaries and conventions}
For any set $X$, a natural number $n>0$ and any function $f:X\to X$, we define $f^{(n)}$ as the composition of $f$ with itself $n$ times. If $A$ and $B$ are two sequences, then $AB$ denotes the concatenation of $A$ and $B$, i.e. $A^{\frown}B$. If $A$ and $B$ are considered only as sets, then $AB$ denotes $A\cup B$. Finally, if $H$ is a group and $A$ is a set, then the orbit of $A$ under an action of $H$ will be denoted by $H\cdot A$ or (if it will not lead to any confusion) by $HA$.

Assume that $\mathcal{L}$ is a language.
We denote the set of all $\mathcal{L}$-formulas by $\mathcal{F}_{\mathcal{L}}$.
By an \emph{$\mathcal{L}$-theory} we mean a non-empty and consistent subset of $\mathcal{F}_{\mathcal{L}}$ which includes all its consequences.
By an \emph{inconsistent $\mathcal{L}$-theory} we mean a non empty and inconsistent subset of $\mathcal{F}_{\mathcal{L}}$ which includes all its consequences, which is equal to the whole $\mathcal{F}_{\mathcal{L}}$.
We assume that theories in this paper are theories with infinite models.

Instead of writing $\mathfrak{C}\models\varphi(c)$ (where $\mathfrak{C}$ is in this context a fixed monster model) or $\models\varphi(c)$, we prefer to write $\varphi^{\mathfrak{C}}(c)$. The same if $M\preceq\mathfrak{C}$ (an elementary substructure): sometimes we use $\varphi^M(c)$ instead of $M\models\varphi(c)$. However, we will use $\sigma_g$ instead of $\sigma_g^M$ for a function symbol $\sigma_g$ which corresponds to an automorphism.

Now, let $N$ and $N'$ be $\mathcal{L}$-structures and let $E$ be a subset of $N$. We use $\langle E\rangle_{\mathcal{L}}$ to denote the $\mathcal{L}$-substructure of $N$ generated by $E$. Moreover, $\acl_{\mathcal{L}}^N(E)$ denotes the algebraic closure of $E$ in $N$ in the sense of the language $\mathcal{L}$ and the $\mathcal{L}$-theory $\theo(N)$ (similarly for $\dcl_{\mathcal{L}}^{N}(E)$ and $\tp_{\mathcal{L}}^N(a/E)$). We say that $N$ is \emph{existentially closed in $N'$} if for every quantifier free $\mathcal{L}$-formula $\varphi(x,y)$ and every finite tuple $b\subseteq N$, $|b|=|y|$, we have
$$N'\models(\exists y)\big(\,\varphi(b,y)\,\big)\quad\Rightarrow\quad
N\models(\exists y)\big(\,\varphi(b,y)\,\big).$$
We write $N\leqslant_1 N'$ if $N$ is existentially closed in $N'$. We say that $M$ is \emph{existentially closed among models of the theory $T$} if for every $M'\models T$ such that $M\subseteq M'$, it follows $M\leqslant_1 M'$. If additionally $M$ is a model of $T$, the it is called shortly an \emph{existentially closed model of the theory $T$}.
If every model of the theory $T$ is existentially closed, the theory $T$ will be called \emph{model complete}. If $T$ is model complete, $M,N\models T$ and $M\subseteq N$, then $M\preceq N$.

\begin{definition}\label{mod_cons}
We call an $\mathcal{L}$-theory $T'$ \emph{model-consistent with} the $\mathcal{L}$-theory $T$ if for each $M\models T$ there exists $M'\models T'$ such that $M\subseteq M'$. 
Equivalently, for each $M\models T$, $T'\cup\Dgat(M)$ is consistent.
\end{definition}

\begin{remark}\label{lud_fact1}
A theory $T'$ is model-consistent with the theory $T$ if and only if $T'_{\forall}\subseteq T_{\forall}$ (i.e. the universal part of the theory $T$ contains the universal part of the theory $T'$).
\end{remark}

We provide here our favourite definition of a model companion (in the spirit of \cite{kaiser1}). The reader may notice that theories of the main interest 
of Section \ref{sec:forking} are model companions (of some theories).

\begin{definition}
We call an $\mathcal{L}$-theory $T'$ a \emph{model companion} of the
$\mathcal{L}$-theory $T$ if the following hold
\begin{enumerate}
\item[i)] $T$ is model-consistent with $T'$,
\item[ii)] $T'$ is model-consistent with $T$,
\item[iii)] $T'$ is model-complete.
\end{enumerate}
\end{definition}

If a model companion of the theory $T$ exists, then it is unique.
The model companion of $T$ will be denoted by $T^{\mc}$.

\begin{remark}
If a model companion $T^{\mc}$ of the theory $T$ exists, then $T_{\forall\exists}\subseteq T^{\mc}$ (i.e. the theory $T^{\mc}$ contains all the $\forall\exists$-formulas belonging to $T$).
\end{remark}

For the rest of this paper we fix a group $(G,\ast)$ (time to time we will assume additional properties of $G$). We are working with an $\mathcal{L}$-theory $T$ and with the language $\mathcal{L}^{G}$ which is the language $\mathcal{L}$ extended by unary function symbols $\bar{\sigma}=(\sigma_g)_{g\in G}$. 

\begin{definition}
\begin{enumerate}
\item We introduce set of $\mathcal{L}^G$-formulas $A_G$, which contains exactly the following axioms:
\begin{enumerate}
\item[i)] $\sigma_g$ is an automorphism of $\mathcal{L}$-structure for every $g\in G$,
\item[ii)] $\sigma_g\circ\sigma_h=\sigma_{g\ast h}$ for every $g,h\in G$.
\end{enumerate}
\item
Let $(M,(\sigma_g)_{g\in G})$ be an $\mathcal{L}^G$-structure. We say that $(\sigma_g)_{g\in G}$ is a \emph{$G$-action} on $M$ if $(M,(\sigma_g)_{g\in G})\models A_G$.

\item
If $T$ is an $\mathcal{L}$-theory, then $T_G$ is an $\mathcal{L}^G$-theory
equal to the set of consequences of $T\cup A_G$, i.e. $T_G=\Cn(T\cup A_G)$.
\end{enumerate}
\end{definition}

We will skip parenthesis in ``$(T_G)^{\mc}$" and abbreviate it to ``$T_G^{\mc}$".
Note that $A_G\subseteq(T_G)_{\forall\exists}\subseteq T_G^{\mc}$, hence each model of $T_G^{\mc}$ (if the model companion exists) is equipped with a $G$-action. Moreover, an $\mathcal{L}^G$-structure $(M,(\sigma_g)_{g\in G})$ may be denoted by ``$(M,\bar{\sigma})$". Note the following, easy but important, fact.

\begin{fact}
For any theory $T$ we have $(T_G)_{\forall}\subseteq (T_{\forall})_G$.
\end{fact}

\begin{remark}\label{tmcgmc.tgmc}
Assume that $T$ is inductive, and the theories $T^{\mc}$ and $T_G^{\mc}$ exist. Then the following are equivalent.
\begin{enumerate}
\item The theory $(T^{\mc})_G^{\mc}$ exists and it follows $(T^{\mc})_G^{\mc}=T_G^{\mc}$.
\item For every $(M,(\sigma_g)_{g\in G})\models T_G^{\mc}$, we have $M\models T^{\mc}$.
\end{enumerate}
\end{remark}

\begin{proof}
Assume the first statement. Since $T^{\mc}$ is inductive, it follows $T^{\mc}\subseteq (T^{\mc})_G\subseteq (T^{\mc})_G^{\mc}$, hence every model $(M,(\sigma_g)_{g\in G})$ of $(T^{\mc})_G^{\mc}=T_G^{\mc}$ satisfies $M\models T^{\mc}$.

To prove implication from the second point to the first point, it is enough to show that $T_G^{\mc}$ is the model companion of $(T^{\mc})_G$. Of course, $T_G^{\mc}$ is model complete. Let $(M,(\sigma_g)_{g\in G})\models T_G^{\mc}$. Because $M\models T^{\mc}$, we have $(M,(\sigma_g)_{g\in G})\models (T^{\mc})_G$. On the other hand, if $T$ is inductive, then $T\subseteq T^{\mc}$. Therefore every model $(M,(\sigma_g)_{g\in G})$ of the theory $(T^{\mc})_G=\Cn(T^{\mc}\cup A_G)$ is a model of the theory $T_G=\Cn(T\cup A_G)$, so it embeds into a model of $T_G^{\mc}$.
\end{proof}

We start with an example which should come to our mind as the first one.

\begin{example}[The empty language and the empty theory case]\label{example.empty}
If we start with $\mathcal{L}=\emptyset$ (formulas consist only of equalities and inequalities of variables) and $T=\emptyset$ then 
$$T^{\mc}=\lbrace(\exists x_1,\ldots,x_n)(\bigwedge\limits_{i\neq j} x_i\neq x_j)\;\;|\;\; n\in\mathbb{N}\rbrace,$$
i.e. the axioms for being an infinite set. Introduce the following set of $\mathcal{L}^G$-formulas
$$T':=T_G\cup\lbrace (\forall x_1,\ldots,x_n)(\exists y)\big(\bigwedge\limits_{i=1}^{n}y\neq x_i\;\wedge\;
\bigwedge\limits_{j=1}^{m}\sigma_{g_j}(y)=y\;\wedge\;\bigwedge\limits_{k=m+1}^{m'}\sigma_{g_k}(y)\neq y\big)|$$ 
$$|\;\;H<G \text{ is finitely generated},\;n,m,m'\in\mathbb{N},$$
$$g_1,\ldots,g_m\in H,\;g_{m+1},\ldots,g_{m'}\in G\setminus H\; \rbrace.$$
Note that if $M\models T_G$, then
$$M\amalg\big(\coprod_{\mathclap{\substack{H<G\\ \text{finitely generated} }}}G/H\big)^{\amalg\omega},$$
where $G/H$ denotes the set of left cosets considered with the standard left action of $G$, is a model of $T'$ and an $\mathcal{L}^G$-extension of $M$.

Now, assume that $M\models T'$, there is $N\models T_G$ which extends $M$ and for some finite $\bar{m}\subseteq M$ and a system of equalities and inequalities in $\mathcal{L}^G$, say $\varphi(\bar{x},y)$, we have $N\models(\exists y)(\varphi(\bar{m},y))$. To show that $M\models(\exists y)(\varphi(\bar{m},y))$ we need to analyse just a few types of equalities/inequalities, which can occur in $\varphi(\bar{x},y)$, mainly
\begin{IEEEeqnarray*}{rCl}
\sigma_g(y) &=& y, \\
\sigma_g(y) &\neq & y, \\
\sigma_g(y) &=& x_i, \\
\sigma_g(y) &\neq & x_i.
\end{IEEEeqnarray*}
If there is a formula of the form $\sigma_g(y)=x_i$, then obviously element $\sigma_g^{-1}(m_i)\in M$ is the solution for $\varphi(\bar{m},y)$. Therefore we can assume that in $\varphi(\bar{x},y)$ there are no formulas of the form $\sigma_g(y)=x_i$. Moreover, we can even demand that our solution will be different from each $m_i$, i.e. we add formulas $y\neq x_i$, $i\leqslant n$, to the system of equalities and inequalities given by $\varphi(\bar{x},y)$.
Define $H:=\langle g\;\;|\;\; ``\sigma_g(y)=y"\text{ occurs in }\varphi(\bar{x},y)\rangle$, of course it is finitely generated. If $\sigma_g(y)\neq y$ occurs in $\varphi(\bar{x},y)$ then $g\not\in H$ (otherwise it would result in a contradiction for the realization of $\varphi(\bar{m},y)$ in $N$). Hence formula $(\forall\bar{x})(\exists y)(\varphi(\bar{x},y))$ can be viewed as an element of $T'$, so $M\models
(\exists y)(\varphi(\bar{m},y))$.

Now we will move to the superstability of $T^{\mc}$ and $T_G^{\mc}$ for the ``empty" theory $T$. 
Superstability of $T^{\mc}$ is well-known (in fact $T^{\mc}$ is strongly minimal), but we provide an argument which we also use to show superstability of $T_G^{\mc}$.
We note that a complete $\mathcal{L}$-type in $T^{\mc}$ over some set of parameters $A$, say $p(x)$, can be chosen only in $|A|+1$ many different ways (i.e. $x$ can be ``equal" to some $a\in A$ or the type $p(x)$ states that ``$x\not\in A$"). Similarly for a type in $n$ variables, $p(x_1,\ldots,x_n)$, we can limit the amount of possibilities: not more than $|A|$ if $|A|\geqslant\omega$.
Therefore $T^{\mc}$ is superstable. We proceed to the case with a $G$-action. Let $p(x)$ be a complete $\mathcal{L}^G$-type in $T_G^{\mc}$ over some set of parameters $A$. Type $p(x)$ can state ``$x\in G\cdot A$" (at most $|G|\cdot|A|$ different ways to do this) or ``$x\not\in G\cdot A$" and
$\lbrace (\sigma_g(x)=x)^{\eta(g)}\;|\;g\in G\rbrace$, where $\eta\in 2^G$,
``$(\sigma_g(x)=x)^1$" corresponds to ``$\sigma_g(x)=x$", and ``$(\sigma_g(x)=x)^0$" corresponds to ``$\sigma_g(x)\neq x$" (we use here quantifier elimination for $T_G^{\mc}$ which can be proven straightforward using the previously provided axioms of $T_G^{\mc}$). 
If $\lambda\geqslant 2^{|G|}+\omega$ and $|A|\leqslant\lambda$, then we have at most $\lambda$ many types over $A$ in $T_G^{\mc}$. It means that $T_G^{\mc}$ is also superstable. 
%In Example \ref{example.empty2}, we provide a description of the forking independence in $T_G^{\mc}$ and count the SU-rank of $T_G^{\mc}$.
\end{example}

The above example shows that for one theory, the ``empty" theory $T$, $T_G^{\mc}$ exists for any group $G$. Usually the situation is much more complicated, i.e. for most of the theories $T$, proving the existence of $T_G^{\mc}$ is a hard task highly depending on the choice of $G$.

\begin{question}\label{q220}
Which assumptions about $T$ and $G$ assert the existence of $T_G^{\mc}$?
\end{question} 
\noindent
Baldwin, Kikyo and Shelah proved a few negative results about existence of $T_{\mathbb{Z}}^{\mc}$ in \cite{kikyo1}, \cite{kishe} and in \cite{balshe}.
The first counterexamples for existence of $T_{\mathbb{Z}}^{\mc}$, which motivated Kikyo to focus on this topic, were: the theory of random graph, DLO$_0$ and the theory of atomless Boolean algebras (see the introduction to \cite{kikyo1}). In fact, Baldwin, Kikyo and Shelah investigated existence of model companions only for theories with one automorphism (instead of theories with $G$-actions), and we hope that our generalization will lead to a deeper understanding why such a model companion exists or not. First of all, if $G$ is finite, then $T_G^{\mc}$ exists for the ``counterexample theories":
\begin{itemize}
\item the theory of the random graph (Example \ref{rand.graph}),
\item DLO$_0$ (Example \ref{orders}),
\item the theory of atomless Boolean algebras (Remark \ref{Boole2}).
\end{itemize}

\begin{example}[The random graph]\label{rand.graph}
Assume that $|G|=e$, $(G,\ast)=(\lbrace 1,\ldots,e\rbrace,\ast)$.
The following argument depends only on one property of the theory of graphs. Mainly, 
the atomic diagram of a finite tuple is finite (hence the example and axioms provided in it, can be generalized on other theories owing this property). In the lines below, we define a \emph{consistent configuration} $Q$, which in fact is a conjunction of all formulas belonging to the atomic diagram of some finite tuple of vertices.

Let $\mathcal{L}$ consist of a binary relation $R$, and let $T$ state that $R$ is irreflexive and symmetric ($R$ indicates the existence of an edge between two vertices). The theory $T$ has a model companion $T^{\mc}$ given by the axioms of $T$ together with an additional axiom scheme:
\begin{center}
for every finite sets $X$, $Y$ satisfying $X\cap Y=\emptyset$ there exists a vertex $v$  such that $v\;R\;x$ for all $x\in X$ and $v\centernot R y$ for all $y\in Y$.
\end{center}
We can easily check that $T_G^{\mc}$ exists, we give axioms for this theory and sketch the proof of being a model companion of $T_G$. We start with $T_G$, which is given by axioms of $T$, axioms stating that each $\sigma_i$, where $i\leqslant e$,  is an automorphism and axioms of the form $\sigma_{i}\circ\sigma_{j}(x)=\sigma_{i\ast j}(x)$. 

Let $x_1,\ldots,x_n, y_1,\ldots,y_{n'}$ be variables and let 
$$Z_0:=\bigcup\limits_{i\leqslant e}\lbrace \sigma_i(x_1),\ldots,\sigma_i(x_n)\rbrace,\quad Z:=\bigcup\limits_{i\leqslant e}\lbrace \sigma_i(x_1),\ldots,\sigma_i(x_n),\sigma_i(y_1),\ldots,\sigma_i(y_{n'})\rbrace.$$
Note that $Z_0$ is a set of $\mathcal{L}^G$-terms in variables $x_i$ and $Z$ is a set of $\mathcal{L}^G$-terms in variables $x_i$, $y_j$,  where $i\leqslant n$ and $j\leqslant n'$. 
We code \emph{the configuration} of $Z$ by the following function
$$Q:Z\times Z\ni(z_1,z_2)\mapsto Q_{(z_1,z_2)}\in\lbrace =,R,\centernot R\rbrace,$$
where if $Q_{(z_1,z_2)}$ is equal to $\centernot R$, the string $z_1Q_{(z_1,z_2)}z_2$ should be understood as $z_1\centernot Rz_2\wedge z_1\centernot=z_2$ (the rest is standard).
It is rather technical, but it can be checked whether $Q$ describes a configuration, which is \emph{consistent} (and there is a finite condition to check this). Where ``consistent" means that there exists $V\models T_G$ and there exist $w_1,\ldots,w_n,v_1,\ldots,v_{n'}\in V$ such that 
$$V\models\Big(\bigwedge\limits_{z_1,z_2\in Z}z_1Q_{(z_1,z_2)}z_2\Big)[w_1,\ldots,w_n,v_1,\ldots,v_{n'}],$$
where the formula $\bigwedge\limits_{z_1,z_2\in Z}z_1Q_{(z_1,z_2)}z_2$ in variables $x_1,\ldots, x_n,y_1,\ldots,y_{n'}$ is evaluated on $w_1,\ldots,w_n$, $v_1,\ldots,v_{n'}$.
In such a case we call $Q$ a \emph{consistent configuration} of variables $x_1,\ldots,x_n, y_1,\ldots,y_{n'}$.
A model $V$ of the theory $T'$ satisfies $V\models T_G$ and 
for each $n,n'\in\mathbb{N}$ and any consistent configuration  $Q$ of variables $x_1,\ldots,x_n, y_1,\ldots,y_{n'}$,
it satisfies
the following (in the above notation) axiom:
$$
(\forall x_1,\ldots x_n)\Big(\bigwedge\limits_{z_1,z_2\in Z_0}z_1Q_{(z_1,z_2)}z_2\rightarrow
(\exists y_1,\ldots,y_{n'})\big(
\bigwedge\limits_{z_1,z_2\in Z}z_1Q_{(z_1,z_2)}z_2\big)\Big).$$
Assume that $(V,R,\overline{\sigma})\models T_G$. After adding suitably many vertices to $V$, we obtain an $\mathcal{L}^G$-extension, which is a model of $T'$. 

Similarly to the theory $T^{\mc}$, $T'$ eliminates quantifiers. The proof is rather standard and we only sketch the main steps. 
If $\varphi(\bar{x},y)$ is a conjunction of atomic and negations of atomic formulas in $\mathcal{L}^G$, we want to eliminate the quantifier in $(\exists y)\,\big(\varphi(\bar{x},y)\big)$.
There are two cases. Firstly, there is no consistent configuration $Q$ such that $\varphi(\bar{x},y)$ is contained (as a collection of atomic formulas and negations of atomic formulas, 
maybe after some permutation) in $\bigwedge\limits_{z_1,z_2\in Z}z_1Q_{(z_1,z_2)}z_2$ (notation as above for $n'=1$). In this case, we have
$$T'\vdash (\exists y)\,\big(\varphi(\bar{x},y)\big)\leftrightarrow \bar{x}\neq\bar{x}.$$
Secondly, there exists a consistent configuration $Q$ such that $\varphi(\bar{x},y)$ is contained in $\bigwedge\limits_{z_1,z_2\in Z}z_1Q_{(z_1,z_2)}z_2$. We can show that
$$T'\vdash (\exists y)\,\big(\varphi(\bar{x},y)\big)\leftrightarrow \bigvee\limits_{i\leqslant r} (\exists y)\big(\bigwedge\limits_{z_1,z_2\in Z}z_1Q^i_{(z_1,z_2)}z_2 \big),$$
where $Q^1,\ldots,Q^r$ are all consistent configurations containing $\varphi(\bar{x},y)$ (there are only finitely many such configurations, because there only finitely many possibilities to define the atomic diagram of a tuple of a fixed length). By our axioms, it follows
$$T'\vdash (\exists y)\,\big(\varphi(\bar{x},y)\big)\leftrightarrow \bigvee\limits_{i\leqslant r}\; \bigwedge\limits_{z_1,z_2\in Z_0}z_1Q^i_{(z_1,z_2)}z_2.$$

The theory $T'$ is model complete, and so it is a model companion $T_G^{\mc}$ for $T_G$. Kikyo worked under the assumption $T=T^{\mc}$ (i.e. he considered random graphs with added $G$-action, we consider graphs with added $G$-action).
Therefore, if we want to see that $T'$ is a ``counterexample", 
we need to show that reducts of models of $T'$ are models of $T^{\mc}$.
Assume that $(V,(\sigma_g)_{g\in G})\models T'$, we need to prove that $V$ satisfies the random graph axioms.

Let $X,Y\subseteq V$ be finite, our goal is to find $v\in V$ such that $v$ has an edge with each element of $X$ and has no edge with any element of $Y$. Consider the graph $W:=GX\cup GY \cupdot G$ with edges on $GX\cup GY$ copied from $V$, and the new edges $gR(gw_X)$ defined for all $g\in G$ and all $w_X\in X$. There is a natural $G$-action on $W$ and we see that $1\in G\subseteq W$ has an edge with each element of $X$ and has no edge with any element of $Y$. We take the consistent configuration $Q$ given by the atomic diagram of $(W,(\sigma_g)_{g\in G})$. The axiom of $T'$ corresponding to $Q$ assures existence of the desired vertex $v$.

Summarizing, we defined an $\mathcal{L}^G$-theory $T'$ which extends the theory of graphs with a $G$-action, $T_G$, and has quantifier elimination. We stated that models of $T_G$ extend to models of $T'$, hence $T'=T_G^{\mc}$. Moreover, we noted that $T^{\mc}\subseteq T_G^{\mc}$, hence Remark \ref{tmcgmc.tgmc} implies that $(T^{\mc})_G^{\mc}=T_G^{\mc}$ (i.e. random graphs with action of a finite group have a model companion).

Martin Hils pointed out to us that $T_G^{\mc}$ in the case of the theory of random graph can be described as the theory of the Fra\"{i}ss\'{e} limit of the class of finite graphs equipped with a $G$-action. He also noted that $T_G^{\mc}$ as a Fra\"{i}ss\'{e} limit is $\omega$-categorical.
\end{example}

\begin{definition}\label{def.gen.inv}
Let $M$ be an $\mathcal{L}^G$-structure. By $M^G$ we denote the \emph{$G$-invariants} (\emph{invariants}), given by $$M^G=\lbrace m\in M\;|\;(\forall g\in G)(\sigma_g(m)=m)\rbrace.$$
\end{definition}

\begin{example}[Linear orders]\label{orders}
Assume that we have an $\mathcal{L}^G$-structure $M$.
If there exists an $M^G$-definable linear ordering and $|G|=e\in\mathbb{N}$, then $G$ acts trivially on $M$, i.e.: $M^G=M$. It is because for any preserving order bijection $f:M\to M$ such that $f^{(n)}=\id$, for some $n>0$, element $f(m)\in M$ can not be strictly greater or smaller than element $m\in M$.

Therefore, if $T$ is the theory of linear orders, then $T_G$ is (informally) equal to $T$. Hence $T_G^{\mc}$ is equal to DLO$_0$, which is a model companion of $T$ (in this case $T_G^{\mc}$ is informally equal to $T^{\mc}$). Recall that the author of \cite{kikyo1} stated that $T=$DLO$_0$ is a counterexample for existence of $T_{\mathbb{Z}}^{\mc}$ (we have $T^{\mc}=$DLO$_0$), so our approach differs from the original one. But in this particular case, we can also start with $T=$DLO$_0$ and still $T_G^{\mc}$ exists and is again equal to DLO$_0$.

Similarly for the theory of ordered fields and its model companion as well as the theory of real-closed fields (and the case of a finite group $G$).
\end{example}

 \begin{example}[Rings of exponent $2$]\label{ring_exp}
Assume that $\mathcal{L}$ is the language of rings (i.e. $\lbrace +,-,\cdot,0,1\rbrace$). 
By ``the theory of rings" we mean an $\mathcal{L}$-theory which models are commutative rings with unit.
We introduce here a notation which is used in \cite{nacfa} and is presented there in a more systematic way.

If an $\mathcal{L}^G$-structure $(R,(\sigma_g)_{g\in G})$ is a model of $T_G$ for $T=$ ``the theory of rings", we call it a \emph{$G$-transformal ring}. If moreover $I\trianglelefteqslant R$, we say that $I$ is  a \emph{$G$-invariant ideal} if for each $g\in G$ it is
$$\sigma_g(I)\subseteq I.$$

If $I$ is a $G$-invariant ideal,
then there is a natural $G$-transformal ring structure on $R/I$.

Again, assume that $|G|=e\in\mathbb{N}$, $(G,\ast)=(\lbrace 1,\ldots,e\rbrace,\ast)$.
Let $(R,(\sigma_g)_{g\in G})$ be a $G$-transformal ring. 
For any $\bar{r}=(r_1,\ldots,r_n)\in R^n$ we use the following convention
$$\sigma_k(\bar{r})=\big(\sigma_k(r_1),\ldots,\sigma_k(r_n)\big),$$
where $k\leqslant e$. If $\bar{r}_1,\ldots,\bar{r}_e\in R^n$, then
$$\bar{\sigma}\big(\bar{r}_1,\ldots,\bar{r}_e\big):=
\big(\sigma_1(\bar{r}_1),\ldots,\sigma_e(\bar{r}_e)\big),$$
$$\bar{\sigma}(\bar{r}_1):=\big(\sigma_1(\bar{r}_1),\ldots,\sigma_e(\bar{r}_1)\big).$$

We set a $G$-transformal ring structure on the ring of polynomials over $R$.
Fix $n>0$ and let $X_i$, where $i\leqslant e$, denote the $n$-tuple of variables, $(X_{i,1},\ldots,X_{i,n})$. 
A $G$-transformal
ring structure on the ring $R[X_1,\ldots,X_e]$ is given by
$$\sigma_{k}(f(X_1,\ldots,X_e))=f^{\sigma_{k}}(X_{k\ast 1},\ldots,X_{k\ast e}),$$
where 
$$\Big(\sum\limits_{\mathbf{i}} r_{\mathbf{i}} \bar{X}^{\mathbf{i}} \Big)^{\sigma_{k}}=
\sum\limits_{\mathbf{i}} \sigma_{k}(r_{\mathbf{i}}) \bar{X}^{\mathbf{i}} .$$

We focus on the theory of rings of exponent $2$, i.e. the theory of rings satisfying additional axiom
$$x^2=x.$$

Let $T$ be the theory of rings of exponent $2$, and let $(R,(\sigma_g)_{g\in G})\models T_G$. Because $(R,(\sigma_g)_{g\in G})$ is a $G$-transformal ring, we have a $G$-transfomal ring structure on the ring $R[X_1,\ldots,X_e]$, as above. Moreover, 
$$I_2=(X_{i,j}^2-X_{i,j}\;|\;i\le e,\,j\le n)$$
 is a $G$-invariant ideal and so $R[t_1,\ldots,t_e]:=R[X_1,\ldots,X_e]/I_2$ is a $G$-transformal ring of exponent $2$, where $t_i=(t_{i,1},\ldots,t_{i,n})$ is the image of $(X_{i,1},\ldots,X_{i,n})$ under the quotient map.
The most important property of $R[t_1,\ldots,t_e]$ is that 
it is a free $R$-module of finite rank, so each element of $R[t_1,\ldots,t_e]$ is represented by a finite sequence of elements of $R$, of length bounded by $2^{ne}$.

The theory $T_G^{\mc}$ exists and is given by the following (compare with the axioms in DExample \ref{gtcf1}). An $\mathcal{L}^G$-structure
$(R,(\sigma_g)_{g\in G})$ is a model of $T'$ if $(R,(\sigma_g)_{g\in G})\models T_G$ and 
if for every $n\in\mathbb{N}_{>0}$, we have
\begin{itemize}
\item[($\diamondsuit$)] every finitely generated $I,J\trianglelefteqslant R[t_1,\ldots,t_e]$
(as in the above notation) such that $I\subsetneq J$ and $I$ is a $G$-invariant ideal, there is $r\in R^n$ satisfying $\bar{\sigma}(r)\in V_R(I)\setminus V_R(J)$.
\end{itemize}
The expression ``$\bar{\sigma}(r)\in V_R(I)$" should be understood as 
$$I\subseteq\ker\big(\ev_{\bar{\sigma}(r)}:R[t_1,\ldots,t_e]\to R\big),$$
where $\ev_{\bar{\sigma}(r)}:R[t_1,\ldots,t_e]\to R$ it the unique map such that composed with the quotient map it is equal to the common evaluation map
$\ev_{\bar{\sigma}(r)}:R[X_1,\ldots,X_e]\to R$.

It is a standard argument to show that $T'$ is a model companion $T_G^{\mc}$ of $T_G$. The reader may consult proofs of Lemma 2.8 and Lemma 2.9 in \cite{nacfa}. We note here 
one fact which is used in analogons of these proofs.
For any polynomial $F\in R[X_1,\ldots,X_e]$ there exists a ``truncated" polynomial $\tilde{F}\in R[t_1\ldots,t_e]$ such that for any $r\in R^n$ we have
$$F(r)=0\qquad\iff\qquad \tilde{F}(r)=0.$$
\end{example}
 
 \begin{remark}[Atomless Boolean algebras]\label{Boole2}
The theory considered in Example \ref{ring_exp}, is the theory of Boolean rings which are structures canonically dual to Boolean algebras, see \cite[\S 6.3]{poizat001}
for the quantifier-free translation (i.e. extensions by definitions without quantifiers in both ways).
Therefore, we have proved existence of $T_G^{\mc}$ for $T=$ ``theory of Boolean algebras" and a finite $G$. Recall that 
a model companion of the theory of atomless Boolean algebras with an automorphism does not exist. Theory of atomless Boolean algebras is the model companion of $T$, denoted in this remark by $T^{\mc}$. Therefore the theory $(T^{\mc})_{\mathbb{Z}}^{\mc}$ does not exist, but the theory $T_G^{\mc}$ does.

On the other hand, $\mathcal{L}$-reducts of models of $T_G^{\mc}$ are atomless Boolean algebras. To see this assume that $R\models T_G^{\mc}$ and $r\in R$. Being an atom means that there is no $s\in R$ such that $r\cdot s=r$ and $r\neq s$. Consider $R[t_1,\ldots,t_e]$ defined as in Example \ref{ring_exp} for $|t_1|=\ldots=|t_e|=1$, and $R':=R[t_1,\ldots,t_e]/I$, where $I$ is the ideal generated by the elements $\sigma_i(r)t_i-\sigma_i(r)$, for $i\leqslant e$. We note that $R\subseteq R'$ is an $\mathcal{L}^G$-extension and $R'\models(\exists y)(r\cdot y=y\;\wedge\;r\neq y)$. Model completeness of $T_G^{\mc}$ implies that we also have $R\models(\exists y)(r\cdot y=y\;\wedge\;r\neq y)$, hence $r$ can not be an atom.

Finally, by Remark \ref{tmcgmc.tgmc}, the theory $T_G^{\mc}$ is equal to $(T^{\mc})_G^{\mc}$. Therefore the theory $(T^{\mc})_{\mathbb{Z}}^{\mc}$ does not exist, but the theory $(T^{\mc})_G^{\mc}$ does.
\end{remark}

 \begin{example}[Fields]\label{gtcf1}
Let $T=$ ``theory of fields" in the language of rings as in Example \ref{ring_exp} and let $|G|=e\in\mathbb{N}$. The existence and properties of $T_G^{\mc}$ were one of the the main motivations for this paper and a model for the idea of model-theoretic dynamics. We just recall here the definition of the theory $G-\tcf=T_G^{\mc}$ from \cite{nacfa}, for the proofs we refer the reader to \cite{nacfa}. We use the notation from Example \ref{ring_exp}, and a \emph{$G$-transformal field} means a $G$-transformal ring, which is a field.

 A $G$-transformal field $(K,\overline{\sigma})$ is a model of $G-\tcf$, if
 for every $n\in\mathbb{N}_{>0}$ it satisfies the following axiom scheme:
\begin{itemize}
 \item[($\clubsuit$)]  
 for any $I,J\trianglelefteqslant K[X_1,\ldots,X_e]$ (as in the above notation) such that $I\subsetneq J$
 and $I$ is a $G$-invariant prime ideal,
 there is $a\in K^n$ satisfying $\overline{\sigma}(a)\in V_K(I)\setminus V_K(J)$.
\end{itemize}
Note that ``iterations" like $T_{\mathbb{Z}/2\mathbb{Z}}^{\mc}$ for $T=\mathbb{Z}/2\mathbb{Z}-\tcf$ are equal to $\mathbb{Z}/2\mathbb{Z}\times\mathbb{Z}/2\mathbb{Z}-\tcf$.
The theories like $\mathbb{Z}/2\mathbb{Z}\times\mathbb{Z}/2\mathbb{Z}-\tcf$ exist, so also the ``iteration" theories exist.
The same remains true for the general case (for products $G_1\times G_2$ of finite groups $G_1$ and $G_2$).
\end{example}

The theory $G-\tcf$, where $G$ is finite, inherits some nice properties from the theory ACF, which is the model companion of $T=$"theory of fields". Similarly for ACFA, which 
should be understood as the theory denoted in our convention by $\mathbb{Z}-\tcf$ or by $\acf_{\mathbb{Z}}^{\mc}$. ACF is the theory for the most model theorists, a benchmark if we think that
\begin{center}
``model theory" $=$ ``theory of fields" $-$ ``fields".
\end{center}
It seems to be true that $T_G^{\mc}$ inherits similar properties
from $T^{\mc}$ (for an arbitrary theory $T$) as $G-\tcf$ and ACFA do from ACF (of course if both $T^{\mc}$ and $T_G^{\mc}$ exist).
One of the main questions investigated in this paper is the following one.

\begin{question}
What properties 
the theory $T_G^{\mc}$ inherits from the theory $T^{\mc}$?
\end{question}

\section{Galois theory for structures with group action}\label{sec:galois}
Before we start to investigate the forking independence and other pure model theoretic notions and properties, we provide more facts about invariants, \emph{Galois groups} and \emph{regularity}. One could say: we are going now to do some algebra without algebra.

\subsection{PAC revisited}
We fix (``once for all") an $\mathcal{L}$-structure $\mathfrak{D}$ which is $\kappa_{\mathfrak{D}}$-saturated and $\kappa_{\mathfrak{D}}$-strongly homogeneous and set $T':=\theo_{\mathcal{L}}(\mathfrak{D})$. In other words: $\mathfrak{D}$ is a monster model for the complete theory $T'$.

We start with Definition \ref{regular.def}, which is a very general one and therefore we wished to provide it before we start to assume additional 
properties of $T'$. On the other hand we wanted to motivate the introduction of a new definition (the definition of a \emph{PAC substructure}) by comparing it to the previous ones. It could not be done without references to the results obtained in the next parts of the thesis. Instead of moving the comparison of the new definition of a PAC substructure to an appendix, we decided to provide it 
after the definition appears, but with references (in proofs of Proposition \ref{PACvsPACPP} and Proposition \ref{PAC.Hru}) to Lemma \ref{PACclaim}, which assumes the stability of $T'$ (similarly as Proposition \ref{PACvsPACPP} and Proposition \ref{PAC.Hru}).

The first point of the below definition is extracted from the proof of \cite[Theorem 3.7]{ChaPil}, and it is motivated by some phenomena in the algebra of fields. The second point is another possibility for the well known definition of a \emph{PAC} substructure, which was introduced in Hrushovski's manuscript (\cite{manuscript}) and then generalized by Pillay and Polkowska in \cite{PilPol}. 
 
Our motivation for the definition of a PAC substructure was a phenomena from the field theory, i.e. the description of PAC fields made in \cite[Proposition 11.3.5]{FrJa}. It turns out that our definition of a PAC substructure is a more faithful generalization of the Hrushovski's definition than the definition from \cite{PilPol}. We discuss it in this subsection.

\begin{definition}\label{regular.def}
\begin{enumerate}
\item Let $E\subseteq A$ be small subsets of $\mathfrak{D}$. We say that $E\subseteq A$ is \emph{$\mathcal{L}$-regular} (or just \emph{regular}) if
$$\dcl_{\mathcal{L}}^{\mathfrak{D}}(A)\cap\acl_{\mathcal{L}}^{\mathfrak{D}}(E)=\dcl_{\mathcal{L}}^{\mathfrak{D}}(E).$$

\item Let $N$ be a small $\mathcal{L}$-substructure of $\mathfrak{D}$. We say that $N$ is \emph{pseudo-algebraically closed} (\emph{PAC}) if for every small $\mathcal{L}$-substructure $N'$ of $\mathfrak{D}$, which is $\mathcal{L}$-regular extension of $N$, it follows $N\leqslant_1 N'$ (i.e. $N$ is existentially closed in $N'$).
\end{enumerate}
\end{definition}
\noindent

\begin{remark}\label{regular.remark}
\begin{enumerate}
\item
After posting this paper on Arxive, Silvain Rideau informed us that 
our definition of regularity coincides with the definition of stationarity given in \cite[Definition 5.17]{silvain002} and considered in \cite{hrushovski_onfinima}, and we had not been aware of that. However, our approach to this property is more algebraic and we provide more facts about it arising from the algebra of fields (e.g. Corollary \ref{regular.PAPA}, Lemma \ref{lang413} and Remark \ref{rem.413}). Lemma \ref{PACclaim} puts more light on the relation between term ``stationarity" and term ``regularity" in this context.

\item We will assume that $T'$ eliminates imaginaries and therefore we formulated our definition of regularity as above. However, it can be generalized by 
passing with the previous condition to the imaginary sorts:
$$\dcl_{\mathcal{L}^{\eq}}^{\mathfrak{D}^{\eq}}(A)\cap\acl_{\mathcal{L}^{\eq}}^{\mathfrak{D}^{\eq}}(E)=\dcl_{\mathcal{L}^{\eq}}^{\mathfrak{D}^{\eq}}(E).$$

\item
Let us recall that a field extension $k\subseteq K$ is regular (in the classical sense) if and only if $K$ is separable over $k$ and $k$ is relatively algebraically closed in $K$. 
In the implementation of our definition of regularity to the field case,
we avoid the separability condition and demand that the \emph{perfect closure} of $k$ is relatively algebraically closed in the \emph{perfect closure} of $K$, i.e. we say that $k\subseteq K$ is $\mathcal{L}$-regular if
$$K^{\perf}\cap k^{\alg} = k^{\perf}.$$
Most of the time we will work in the situation which corresponds 
to the case of a perfect basis field $k$ (i.e. $k=k^{\perf}$), and then the last condition reduces to $K\cap k^{\alg}=k$ (by \cite[Lemma 4.10]{lang2002algebra}).

\item
Note that the regularity condition is invariant under the action of automorphisms.

\item
Of course, if $E$ is algebraically closed, then $E\subseteq A$ is regular for any small $A$.

\item
If $E\subseteq A$ is regular, and $E\subseteq A'\subseteq A$, then $E\subseteq A'$ is regular.

\item 
Assume that $E\subseteq A$ and $A\subseteq B$ are regular. It follows that $E\subseteq B$ is regular.

\item
Assume that $T'$ has quantifier elimination and let $P$ be a small $\mathcal{L}$-substructure of $\mathfrak{D}$. There exists a small $\mathcal{L}$-substructure $P^{\ast}$ of $\mathfrak{D}$ such that $P\subseteq P^{\ast}$ and $P^{\ast}$ is PAC. To see this it is enough to consider a small existentially closed model of $T'_{\forall}$ which extends $P$ and embeds it into $\mathfrak{D}$ over $P$. Quantifier elimination is needed here to move, if necessary, image of this embedding such that it will contain $P$. However, if we repeat the proof of existence of existentially closed extensions as we do in the proof of Proposition \ref{prop.PAC.extension}, we may discard the assumption about quantifier elimination and still prove the existence of PAC extensions.

%\item Assume that $P\subseteq\mathfrak{D}$ is PAC and let $P_1\subseteq\acl_{\mathcal{L}}^{\mathfrak{D}}(P)$. Then also $P_1$ is PAC.
\end{enumerate}
\end{remark}

We analyze regular extensions in Subsection \ref{subs:stable.case}, but before moving to other concepts, we provide the following lemma, which shows 
that the notion of a regular extension generalizes the notion of being an extension of an existentially closed structure. Moreover, being a PAC substructure $P$ means that $P$ is existentially closed in every regular extension, and (if $T'$ has quantifier elimination), by Lemma \ref{lang410}, 
being a PAC substructure $P$ means that $P$ is existentially closed only in regular extensions
(in other words: ``PAC $=$ existentially closed exactly in regular extensions"). Note that, if we assume stability of $T'$ (and EI and QE), by Corollary \ref{PAC.substructures}, every existentially closed substructure of a PAC structure is PAC.

\begin{lemma}\label{lang410}
If the $\mathcal{L}$-theory $T'$ admits quantifier elimination and for some small $\mathcal{L}$-substructures $P\subseteq N$ of $\mathfrak{D}$ it is $P\leqslant_1 N$, then $P\subseteq N$ is regular.
\end{lemma}

\begin{proof}
Assume that $a\in\dcl_{\mathcal{L}}^{\mathfrak{D}}(N)\cap\acl_{\mathcal{L}}^{\mathfrak{D}}(P)$.
Consider quantifier free $\mathcal{L}$-formulas $\varphi$ and $\psi$ such that for some finite tuples $n\subseteq N$ and $p\subseteq P$ it follows
$$\varphi(n,\mathfrak{D})=\lbrace a\rbrace,\qquad |\psi(p,\mathfrak{D})|<\infty,\qquad
\psi(p,\mathfrak{D})=\aut_{\mathcal{L}}(\mathfrak{D}/P)\cdot a.$$
Let $\phi(y)$ be a quantifier free $\mathcal{L}$-formula with parameters from $P$, such that $\phi^{\mathfrak{D}}(y)$ if and only if $\mathfrak{D}$ satisfies
$$(\exists\;x)\Big(\psi(p,x)\;\wedge\;\varphi(y,x)\;\wedge\;(\forall\;x')\big(\varphi(y,x')\rightarrow x=x'\big)\Big).$$
Notice that $\phi^{\mathfrak{D}}(n)$, hence (since $\phi(y)$ is quantifier free) $\phi^{N}(n)$. 
Since $P\leqslant_1 N$, we get $P\models\exists y\;\phi(y)$. Let $p'\subseteq P$ be such that $\phi^P(p')$, it follows $\phi^{\mathfrak{D}}(p')$. We obtain
$$\mathfrak{D}\models
(\exists\;x)\Big(\psi(p,x)\;\wedge\;\varphi(p',x)\;\wedge\;(\forall\;x')\big(\varphi(p',x')\rightarrow x=x'\big)\Big),$$
which means that there exists a solution of $\psi(p,x)$ which is contained in $\dcl_{\mathcal{L}}^{\mathfrak{D}}(P)$, say $a'$. 
Because $\psi^{\mathfrak{D}}(p,a')$, we obtain $a'=f(a)$ for some $f\in\aut_{\mathcal{L}}(\mathfrak{D}/P)$. Therefore $a=f^{-1}(a')=a'\in\dcl_{\mathcal{L}}^{\mathfrak{D}}(P)$.
\end{proof}

\begin{remark}
If the theory $T'$ has quantifier elimination, then for any small $\mathcal{L}$-substructure $P$ there exists a non-trivial regular extension. To see this, consider a non-trivial elementary extension $P'\succeq P$. Structure $P'$ can be embedded into $\mathfrak{D}$, and then, by the quantifier elimination, moved over $P$. The thesis follows from Lemma \ref{lang410}. Compare to Corollary \ref{cor.stationary_types_exist}.
\end{remark}

Now, we will show that for any small $\mathcal{L}$-substructure $P$ of $\mathfrak{D}$ there exists a small $\mathcal{L}$-substructure $P^{\ast}$ of $\mathfrak{D}$ such that $P\subseteq P^{\ast}$ and $P^{\ast}$ is PAC.
The proof is similar to the proof of existence of existentially closed models, but we restrict the procedure to regular extensions. The expected goal would be to achieve a description of minimal PAC extensions of a given substructure, but it is not related to the main subject of this paper and therefore it will not be undertaken in the below text.

\begin{lemma}\label{lemma.reg.trans}
Let $(A_{\alpha})_{\alpha\leqslant\lambda}$ be an ascending sequence of small $\mathcal{L}$-substructures of $\mathfrak{D}$ such that for each $\alpha<\lambda$ it follows that $A_{\alpha}\subseteq A_{\alpha+1}$ is regular. 
For each $\beta<\lambda$, it follows that $A_{\beta}\subseteq A_{\lambda}$
is regular.
\end{lemma}

\begin{proof}
We need to show that
$$\dcl_{\mathcal{L}}^{\mathfrak{D}}(A_{\lambda})\;\cap\;\acl_{\mathcal{L}}^{\mathfrak{D}}(A_{\beta})=\dcl_{\mathcal{L}}^{\mathfrak{D}}(A_{\beta}),$$
which can be done by a transfinite induction which uses Remark \ref{regular.remark}.(7).
\end{proof}
 
\begin{prop}\label{prop.PAC.extension}
Let $P$ be a small $\mathcal{L}$-substructure of $\mathfrak{D}$. There exists
a small $\mathcal{L}$-substructure $P^{\ast}$ of $\mathfrak{D}$ such that $P\subseteq P^{\ast}$ is regular and $P^{\ast}$ is PAC.
\end{prop}

\begin{proof}
We will construct a tower $(P_n)_{n<\omega}$ of small $\mathcal{L}$-substructures of $\mathfrak{D}$. We set $P_0:=P$. Assume that we already have constructed $P_n$ and let $\varphi_{\alpha}(p_{\alpha},x)$, $\alpha<\lambda$, be an enumeration of all quantifier free $\mathcal{L}$-formulas over $P_n$.

We recursively define an auxiliary tower of small $\mathcal{L}$-substructures of $\mathfrak{D}$, $(P_{n,\alpha})_{\alpha<\lambda}$. Of course, $P_{n,0}:=P_n$ and if $\beta<\lambda$ is a limit ordinal, then
$$P_{n,\beta}:=\bigcup\limits_{\alpha<\beta}P_{n,\alpha}.$$
Now, assume that we have $P_{n,\alpha}$. If there is a small $\mathcal{L}$-substructure $P'$ of $\mathfrak{D}$ such that $P_{n,\alpha}\subseteq P'$ is regular and $P'\models\exists\;x\;\varphi_{\alpha}(p_{\alpha},x)$, then take $P_{n,\alpha+1}=P'$ for some arbitrary choice of such $P'$. If there is no
small $\mathcal{L}$-substructure $P'$ of $\mathfrak{D}$ such that $P_{n,\alpha}\subseteq P'$ is regular and $P'\models\exists\;x\;\varphi_{\alpha}(p_{\alpha},x)$, then set $P_{n,\alpha+1}=P_{n,\alpha}$.
Define $P_{n+1}$ as $$\bigcup\limits_{\alpha<\lambda}P_{n,\alpha}.$$

Now, we define $P^{\ast}$,
$$P^{\ast}:=\bigcup\limits_{n<\omega}P_n.$$
It is a small $\mathcal{L}$-substructure of $\mathfrak{D}$ which extends $P$. We need to check whether $P^{\ast}$ is PAC. Let $N$ be a small $\mathcal{L}$-substructure of $\mathfrak{D}$ such that $P^{\ast}\subseteq N$ is regular. Assume that $N\models \exists\;x\;\varphi(p,x)$ for some $p\subseteq P^{\ast}$ and quantifier free $\mathcal{L}$-formula $\varphi(y,x)$. There exist $n<\omega$ and some ordinal $\alpha$ such that $p\subseteq P_{n,\alpha}$ and $\varphi(p,x)$ is equal to $\varphi_{\alpha}(p_{\alpha},x)$. By Lemma \ref{lemma.reg.trans}, $P_{n,\alpha}\subseteq P^{\ast}$ is regular, hence also $P_{n,\alpha}\subseteq N$ is regular. Therefore $P_{n,\alpha+1}\models\exists\;x\;\varphi(p,x)$ and so $P^{\ast}\models \exists\;x\;\varphi(p,x)$.

The extension $P\subseteq P^{\ast}$ is regular by Lemma \ref{lemma.reg.trans}.
\end{proof}

\begin{comment}
\begin{remark}\label{PAC.closure}
If we assume that $T'$ has quantifier elimination and elimination of imaginaries, and is stable, then we can use Lemma \ref{lang413} to prove existence of the smallest PAC-extension of a small $\mathcal{L}$-substructure $P$ of $\mathfrak{D}$. More precisely, assume that $\lbrace P_i\;|\;i\in I\rbrace$ enumerates all small $\mathcal{L}$-substructures of $\mathfrak{D}$, which extend $P$ and are PAC. By Proposition \ref{prop.PAC.extension}, 
$\lbrace P_i\;|\;i\in I\rbrace\neq\emptyset$. Consider
$$\check{P}:=\bigcap\limits_{i\in I}P_i.$$
We need to show that $\check{P}$ is PAC. To do this we will need the following claim.
\
\\
\\
\textbf{Claim:}
If $P$ and $M$ are small $\mathcal{L}$-substructures of $\mathfrak{D}$ such that $P$ is PAC, then $P\cap M$ is PAC.
\\
Proof of the claim: Assume $P\cap M\subseteq N$ is regular and $N\models\exists\;x\;\varphi(p,x)$ for some $p\subseteq P\cap M$ and quantifier free $\mathcal{L}$-formula. Let $N'\equiv_{P\cap M} N$ be such that $P\ind^{\mathfrak{D}}_{P\cap M} N'$. It follows $\dcl_{\mathcal{L}}^{\mathfrak{D}}(P\cap M,N')\models\exists\;x\;\varphi(p,x)$.
By Lemma \ref{lang413}, $P\subseteq\dcl_{\mathcal{L}}^{\mathfrak{D}}(P\cap M,N')$ is regular...
\end{remark}

\vspace{10mm}
\hrule
\vspace{5mm}
\end{comment}

Before we compare our definition of a PAC substructure to the existing ones, we remind below those definitions of a PAC substructure:
\begin{itemize}
\item the one given by Hrushovski will be distinguished  by ``Hru"-subscript: PAC$_{\text{Hru}}$,
\item the one given by Pillay and Polkowska will be distinguished by ``PP"-subscript: PAC$_{\text{PP}}$.
\end{itemize}
 
\begin{definition}[Definition 1.2 in \cite{manuscript}]\label{hru.PAC.def}
Let $T'$ be strongly minimal and has quantifier elimination and elimination of imaginaries. Let $M\models T'$ satisfy the \emph{definable multiplicity property} (consult ``Framework" at page 8. in \cite{manuscript}). A subset $P$ of $M$ is a \emph{PAC$_{\text{Hru}}$} subset of $M$, if every multiplicity $1$ formula with parameters from $P$ has a solution in $P$.
\end{definition}

\begin{definition}[Definition 3.1 in \cite{PilPol}]\label{PilPol.PAC.def}
Let $T'$ be stable, $\kappa\geqslant |T'|^+$ a cardinal, $M\models T'$, 
and let $P$ be an $\mathcal{L}$-substructure of $M$. We say that $P$ is a \emph{$\kappa$-PAC$_{\text{PP}}$} 
substructure of $M$ if whenever $A\subseteq P$ has cardinality smaller than $\kappa$ and $p(x)$ is complete stationary type over $A$ (in the sense of $M$), then $p$ has a realization in $P$.
\end{definition}

\begin{prop}\label{PACvsPACPP}
Assume that $T'$ is stable, allows quantifier elimination and elimination of imaginaries. 
Let $\kappa\geqslant|T'|^+$ and let $N$ be an $\mathcal{L}$-substructure of $\mathfrak{D}$. 
\begin{enumerate}
\item If the substructure $N$ is $\kappa$-saturated (in the sense of quantifier free part of the $\mathcal{L}$-theory $\theo(N)$) and PAC, then $N$ is a $\kappa$-PAC$_{\text{PP}}$ substructure.

\item If $N$ is a $\kappa$-PAC$_{\text{PP}}$ substructure, then it is a PAC substructure.
\end{enumerate}
\end{prop}

\begin{proof}
For the proof of the first point, assume that $N$ is $\kappa$-saturated and PAC and let $p$ be a stationary type over $A\subseteq N$, where $|A|<\kappa$. We extend $p$ to a non-forking extension over $N$, say $p'$, which is also stationary. By Lemma \ref{PACclaim}, for each $m\models p'$, the extension $N\subseteq\dcl_{\mathcal{L}}^{\mathfrak{D}}(Nm)$ is regular, so $N\leqslant_1\dcl_{\mathcal{L}}^{\mathfrak{D}}(Nm)$ (PAC).
Let us choose $m\models p'$.
Because $N\leqslant_1\dcl_{\mathcal{L}}^{\mathfrak{D}}(Nm)$ and the set of formulas
$$\lbrace\psi(x)\;|\;\psi(x)\in \qftp_{\mathcal{L}}^{\mathfrak{D}}(m/A)\rbrace$$
is consistent in $N$, the saturation of $N$ implies existence of $n\in N$ such that
$$n\models\qftp_{\mathcal{L}}^{\mathfrak{D}}(m/A),$$
which implies $n\models p$.

It remains to show that being a $\kappa$-PAC$_{\text{PP}}$ substructure implies being a PAC substructure. Assume that $N\subseteq N'$ is regular, $m\in N'$, $n$ is a finite tuple from $N$ and for some quantifier free $\mathcal{L}$-formula it follows $\varphi^{\mathfrak{D}}(n,m)$. Regularity of $N\subseteq N'$ implies
$$\dcl_{\mathcal{L}}^{\mathfrak{D}}(Nm)\cap\acl_{\mathcal{L}}^{\mathfrak{D}}(N)=\dcl_{\mathcal{L}}^{\mathfrak{D}}(N),$$
hence by Lemma \ref{PACclaim}, we obtain that $\tp_{\mathcal{L}}^{\mathfrak{D}}(m/N)$ is stationary.

We introduce the type $p:=\tp_{\mathcal{L}}^{\mathfrak{D}}(m/\acl_{\mathcal{L}}^{\mathfrak{D}}(N))$, which is stationary.
Cleary, the type $p|_N$ is stationary and $m\ind_N^{\mathfrak{D}}\acl_{\mathcal{L}}^{\mathfrak{D}}(N)$, so, by Remark 2.26.iii) in \cite{anandgeometric}, $\cb(p)\subseteq\dcl_{\mathcal{L}}^{\mathfrak{D}}(N)$. There exists $n\subseteq N_0\subseteq N$ such that $|N_0|\leqslant|T|$ and $m\ind_{N_0} N$,
hence $m\ind_{\acl_{\mathcal{L}}^{\mathfrak{D}}(N_0)}^{\mathfrak{D}}\acl_{\mathcal{L}}^{\mathfrak{D}}(N)$ and $p|_{\acl_{\mathcal{L}}^{\mathfrak{D}}(N_0)}$ is stationary and therefore $\cb(p)\subseteq\acl_{\mathcal{L}}^{\mathfrak{D}}(N_0)$. 
Note that
$\cb(p)\subseteq\dcl_{\mathcal{L}}^{\mathfrak{D}}\big(\dcl_{\mathcal{L}}^{\mathfrak{D}}(N)\cap\acl_{\mathcal{L}}^{\mathfrak{D}}(N_0)\big)$, so - again by Remark 2.26.iii) in \cite{anandgeometric} -
$p|_{\dcl_{\mathcal{L}}^{\mathfrak{D}}(N)\cap\acl_{\mathcal{L}}^{\mathfrak{D}}(N_0)}$ is stationary.
Because $|\dcl_{\mathcal{L}}^{\mathfrak{D}}(N)\cap\acl_{\mathcal{L}}^{\mathfrak{D}}(N_0)|\leqslant|T|$, type $p|_{\dcl_{\mathcal{L}}^{\mathfrak{D}}(N)\cap\acl_{\mathcal{L}}^{\mathfrak{D}}(N_0)}$ has a realization in $N$, say $m_N\in N$. It follows $\varphi^{\mathfrak{D}}(n,m_N)$. Since $\varphi$ is quantifier free, it is also $N\models\exists x\;\varphi(n,x)$.
\end{proof}

We see from Proposition \ref{PACvsPACPP} that the definition of a PAC substructure from \cite{PilPol} is stronger (or at least not weaker) than our definition of a PAC substructure, but is implied by our definition of a PAC substructure if the substructure is saturated enough. Something similar happens in the case of strongly minimal theories. Mainly, by \cite[Remark 3.7]{PilPol} definition of a PAC substructure from \cite{PilPol} implies Hrushovski's definition of a PAC substructure for strongly minimal theories (\cite{manuscript}), and Hrushovski's definition of a PAC substructure, if the substructure is saturated, implies definition of a PAC substructure from \cite{PilPol}.
We visualize the situation:
\begin{itemize}
\item in the set-up of a strongly minimal theory
$$\text{PAC}_{\text{Hru}}\;\;+\;\;\kappa-\text{saturated}\quad\subseteq\quad\kappa-\text{PAC}_{\text{PP}}\quad\subseteq\quad\text{PAC}_{\text{Hru}},$$

\item in the set-up of a stable theory
$$\text{PAC}_{\text{reg}}\;\;+\;\;\kappa-\text{saturated}\quad\subseteq\quad\kappa-\text{PAC}_{\text{PP}}\quad\subseteq\quad\text{PAC}_{\text{reg}},$$
\end{itemize}
where by PAC$_{\text{Hru}}$ we denote Hrushovski's definition of a PAC substructure, and by PAC$_{\text{reg}}$ our definition of a PAC substructure (only in this case, we use subscript for our definition of a PAC substructure). There is a very natural question: whether our definition of a PAC substructure and Hrushovski's definition of a PAC substructure describe the same objects in the strongly minimal set-up? The next result provides the answer.

\begin{prop}\label{PAC.Hru}
Assume that $T'$ is strongly minimal, allows quantifier elimination and eliminates imaginaries (and satisfies assumptions from Hrushovski's manuscript, \cite{manuscript},
needed for his definition of a PAC substructure, which we do not use in the proof, see Definition \ref{hru.PAC.def}).
Let $N$ be an $\mathcal{L}$-substructure of $\mathfrak{D}$.
The following are equivalent.
\begin{enumerate}
\item The substructure $N$ is a PAC substructure.
\item The substructure $N$ is a PAC$_{\text{Hru}}$ substructure.
\end{enumerate}
\end{prop}

\begin{proof}
Assume that $N$ is a PAC substructure, $n$ is a finite tuple from $N$ and the $\mathcal{L}$-formula $\varphi(n,x)$ is of multiplicity $1$. Chose $m\in\mathfrak{D}$ such that $\varphi^{\mathfrak{D}}(n,m)$ and let $p:=\tp_{\mathcal{L}}^{\mathfrak{D}}(m/N)$. We see that $p$ is stationary, hence by Lemma \ref{PACclaim}, 
$N\subseteq\dcl_{\mathcal{L}}^{\mathfrak{D}}(Nm)$ is a regular extension. Our assumption implies that $N$ is existentially closed in $\dcl_{\mathcal{L}}^{\mathfrak{D}}(Nm)$. Quantifier elimination implies that there exists a realization of $\varphi(n,x)$ in $N$.

We proceed now to the proof of the second implication.
Assume that $N$ is a PAC$_{\text{Hru}}$ substructure, $N'$ is regular over $N$, $n$ is a finite tuple from $N$, $m\in N'$ and for a quantifier free $\mathcal{L}$-formula $\varphi(y,x)$ we have $\varphi^{\mathfrak{D}}(n,m)$.

Because of regularity and Lemma \ref{PACclaim}, we obtain that $p:=\tp_{\mathcal{L}}^{\mathfrak{D}}(m/N)$ is stationary, hence $\mlt(p)=1$. Therefore for some $\psi(n',x)\in p$ we have $\mlt(\psi)=1$. Since $\mlt(\psi\wedge\varphi)=1$, there exists $m'\in N$ such that
$$\psi^{\mathfrak{D}}(n',m')\wedge\varphi^{\mathfrak{D}}(n,m').$$
Because $\varphi$ is quantifier free, we obtain
$$N\models\exists x\;\varphi(n,x).$$
\end{proof}

We tried to generalize Lemma 1.17 from \cite{manuscript} to our, i.e. stable, context. Such a generalization would help in the description of structures with a $G$-action (see Remark \ref{frattini2}).
The hope was hidden in using our definition of a PAC substructure.
Unfortunately we met some problems at the end of our proof of the following conjecture, where the absolute Galois group of a $\mathcal{L}$-substructure of $\mathfrak{D}$ is the group $\aut_{\mathcal{L}}(\acl_{\mathcal{L}}^{\mathfrak{D}}(N)/N)$.

\begin{conj}
Assume that $T'$ is stable, eliminates quantifiers and imaginaries. The absolute Galois group of a definably closed PAC substructure of $\mathfrak{D}$ is projective.
\end{conj}

However, we proved a weaker version of the above conjecture, mainly:

\begin{prop}
Assume that $T'$ is stable, eliminates quantifiers and imaginaries. The absolute Galois group of a definably closed \textbf{bounded} PAC substructure of $\mathfrak{D}$ is projective.
\end{prop}

Proof of the above proposition and possible generalizations of the proposition will be provided in a subsequent paper.

\subsection{Preliminaries from general Galois theory}
There is a really nice introduction to the general Galois theory (check \cite{invitation}), but it does not cover the infinite extensions case (an extension is finite, if it is generated by a finite tuple, Definition 2. in \cite{invitation}). Another source of needed facts from the general Galois theory, is \cite{casfar}, which covers also infinite extensions. However, the authors of \cite{casfar} do not require in the definition of Galois extension $A\subseteq B$ that $A$ is definably closed, and that is more natural for us (it corresponds to perfect fields in the fields case).
Therefore, we provide, using our favourite approach, the necessary facts in this subsection.
From this point we assume that $T'$ allows to \textbf{eliminate quantifiers}.

\begin{lemma}\label{fhat_lemma}
If $M,M'$ are small $\mathcal{L}$-substructures of $\mathfrak{D}$ and there is
an $\mathcal{L}$-isomorphism $f:M\to M'$, then there exists $\hat{f}\in\aut_{\mathcal{L}}(\mathfrak{D})$ such that $\hat{f}|_M=f$.
\end{lemma}

\begin{proof}
Very easy and standard, let $\varphi(x)$ be an $\mathcal{L}$-formula and $\varphi_0(x)$ a quantifier-free $\mathcal{L}$-formula which is equivalent modulo $T'$ to $\varphi(x)$. For any tuple $m\subseteq M$ we have
$$\mathfrak{D}\models \varphi(m)\;\iff\;\mathfrak{D}\models\varphi_0(m)\;\iff\;M\models\varphi_0(m)\;\iff$$ $$M'\models\varphi_0(f(m))\;\iff\;\mathfrak{D}\models\varphi_0(f(m))\;\iff\mathfrak{D}\models\varphi(f(m)).$$
Because $\mathfrak{D}$ is $\kappa_{\mathfrak{D}}$-strongly homogeneous, there exists a proper $\hat{f}$.
\end{proof}

\begin{definition}\label{galois.ext.def}
\begin{enumerate}
\item Assume that $A\subseteq C$ are small $\mathcal{L}$-substructures of $\mathfrak{D}$. We say that $C$ is \emph{normal over $A$} (or we say that $A\subseteq C$ is a \emph{normal extension}) if $\aut_{\mathcal{L}}(\mathfrak{D}/A)\cdot C\subseteq C$.

\item Assume that $A\subseteq C\subseteq\acl_{\mathcal{L}}^{\mathfrak{D}}(A)$ are small $\mathcal{L}$-substructures of $\mathfrak{D}$ such that $A=\dcl_{\mathcal{L}}^{\mathfrak{D}}(A)$, $C=\dcl_{\mathcal{L}}^{\mathfrak{D}}(C)$ and $C$ is normal over $A$. In this situation we say that $A\subseteq C$ is a \emph{Galois extension}.
\end{enumerate}
\end{definition}

\begin{fact}\label{normal_morph}
If $A\subset C$ is normal and $f\in\aut_{\mathcal{L}}(\mathfrak{D}/A)$, then $f(C)=C$.
\end{fact}

\begin{proof}
By the definition of a normal extension, it follows $f(C)\subseteq C$ and $f^{-1}(C)\subseteq C$, which implies $f(C)\subseteq C$ and $C\subseteq f(C)$.
\end{proof}

\begin{fact}[Corollary 7. in \cite{invitation}]
Let $A$, $B$ and $C$ be small $\mathcal{L}$-substructures of $\mathfrak{D}$ such that $A\subseteq B\subseteq C\subseteq \acl_{\mathcal{L}}^{\mathfrak{D}}(A)$, $C$ and $B$ are normal over $A$. Then
$$\xymatrix{1 \ar[r] & \aut_{\mathcal{L}}(C/B) \ar[r]^{\subseteq} & \aut_{\mathcal{L}}(C/A) \ar[r]^{|_B} & \aut_{\mathcal{L}}(B/A) \ar[r] & 1}$$
is an exact sequence and hence $\aut_{\mathcal{L}}(C/B)\trianglelefteqslant \aut_{\mathcal{L}}(C/A)$.
\end{fact}

\begin{proof}
It is an easy generalization of \cite[Corollary 7.]{invitation}, which uses quantifier elimination and Fact \ref{normal_morph}.
\end{proof}

\begin{remark}[Krull topology]
For a small $\mathcal{L}$-substructure $C$ of $\mathfrak{D}$ we consider \emph{Krull topology} on $C^C$ (functions from $C$ to $C$). The base of open sets for this topology is given by the family of sets of the following form
$$U(\bar{a},\bar{b}):=\lbrace f\in C^C\;|\; f(\bar{a})=\bar{b}\rbrace,$$
where $\bar{a}$ and $\bar{b}$ are finite tuples from $C$ of the same length.
\end{remark}

Our goal is to define a topology on $\aut_{\mathcal{L}}(C/A)$ for a Galois extension (or more generally for a normal extension) $A\subseteq C$ such that $\aut_{\mathcal{L}}(C/A)$ will be a topological group. Of course, we will consider topology induced from the Krull topology on $C^C$, but to do this we need to state that $\aut_{\mathcal{L}}(C/A)$ is a closed subset (Fact \ref{krull.topology}). Checking that the group operation and taking inverse is continuous in the induced topology is straightforward and omitted.

\begin{fact}\label{krull.topology}
If $A\subseteq C$ is a normal extension, then $\aut_{\mathcal{L}}(C/A)$ is closed subset of $C^C$ (in the Krull topology).
\end{fact}

\begin{proof}
It is a standard proof which we borrow from a book about field theory written by Jerzy Browkin (\cite{browkin}, p. 134).
Assume that $f\in C^C$ belongs to the closure of $\aut_{\mathcal{L}}(C/A)$. Take any 
function $F:C^n\to C^m$ and any $r$-ary relation $R$ which belong to the $\mathcal{L}$-structure of $C$ (i.e. they correspond to function/relation symbols of the language $\mathcal{L}$). Then take any $a\in A$, $c,c'\in C$, $\bar{c}_1,\bar{c}_2\subseteq C$ such that $|\bar{c}_1|=n$ and $|\bar{c}_2|=r$, and consider $U(\bar{a},\bar{b})$ for
$$\bar{a}=(a,c,c',\bar{c}_1,F(\bar{c}_1),\bar{c}_2),\qquad\bar{b}=\big(f(a),f(c),f(c'),f(\bar{c}_1),f(F(\bar{c}_1)),f(\bar{c}_2)\big).$$
Because $f$ belongs to the closure of $\aut_{\mathcal{L}}(C/A)$ and $U(\bar{a},\bar{b})$ is an open neighbourhood of $f$, there is an automorphism $h\in U(\bar{a},\bar{b})\cap\aut_{\mathcal{L}}(C/A)$.
Therefore
$$f(a)=h(a)=a,$$
$$f(F(\bar{c}_1))=h(F(\bar{c}_1))=F(h(\bar{c}_1))=F(f(\bar{c}_1)),$$
$$R(\bar{c}_2)\;\iff\;R(h(\bar{c}_2))\;\iff\; R(f(\bar{c}_2)),$$
$$f(c)=f(c')\;\Rightarrow\;h(c)=h(c')\;\Rightarrow\;c=c'.$$
That means that $f:C\to C$ is an $\mathcal{L}$-monomorphism of $C$ over $A$. After use of Lemma \ref{fhat_lemma} for $M=C$ and $M'=f(C)$, we can can treat $f$ as an automorphism of $\mathfrak{D}$ over $A$, which by Fact \ref{normal_morph} turns out to be also an automorphisms of $C$ over $A$.
\end{proof}

\begin{fact}
Let $A\subseteq C$ and $A\subseteq B$ be Galois extensions and let $B\subseteq C$. Then all mappings in the sequence
$$\xymatrix{1 \ar[r] & \aut_{\mathcal{L}}(C/B) \ar[r]^{\subseteq} & \aut_{\mathcal{L}}(C/A) \ar[r]^{|_B} & \aut_{\mathcal{L}}(B/A) \ar[r] & 1}$$
are continuous.
\end{fact}

\begin{proof}
Straightforward, if we use the above definition of the Krull topology.
\end{proof}

\begin{fact}\label{fact314}
Assume that $A\subseteq C$ is a Galois extension and $A\subseteq B=\dcl_{\mathcal{L}}^{\mathfrak{D}}(B)\subseteq C$. The extension $A\subseteq B$ is Galois if and only if
$$\aut_{\mathcal{L}}(C/B)\trianglelefteqslant \aut_{\mathcal{L}}(C/A).$$
\end{fact}

\begin{proof}
The proof of \cite[Lemma 9.]{invitation} still works for the infinite extensions case.
\end{proof}

From now we assume that $T'$ additionally admits \textbf{elimination of imaginaries}.

\begin{fact}[The Galois correspondence]\label{galois.correspondence}
Let $A\subseteq C$ be a Galois extension, introduce
$$\mathcal{B}:=\lbrace B\;|\;A\subseteq B=\dcl_{\mathcal{L}}^{\mathfrak{D}}(B)\subseteq C\rbrace,$$
$$\mathcal{H}:=\lbrace H\;|\; H\leqslant\aut_{\mathcal{L}}(C/A)\text{ is closed}\rbrace.$$
Then $\alpha(B):=\aut_{\mathcal{L}}(C/B)$ is a mapping between $\mathcal{B}$ and $\mathcal{H}$, $\beta(H):=C^H$ is a mapping between $\mathcal{H}$ and $\mathcal{B}$ and it follows
$$\alpha\circ\beta=\id,\qquad\beta\circ\alpha=\id.$$
\end{fact}

\begin{proof}
We omit the proof. Again it is a standard one, which uses at some point \cite[Theorem 12.]{invitation}.
\end{proof}

\begin{remark}\label{remark.gal}
Assume that $A\subseteq C$ is a Galois extension, $A\subseteq B_1,B_2\subseteq C$, $B_1=\dcl_{\mathcal{L}}^{\mathfrak{D}}(B_1)$, $B_2=\dcl_{\mathcal{L}}^{\mathfrak{D}}(B_2)$ and $H_1,H_2\leqslant\aut_{\mathcal{L}}(C/A)$ are closed subgroups. It follows
\begin{enumerate}
\item $\aut_{\mathcal{L}}(C/\dcl_{\mathcal{L}}^{\mathfrak{D}}(B_1\cup B_2))=
\aut_{\mathcal{L}}(C/B_1)\cap\aut_{\mathcal{L}}(C/B_2)$,

\item $C^{H_1\cap H_2}=\dcl_{\mathcal{L}}^{\mathfrak{D}}(C^{H_1}\cup C^{H_2})$.
\end{enumerate}
\end{remark}

\begin{fact}\label{fact317}
If $A\subseteq C$ is a Galois extension, then $\aut_{\mathcal{L}}(C/A)$ is a profinite group.
\end{fact}

\begin{proof}
We need to show that $\aut_{\mathcal{L}}(C/A)$ is a Hausdorff, compact and totally disconnected space. 
We will skip this standard proof.
The reader may adjust the proof of Proposition 7.8 from J.S. Milne ``Fields and Galois theory" (\cite{milneFT}), but with
changes which are necessary in our general context, and using the aforementioned definition of the Krull topology.

\end{proof}

We end this subsection with a very useful lemma, which states that for any $G$-action on an $\mathcal{L}$-structure $N$ (contained in $\mathfrak{D}$) the extension $N^G\subseteq N$ is always a Galois extension.

\begin{lemma}\label{N_Galois}
Assume that $N$ is a small definably closed $\mathcal{L}$-substructure of $\mathfrak{D}$ equipped with a $G$-action $(\tau_g)_{g\in G}$. 
\begin{enumerate}
\item
If $N\subseteq\acl_{\mathcal{L}}^{\mathfrak{D}}(N^G)$, then for each $b\in N$ we have
$$\aut_{\mathcal{L}}(\mathfrak{D}/N^G)\cdot b=G\cdot b,$$
hence $N^G\subseteq N$ is a Galois extension.

\item
If $G$ is finite, then $N\subseteq\acl_{\mathcal{L}}^{\mathfrak{D}}(N^G)$, hence also the first point follows.
\end{enumerate}
\end{lemma}

\begin{proof}
We omit the proof of $\dcl_{\mathcal{L}}^{\mathfrak{D}}(N^G)=N^G$.
Both items will be proved simultaneously.

Let $b_1\in N$ and consider
$$F:=\lbrace \tau_g(b_1)\;|\;g\in G\rbrace\subseteq N.$$
Because of Lemma \ref{fhat_lemma}, the set $F$ can be regarded	as a subset of
$$\aut_{\mathcal{L}}(\mathfrak{D}/N^G)\cdot b_1.$$
If we assume that $N\subseteq\acl_{\mathcal{L}}^{\mathfrak{D}}(N^G)$ (the first item) or $|G|<\infty$ (the second item), it follows that $F$ is finite.
Say $F=\lbrace b_1,\ldots,b_m\rbrace$, where $m\in\mathbb{N}$. We will show that
$$\aut_{\mathcal{L}}(\mathfrak{D}/N^G)\cdot b_1=F.$$
By elimination of imaginaries there exists a code $b_F\in\mathfrak{D}$ for the set $F$. Set $F$ is $\lbrace b_1,\ldots,b_m\rbrace$-definable, so $b_F\in\dcl_{\mathcal{L}}^{\mathfrak{D}}(b_1,\ldots,b_m)\subseteq\dcl_{\mathcal{L}}^{\mathfrak{D}}(N)=N$.
We extend, by Lemma \ref{fhat_lemma}, each $\tau_g$ to $\tilde{\tau}_g\in\aut_{\mathcal{L}}(\mathfrak{D}/N^G)$. Now we observe that
\begin{IEEEeqnarray*}{rCl}
(\forall g\in G)(\tilde{\tau}_g(F)=F) &\iff & (\forall g\in G)(\tilde{\tau}_g(b_F)=b_F)\\
& \iff & (\forall g\in G)(\tau_g(b_F)=b_F) \\
&\iff & b_F\subseteq N^G \\
&\Rightarrow &  (\forall f\in\aut_{\mathcal{L}}(\mathfrak{D}/N^G))(f(b_F)=b_F) \\
&\iff & (\forall f\in\aut_{\mathcal{L}}(\mathfrak{D}/N^G))(f(F)=F).
\end{IEEEeqnarray*}
The last item implies that
$$\aut_{\mathcal{L}}(\mathfrak{D}/N^G)\cdot b_1\subseteq F.$$
\end{proof}

\subsection{Galois groups versus group $G$}\label{subs:galois}
The idyllic situation will be the following one: an $\mathcal{L}$-theory $T$ has a model companion $T^{\mc}$  which has quantifier elimination and elimination of imaginaries (and is stable), and the model companion $T_G^{\mc}$ of the theory $T_G$ exists,
$$\xymatrixcolsep{7.5pc}\xymatrix{ T \ar@{~}^*+\txt{\small adding \\ \small group action}[r] \ar@{~}[d]_*+\txt{\small passing to\\ \small model companion} & T_G \ar@{~}[d]^*+\txt{\small passing to\\ \small model companion} \\ T^{\mc} & T_G^{\mc}}$$
Existence of $T^{\mc}$ and $T_G^{\mc}$ is a strong assumption and therefore we want to avoid it (at least until Section \ref{sec:forking}). Note that in the idyllic situation $T_{\forall}\subseteq T^{\mc}$.

Recall that we are working in a big $\mathcal{L}$-structure $\mathfrak{D}$ such that $T'=\theo_{\mathcal{L}}(\mathfrak{D})$ allows quantifier elimination and elimination of imaginary elements (in the idyllic situation $T'$ is a completion of $T^{\mc}$). 
Now, let $M$ be a small $\mathcal{L}$-substructure of $\mathfrak{D}$ 
%(hence a model of $(T')_{\forall}\,$) 
equipped with a $G$-action $(\sigma_g)_{g\in G}$. Our goal is to study existentially closed models of some $\mathcal{L}^G$-theory, which will correspond in the idyllic situation to $T_G$. 
Assume that $T$ is an $\mathcal{L}$-theory such that $T_{\forall}\subseteq T'$.
We have the following sequence of implications
$$\xymatrixcolsep{3.5pc}\xymatrix{
(M,(\sigma_g)_{g\in G})\text{ is existentially closed among models of }T_G \ar@{=>}[d]^{\text{ since }(T_{\forall})_G\supseteq(T_G)_{\forall}}\\
(M,(\sigma_g)_{g\in G})\text{ is existentially closed among models of }(T_{\forall})_G \ar@{=>}[d]^{\text{ since }T_{\forall}\subseteq T'}\\
(M,(\sigma_g)_{g\in G})\text{ is an existentially closed model of }(T'_{\forall})_G}$$
We easily see that the last line is the most general assumption if we aim to study 
$\mathcal{L}$-substructures of $\mathfrak{D}$ equipped with a $G$-action, which are
existentially closed among models of $T_G$ for some $\mathcal{L}$-theory $T$ (such that $T_{\forall}\subseteq T'$).
Therefore we can even avoid introducing $T$ and focus only on existentially closed models of $(T'_{\forall})_G$ (i.e on existentially closed $\mathcal{L}$-substructures of $\mathfrak{D}$ with a $G$-action), which is an inductive theory.

Let us summarize the assumptions:
\begin{itemize}
\item $\mathfrak{D}$ is a properly saturated $\mathcal{L}$-structure,

\item $T'=\theo_{\mathcal{L}}(\mathfrak{D})$ allows quantifier elimination and elimination of imaginary elements,

\item $M$ is a small $\mathcal{L}$-substructure of $\mathfrak{D}$ equipped with a $G$-action $(\sigma_g)_{g\in G}$,

\item $(M,(\sigma_g)_{g\in G})$ is existentially closed model of $(T'_{\forall})_G$.

\end{itemize}

\begin{comment}
\vspace{7mm}
\hrule
\vspace{5mm}

, which will be a basic assumption in our ``real-life" situation (before Section \ref{sec:forking}, we do not assume existence of $T^{\mc}$ nor existence of $T_G^{\mc}$).

Assume that for an $\mathcal{L}$-theory $T$, it follows $T_{\forall}\subseteq T'$.

such that $(M,(\sigma_g)_{g\in G})$ is existentially closed among models of $T_G$ , however there is no reason to not do this).

One could ask how far we are from assuming that $T'$ is the model companion of $T$. The theory $T'$ has quantifier elimination, hence it is model complete. We assume $T_{\forall}\subseteq (T')_{\forall}$, but it lacks $(T')_{\forall}\subseteq T_{\forall}$. But if we assume that $T=\theo_{\mathcal{L}}(M)$, then it follows $\big(\theo_{\mathcal{L}}(M)\big)_{\forall}=T_{\forall}=(T')_{\forall}$.
\end{comment}

\begin{remark}\label{dcl_M}
It follows $\dcl_{\mathcal{L}}^{\mathfrak{D}}(M)=M$.
\end{remark}

\begin{proof}
For each $m\in\dcl_{\mathcal{L}}^{\mathfrak{D}}(M)$ we chose a quantifier free $\mathcal{L}$-formula $\varphi_m(y,x)$ and a finite tuple $c_m\subseteq M$ such that
$$\varphi_m(c_m,\mathfrak{D})=\lbrace m\rbrace.$$
Now, we will extend $G$-action $(\sigma_g)_{g\in G}$ from $M$ to $\tilde{M}:=\dcl_{\mathcal{L}}^{\mathfrak{D}}(M)$. 
For each $g\in G$ chose, by Lemma \ref{fhat_lemma}, 
$\tilde{\sigma}_g\in\aut_{\mathcal{L}}(\mathfrak{D})$ 
such that $\tilde{\sigma}_g|_M=\sigma_g$. We observe that
\begin{itemize}
\item $\tilde{\sigma}_1|_{\tilde{M}}=\id_{\tilde{M}}$,
\item $\tilde{\sigma}_g(\tilde{M})\subseteq \tilde{M}$ for each $g\in G$,
\item $\tilde{\sigma}_h\tilde{\sigma}_g(m)=\tilde{\sigma}_{hg}(m)$ for each $h,g\in G$ and each $m\in \tilde{M}$,
\item $m=\tilde{\sigma}_g(\tilde{\sigma}_{g^{-1}}(m))\in \im\tilde{\sigma}_g$, for each $m\in \tilde{M}$ and each $g\in G$ (hence $\tilde{\sigma}|_{\tilde{M}}$ is surjection on $\tilde{M}$),
\end{itemize}
therefore $(\tilde{\sigma}_g|_{\tilde{M}})_{g\in G}\leqslant\aut_{\mathcal{L}}(\tilde{M})$ defines a $G$-action on $\tilde{M}$ which extends the $G$-action $(\sigma_g)_{g\in G}$ on $M$.

 Because $\tilde{M}\subseteq\mathfrak{D}$, it follows
$(\tilde{M},(\tilde{\sigma}_g|_{\tilde{M}})_{g\in G})\models (T'_{\forall})_G$.
Since we have $\varphi_m^{\mathfrak{D}}(c_m,m)$ and $\varphi_m(y,x)$ is quantifier free, we obtain that $\varphi_m^{\tilde{M}}(c_m,m)$, so $\tilde{M}\models\exists x\;\varphi_m(c_m,x)$. Hence $M\models\exists x\;\varphi_m(c_m,x)$, say $\varphi_m^{M}(c_m,m_M)$ for some $m_M\in M$. Again, since $\varphi(y,x)$ is quantifier free, it is $\varphi_m^{\mathfrak{D}}(c_m,m_M)$, so $m=m_M\in M$. Therefore $\dcl_{\mathcal{L}}^{\mathfrak{D}}(M)\subseteq M$.
\end{proof}

\begin{cor}\label{cor321}
If $A\subseteq M$, then $\dcl_{\mathcal{L}}^{\mathfrak{D}}(A)\subseteq M$.
\end{cor}

Assume that $N$ is a small $\mathcal{L}$-substructure of $\mathfrak{D}$ equipped with a $G$-action $(\tau_g)_{g\in G}$.
From Remark \ref{dcl_M} and Corollary \ref{cor321} we see that the assumption
$$\dcl_{\mathcal{L}}^{\mathfrak{D}}(N)=N$$
is harmless 
if we aim to study $\mathcal{L}^G$-structures which are existentially closed among models of $T_G$ (or
if we aim to study models of $T_G^{\mc}$). Therefore we make it, i.e. from now on \textbf{we assume that $\dcl_{\mathcal{L}}^{\mathfrak{D}}(N)=N$}.

\begin{remark}\label{inv.perfect}
It follows $\dcl_{\mathcal{L}}^{\mathfrak{D}}(N^G)=N^G$.
\end{remark}

\begin{proof}
Assume that $n\in\dcl_{\mathcal{L}}^{\mathfrak{D}}(N^G)\subseteq N$ and for some $\mathcal{L}$-formula $\varphi(y,x)$ and some finite tuple $c\subseteq N^G$ it is
$$\varphi(c,\mathfrak{D})=\lbrace n\rbrace.$$
Let $g\in G$ and extend $\tau_g$ to $\tilde{\tau}_g\in\aut_{\mathcal{L}}(\mathfrak{D})$ (by Lemma \ref{fhat_lemma}). Then
$$\varphi^{\mathfrak{D}}(c,n)\;\Rightarrow\;\varphi^{\mathfrak{D}}(\tilde{\tau}_g(c),\tilde{\tau}_g(n))\;\Rightarrow\;\varphi^{\mathfrak{D}}(c,\tau_g(n)),$$
which implies $\tau_g(n)=n$.
\end{proof}

\begin{remark}\label{dcl_F}
We have $\dcl_{\mathcal{L}}^{\mathfrak{D}}(N\cap\acl_{\mathcal{L}}^{\mathfrak{D}}(N^G))=N\cap\acl_{\mathcal{L}}^{\mathfrak{D}}(N^G)$.
\end{remark}

\begin{proof}
The intersection of definably closed sets is a definably closed set.
\end{proof}

We will see in Remark \ref{equiv_remark} that a $G$-action on models of $T_G^{\mc}$ (which will be studied in Section \ref{sec:forking}) is mostly determined by that what happens on the relative algebraic closure (i.e. the algebraic closure in $\mathfrak{D}$ intersected with the universe of a model of $T_G^{\mc}$) of empty set or the set of invariants. Therefore, we are especially interested in the description of a $G$-action on sets of the form $F:=N\cap\acl_{\mathcal{L}}^{\mathfrak{D}}(N^G)$.

We note here that $(\tau_g|_F)_{g\in G}$ is a $G$-action on $F$. Of course $\tau_g(F)\subseteq F$ for every $g\in G$. Moreover, $(\tau_g|_F:F\to F)_{g\in G}$ is a collection of $\mathcal{L}$-monomorphisms. We use $\tau_g|_F\circ\tau_{g^{-1}}|_F=\tau_{gg^{-1}}|_F=\id_F$ to see that each $\tau_g|_F$ is onto.

\begin{lemma}\label{F_Galois}
The extension 
$N^G\subseteq F$ is a Galois extension.
\end{lemma}

\begin{proof}
We need to check the assumptions of Lemma \ref{N_Galois} for the $G$-action $(\tau_g|_F)_{g\in G}$. Note that,
by Remark \ref{dcl_F}, it follows $\dcl_{\mathcal{L}}^{\mathfrak{D}}(F)=F$.
Because $N^G\subseteq F\subseteq N$, we see that
$F^G= N^G$, therefore $F\subseteq\acl_{\mathcal{L}}^{\mathfrak{D}}(N^G)=\acl_{\mathcal{L}}^{\mathfrak{D}}(F^G)$. By the first point of Lemma \ref{N_Galois} it follows that $F^G\subseteq F$ is a Galois extension, but $F^G=N^G$.
\end{proof}

\begin{cor}
We have the following:
\begin{enumerate}
\item $\aut_{\mathcal{L}}(\acl_{\mathcal{L}}^{\mathfrak{D}}(N^G)/F)\trianglelefteqslant \aut_{\mathcal{L}}(\acl_{\mathcal{L}}^{\mathfrak{D}}(N^G)/N^G)$ (by Lemma \ref{F_Galois} and Fact \ref{fact314}),

\item $\aut_{\mathcal{L}}(\acl_{\mathcal{L}}^{\mathfrak{D}}(N^G)/N^G)$ and $\aut_{\mathcal{L}}(\acl_{\mathcal{L}}^{\mathfrak{D}}(N^G)/F)$ are profinite groups	 (by Fact \ref{fact317} and Remark \ref{dcl_F},

\item $\aut_{\mathcal{L}}(F/N^G)$ is a profinite group (by Fact \ref{fact317} and Lemma \ref{F_Galois}).
\end{enumerate}
\end{cor}

The following proposition binds together a trace of the group $G$ with some Galois group.
There are similar results in this paper which use
the notion of a \emph{Frattini cover}, e.g. Definition \ref{frattini0}, Corollary \ref{frattini1} and Remark \ref{frattini2}.

\begin{prop}\label{prop.generators}
The group $\aut_{\mathcal{L}}(F/N^G)$ is generated as a profinite group by $(\tau_g|_{F})_{g\in G}$.
\end{prop}

\begin{proof}
It is enough to show that $\cl(H)=\aut_{\mathcal{L}}(F/N^G)$, where $H:=(\tau_g|_{F})_{g\in G}$ and $\cl$ denotes the closure in the topological sense. 
By the Galois correspondence (Fact \ref{galois.correspondence}), we obtain that $N^G\subseteq F^{\cl(H)}$. Because $H\subseteq\cl(H)$, it follows $N^G=F^G=F^H\supseteq F^{\cl(H)}$. Using $F^{\cl(H)}=N^G=F^{\aut_{\mathcal{L}}(F/N^G)}$ and again the Galois correspondence, we see that $\cl(H)=\aut_{\mathcal{L}}(F/N^G)$.
\end{proof}

We end this subsection with a general and technical lemma, which will be more useful after assuming stability (see Proposition \ref{prop.G.closed} and the proof of Proposition \ref{inv.bounded}).

\begin{lemma}\label{lemma37}
Let $\mathcal{G}:=\aut_{\mathcal{L}}(\acl_{\mathcal{L}}^{\mathfrak{D}}(N^G)/N^G)$ and
$\mathcal{N}:=\aut_{\mathcal{L}}(\acl_{\mathcal{L}}^{\mathfrak{D}}(N^G)/F)$. The following are equivalent.
\begin{enumerate}
\item There exists an $\mathcal{L}$-substructure $N'$ of $\mathfrak{D}$ equipped with a $G$-action $(\tau'_g)_{g\in G}$, which extends the $G$-action of $(F,(\tau_g|_F)_{g\in G})$, such that $N^G\subsetneq (N')^G$ and $N'\subseteq \acl_{\mathcal{L}}^{\mathfrak{D}}(N^G)$.

\item There exists a closed subgroup $\mathcal{G}_0<\mathcal{G}$ such that $\mathcal{G}_0\neq\mathcal{G}$ and $\mathcal{G}_0\mathcal{N}=\mathcal{G}$.
\end{enumerate}
\end{lemma}

\begin{proof}
The original proof of \cite[Lemma 3.7]{nacfa}, which is adapted here, was mostly made by Piotr Kowalski.
We know that $N^G\subseteq F$ is a Galois extension and we know what are the topological generators of the profinite group $\mathcal{G}/\mathcal{N}\cong\aut_{\mathcal{L}}(F/N^G)$ (by Proposition \ref{prop.generators}).

We assume point 1) and prove point 2). 
Without loss of generality, we assume that $N'=\dcl_{\mathcal{L}}^{\mathfrak{D}}(N')$ (if not, then we can extend, in a unique way, the $G$-action $(\tau'_g)_{g\in G}$ on definable closure of $N'$).
Let us define
$$\mathcal{G}_0:=\aut_{\mathcal{L}}(\acl_{\mathcal{L}}^{\mathfrak{D}}(N^G)/(N')^G)$$
and observe that $(N')^G\subseteq \dcl_{\mathcal{L}}^{\mathfrak{D}}((N'^G)\cup F)$ is Galois (by Lemma \ref{F_Galois}). We have the following commuting diagram
$$\xymatrix{\mathcal{G}=\aut_{\mathcal{L}}(\acl_{\mathcal{L}}^{\mathfrak{D}}(N^G)/N^G) 
\ar[r]^{\pi}& \aut_{\mathcal{L}}(F/N^G) \\
\mathcal{G}_0=\aut_{\mathcal{L}}(\acl_{\mathcal{L}}^{\mathfrak{D}}(N^G)/(N')^G)
\ar[r]^{\pi'} \ar[u]_{\subseteq}
& \aut_{\mathcal{L}}(\dcl_{\mathcal{L}}^{\mathfrak{D}}((N')^G\cup F)/(N')^G)
\ar[u]^{\alpha},}$$
where $\pi$, $\pi'$ and $\alpha$ are restrictions. It is easy, but not necessary, to notice that $\ker\alpha=\lbrace\id\rbrace$.
Moreover $\alpha$ is continuous, its domain and codomain are Hausdorff and compact spaces. Because image of a compact space is compact and a compact subset of Hausdorff space is closed, we obtain that $\im\alpha$ is a closed subset of $\aut_{\mathcal{L}}(F/N^G)$.
Note that, for each $g\in G$,
$$\alpha\Big(\tau'_g|_{\dcl_{\mathcal{L}}^{\mathfrak{D}}\big((N')^G\cup F\big)}\Big)=\tau_g|_F,$$
hence
$$(\tau_g|_F)_{g\in G}\subseteq\im\alpha\subseteq\aut_{\mathcal{L}}(F/N^G)$$
and, because $(\tau_g|_F)_{g\in G}$ is a dense subset, the map $\alpha$ is onto.
Since $\alpha\circ\pi'$ is surjective,
also the following composition of maps
$$\mathcal{G}_0\subseteq\mathcal{G}\xrightarrow{\pi}\aut_{\mathcal{L}}(F/N^G)\cong\mathcal{G}/\mathcal{N}$$
is surjective. Hence $\pi(\mathcal{G}_0)=(\mathcal{G}_0\mathcal{N})/\mathcal{N}=\mathcal{G}/\mathcal{N}$ and $\mathcal{G}_0\mathcal{N}=\mathcal{G}$.
Because $N^G\subseteq \acl_{\mathcal{L}}^{\mathfrak{D}}(N^G)$ is Galois, $N^G\subseteq (N')^G\subseteq\acl_{\mathcal{L}}^{\mathfrak{D}}(N^G)$ and $\dcl_{\mathcal{L}}^{\mathfrak{D}}((N')^G)=(N')^G$ (by Remark \ref{inv.perfect} for $N'$), we obtain that $\mathcal{G}_0$ is closed. The condition $N^G\subsetneq (N')^G$ implies $\mathcal{G}_0\neq \mathcal{G}$.

Now, we assume point 2) and prove point 1). Let $N':=\acl_{\mathcal{L}}^{\mathfrak{D}}(N^G)^{\mathcal{G}_0\cap\mathcal{N}}$
 and $N'_0:=\acl_{\mathcal{L}}^{\mathfrak{D}}(N^G)^{\mathcal{G}_0}$. Since $\mathcal{G}_0\neq\mathcal{G}$, it follows that $N^G\subsetneq N'_0$. Note that, by Remark \ref{remark.gal}, we have
\begin{IEEEeqnarray*}{rCl}
N' &=& \acl_{\mathcal{L}}^{\mathfrak{D}}(N^G)^{\mathcal{G}_0\cap\mathcal{N}}\\
   &=& \dcl_{\mathcal{L}}^{\mathfrak{D}}\big(\acl_{\mathcal{L}}^{\mathfrak{D}}(N^G)^{\mathcal{G}_0}\cup\acl_{\mathcal{L}}^{\mathfrak{D}}(N^G)^{\mathcal{N}}\big)\\
   &=& \dcl_{\mathcal{L}}^{\mathfrak{D}}(N'_0\cup F).
\end{IEEEeqnarray*}
By the above we note that
$$\aut_{\mathcal{L}}(\mathfrak{D}/N'_0)\cdot N'\subseteq N',$$
so $N'_0\subseteq N'$ is Galois.
Consider the following commuting diagram
$$\xymatrix{ \mathcal{G}_0\cap\mathcal{N}=\aut_{\mathcal{L}}(\acl_{\mathcal{L}}^{\mathfrak{D}}(N^G)/N') \ar[r]^-{\subseteq} \ar[d]_{\subseteq} & \mathcal{G}_0=\aut_{\mathcal{L}}(\acl_{\mathcal{L}}^{\mathfrak{D}}(N^G)/N'_0) \ar[r]^-{\pi'} \ar[d]_{\subseteq}& \aut_{\mathcal{L}}(N'/N'_0) \ar[d]_{\alpha} \\
\mathcal{N}=\aut_{\mathcal{L}}(\acl_{\mathcal{L}}^{\mathfrak{D}}(N^G)/F) \ar[r]^-{\subseteq} & \mathcal{G}=\aut_{\mathcal{L}}(\acl_{\mathcal{L}}^{\mathfrak{D}}(N^G)/N^G) \ar[r]^-{\pi} & \aut_{\mathcal{L}}(F/N^G),}$$
where $\pi$, $\pi'$ and $\alpha$ are restrictions. Because $N'=\dcl_{\mathcal{L}}^{\mathfrak{D}}(N'_0\cup F)$, we see that $\ker\alpha=\lbrace\id\rbrace$.
Moreover, note that
$$\aut_{\mathcal{L}}(F/N^G)=\pi(\mathcal{G})=\pi(\mathcal{G}_0\mathcal{N})=\pi(\mathcal{G}_0)=\alpha\pi'(\mathcal{G}_0),$$
so $\alpha$ is an isomorphism. Finally, we can define a $G$-action on $N'$, which extends $(\tau_g|F)_{g\in G}$. Set $\tau'_g:=\alpha^{-1}(\tau_g|F)$, where $g\in G$. Because $\mathcal{G}_0\cap\mathcal{N}\subseteq \mathcal{N}$, by the Galois correspondence (Fact \ref{galois.correspondence}), we obtain
$$\acl_{\mathcal{L}}^{\mathfrak{D}}(N^G)\supseteq N'=\acl_{\mathcal{L}}^{\mathfrak{D}}(N^G)^{\mathcal{G}_0\cap\mathcal{N}}\supseteq\acl_{\mathcal{L}}^{\mathfrak{D}}(N^G)^{\mathcal{N}}=F.$$
Since $N'_0\subseteq N'$ is Galois and $(\tau'_g)_{g\in G}\leqslant\aut_{\mathcal{L}}(N'/N'_0)$, the Galois correspondence implies that $N'_0\subseteq (N')^G$. 
Because $\mathcal{G}_0\neq \mathcal{G}$, it follows $N'_0\supsetneq N^G$, in particular, we get $N^G\subsetneq (N')^G$.
\end{proof}

\subsection{Regularity in stable case}\label{subs:stable.case}
Let us remind that we are working in a big $\mathcal{L}$-structure $\mathfrak{D}$ such that $T'=\theo_{\mathcal{L}}(\mathfrak{D})$ has quantifier elimination and elimination of imaginaries. Now, we add one more, but a stronger assumption: \textbf{$T'$ is stable}.

We begin with a sequence of facts leading to Corollary \ref{regular.PAPA}, which will be the main tool in many proofs in the rest of this paper.

\begin{fact}
Let $E,A\subseteq\mathfrak{D}$ and let $A$ be $\mathcal{L}$-regular over $E$.
Then there exists a unique extension of $\tp_{\mathcal{L}}^{\mathfrak{D}}(a/E)$ to a type over $A$ for each $a\in\acl_{\mathcal{L}}^{\mathfrak{D}}(E)$.
\end{fact}

\begin{proof}
It is standard and we only sketch the main steps. Let $X$ be the set of realizations of $\tp_{\mathcal{L}}^{\mathfrak{D}}(a/A)$. The set $X$ is finite and it is contained in $\acl_{\mathcal{L}}^{\mathfrak{D}}(E)$, hence almost $E$-definable. Therefore, the code of $X$, say $d$, belongs to $\acl_{\mathcal{L}}^{\mathfrak{D}}(E)$.
From the definition of being a code and definability of $X$ over $A$, we note that also $d\in\dcl_{\mathcal{L}}^{\mathfrak{D}}(A)$ and so (by the regularity) $d\in\dcl_{\mathcal{L}}^{\mathfrak{D}}(E)$. The last thing implies that $X$ is invariant under action of $\aut_{\mathcal{L}}(\mathfrak{D}/E)$, so $\aut_{\mathcal{L}}(\mathfrak{D}/A)\cdot a= X=\aut_{\mathcal{L}}(\mathfrak{D}/E)\cdot a$. Therefore, for any $a'\in\mathfrak{D}$, we have $a'\equiv_E a$ if and only if $a'\equiv_A a$.
\end{proof}

\begin{fact}\label{regular}
Let $E,A\subseteq\mathfrak{D}$, $A$ be $\mathcal{L}$-regular over $E$, $f_1,f_2\in\aut_{\mathcal{L}}(\mathfrak{D})$ and let $f_1|_E=f_2|_E$. Then there exists $h\in\aut_{\mathcal{L}}(\mathfrak{D})$ such that $h|_A=f_1|_A$ and $h|_{\acl_{\mathcal{L}}^{\mathfrak{D}}(E)}=f_2|_{\acl_{\mathcal{L}}^{\mathfrak{D}}(E)}$.
\end{fact}

\begin{proof}
We know that $f_1(A)$ is regular over $f_1(E)$ (Remark \ref{regular.remark}.(4)).
Our aim is to show that for any $\mathcal{L}$-formula $\varphi$,  any tuple $a\in A$ and any tuple $b\in\acl_{\mathcal{L}}^{\mathfrak{D}}(E)$ it follows
$$\varphi^{\mathfrak{D}}(a,b)\quad\iff\quad\varphi^{\mathfrak{D}}(f_1(a),f_2(b)).$$
Then we can naturally extend a partial elementary map given by $f_1$ and $f_2$.

Because $f_1|_E=f_2|_E$, we have
$$\tp_{\mathcal{L}}^{\mathfrak{D}}(f_2(b)/f_1(E))\;\subseteq\;\tp_{\mathcal{L}}^{\mathfrak{D}}(f_2(b)/f_1(Ea)),\;\;\tp_{\mathcal{L}}^{\mathfrak{D}}(f_1(b)/f_1(Ea)).$$
But, by the previous fact, there is only one extension of $\tp_{\mathcal{L}}^{\mathfrak{D}}(f_2(b)/f_1(E))$ to a type over $f_1(A)$, hence also only one extension to a type over $f_1(Ea)$ and so
\begin{equation*}
\tp_{\mathcal{L}}^{\mathfrak{D}}(f_2(b)/f_1(Ea))\;=\;tp_{\mathcal{L}}^{\mathfrak{D}}(f_1(b)/f_1(Ea)).
\end{equation*}
\end{proof}

\begin{cor}\label{reg.stationary}
Assume that $E\subseteq A$ is $\mathcal{L}$-regular and let $A_0\subseteq A$. The type $\tp_{\mathcal{L}}^{\mathfrak{D}}(A_0/E)$ has exactly one extension to a type over $\acl_{\mathcal{L}}^{\mathfrak{D}}(E)$. Therefore $\tp_{\mathcal{L}}^{\mathfrak{D}}(A_0/E)$ is stationary.
\end{cor}

\begin{proof}
Let $f\in\aut_{\mathcal{L}}(\mathfrak{D}/E)$. 
We can extend, by Fact \ref{regular}, $f:A_0\to f(A_0)$ and $\id:\acl_{\mathcal{L}}^{\mathfrak{D}}(E)\to
\acl_{\mathcal{L}}^{\mathfrak{D}}(E)$, to $h\in\aut_{\mathcal{L}}(\mathfrak{D}/\acl_{\mathcal{L}}^{\mathfrak{D}}(E))$ such that $h(a)=f(a)$ for each $a\in A_0$. Therefore
$A_0\equiv_E A_0'$ implies $A_0\equiv_{\acl_{\mathcal{L}}^{\mathfrak{D}}(E)} A_0'$.
\end{proof}

\begin{lemma}\label{PACclaim}
For a small set $E\subseteq\mathfrak{D}$ and a complete type $p$ over $E$ it follows:
\begin{IEEEeqnarray*}{rCl}
p\text{ is stationary} &\iff & (\forall A_0\models p)(E\subseteq EA_0\text{ is $\mathcal{L}$-regular}) \\
 &\iff & (\exists A_0\models p)(E\subseteq EA_0\text{ is $\mathcal{L}$-regular}).
\end{IEEEeqnarray*}
\end{lemma}

\begin{proof}
First of all, note that, because the $\mathcal{L}$-regularity does not change under an action by automorphisms, the second equivalence is obvious. Hence, we need only to prove the first equivalence.

From left to right, by the contraposition. 
This part of the proof arose by help of Thomas Scanlon and was improved by help of Piotr Kowalski.
Let $A_0\models p$ be such that there exists an element
$$b\in\dcl_{\mathcal{L}}^{\mathfrak{D}}(EA_0)\cap\acl_{\mathcal{L}}^{\mathfrak{D}}(E)\setminus\dcl_{\mathcal{L}}^{\mathfrak{D}}(E).$$
Because $b\not\in\dcl_{\mathcal{L}}^{\mathfrak{D}}(E)$, it follows that there exists $f\in\aut_{\mathcal{L}}(\mathfrak{D}/E)$ such that $f(b)\neq b$.
Since $b\in\acl_{\mathcal{L}}^{\mathfrak{D}}(E)$ we have
$$(a_j)_{j\in J}\ind_E^{\mathfrak{D}}b,\qquad \big(f(a_j)\big)_{j\in J}\ind_E^{\mathfrak{D}} b,$$
where $(a_j)_{j\in J}$ is some enumeration of $A_0$.

Assume that $(a_j)_{j\in J}\equiv_{Eb}\big(f(a_j)\big)_{j\in J}$ and let $h\in\aut_{\mathcal{L}}(\mathfrak{D}/Eb)$ witnesses it, i.e. $h\big((a_j)_{j\in J}\big)=\big(f(a_j)\big)_{j\in J}$. Note that $hf|_{EA_0}=\id_{EA_0}$ and so, by $b\in\dcl_{\mathcal{L}}^{\mathfrak{D}}(EA_0)$, it follows $hf(b)=b$.
Hence $f(b)=h^{-1}(b)=b$ and we get a contradiction. Therefore $(a_j)_{j\in J}\not\equiv_{Eb}\big(f(a_j)\big)_{j\in J}$.

We obtain at least two different non forking extensions of $p$ to $Eb$, which is not possible.

Implication from right to left. Assume that $E\subseteq EA_0$ is $\mathcal{L}$-regular for some $A_0\models p$, i.e. $p=\tp_{\mathcal{L}}^{\mathfrak{D}}(A_0/E)$. It is enough to use Corollary \ref{reg.stationary}.
\end{proof}

\begin{cor}\label{cor.stationary_types_exist}
For every small $\mathcal{L}$-substructure $N$ of $\mathfrak{D}$ and every $n<\omega$, there exists a non-algebraic stationary type over $N$ in $n$ many variables.
\end{cor}

\begin{proof}
Consider a small $N'\succeq N$ such that $|N'\setminus N|\geqslant n$. Without loss of generality, we may assume that $N'\subseteq\mathfrak{D}$. By Lemma \ref{lang410}, $N\subseteq N'$ is regular. Take a tuple $a\subseteq N'\setminus N$ of length $n$. By Remark \ref{regular.remark}.(6), also $N\subseteq Na$ is regular, hence, by Lemma \ref{PACclaim}, $\tp_{\mathcal{L}}^{\mathfrak{D}}(a/N)$ is stationary.
\end{proof}

%\begin{prop}\label{prop.PAC_realizes_stationary}
%If $P$ is a small PAC $\mathcal{L}$-substructure of $\mathfrak{D}$ and $p(x)$ is a stationary type (in the sense of $\mathfrak{D}$) over $P$, then quantifier free part of $p(x)$ is consistent with $\theo_{\mathcal{L}}(P)$. Therefore $|P|^{+}$-saturated extension $P'\succeq P$ realizes every stationary type (in the sense of $\mathfrak{D}$) over $P$.
%\end{prop}

\begin{fact}\label{PAPA}
Assume that $E,A,B\subseteq\mathfrak{D}$, $E=\acl_{\mathcal{L}}^{\mathfrak{D}}(E)$, $f_1,f_2\in\aut_{\mathcal{L}}(\mathfrak{D})$ and $f_1|_E=f_2|_E$. If $A\ind^{\mathfrak{D}}_E B$ and $f_1(A)\ind^{\mathfrak{D}}_{f_1(E)}f_2(B)$, then there exists $h\in\aut_{\mathcal{L}}(\mathfrak{D})$ such that $h|_A=f_1|_A$ and $h|_B=f_2|_B$.
\end{fact}

\begin{proof}
It just rephrases the idea from the proof of \cite[Theorem 3.3]{lascar91a} to a little generalization of the PAPA (\cite[Definition 3.1]{lascar91a}, \cite[Definiton 3.3]{ChaPil}).

Forking independence of $A$ and $B$ over $E$ implies forking independence of $f_1(A)$ and $f_1(B)$ over $f_1(E)$. As in the proof of Fact \ref{regular}, we want to show that
for any $\mathcal{L}$-formula $\varphi$,  any tuple $a\in A$ and any tuple $b\in B$ it follows
$$\varphi^{\mathfrak{D}}(a,b)\quad\iff\quad\varphi^{\mathfrak{D}}(f_1(a),f_2(b)).$$

Again, because $f_1|_E=f_2|_E$,
$$\tp_{\mathcal{L}}^{\mathfrak{D}}(f_2(b)/f_1(E))\;\subseteq\;\tp_{\mathcal{L}}^{\mathfrak{D}}(f_2(b)/f_1(Ea)),\;\;\tp_{\mathcal{L}}^{\mathfrak{D}}(f_1(b)/f_1(Ea)).$$
We obtained two non-forking extensions of a stationary type.
\end{proof}

\begin{cor}\label{regular.PAPA}
Assume that $E,A,B\subseteq\mathfrak{D}$, $A$ is $\mathcal{L}$-regular over $E$, $f_1,f_2\in\aut_{\mathcal{L}}(\mathfrak{D})$, $f_1|_E=f_2|_E$. If $A\ind^{\mathfrak{D}}_E B$ and $f_1(A)\ind^{\mathfrak{D}}_{f_1(E)} f_2(B)$, then there exists $h\in\aut_{\mathcal{L}}(\mathfrak{D})$ such that $h|_A=f_1|_A$ and $h|_B=f_2|_B$.
\end{cor}

\begin{proof}
By Fact \ref{regular}, there exists $f_1'\in\aut_{\mathcal{L}}(\mathfrak{D})$ such that $f_1'|_A=f_1|_A$ and $f_1'|_{\acl_{\mathcal{L}}^{\mathfrak{D}}(E)}=f_2|_{\acl_{\mathcal{L}}^{\mathfrak{D}}(E)}$.

From the assumptions we obtain the following
$$A\acl_{\mathcal{L}}^{\mathfrak{D}}(E)\ind^{\mathfrak{D}}_{\acl_{\mathcal{L}}^{\mathfrak{D}}(E)} B\acl_{\mathcal{L}}^{\mathfrak{D}}(E),\quad 
f_1(A)\acl_{\mathcal{L}}^{\mathfrak{D}}(f_2(E))\ind^{\mathfrak{D}}_{\acl_{\mathcal{L}}^{\mathfrak{D}}(f_2(E))} f_2(B)\acl_{\mathcal{L}}^{\mathfrak{D}}(f_2(E)).$$
Finally we use Fact \ref{PAPA} for $f_1'$ and $f_2$.
\end{proof}

Now, after Corollary \ref{regular.PAPA}, we can provide three more facts about regularity: Lemma \ref{lang413}, Corollary \ref{cor.413} and Remark \ref{rem.413}.
These facts are generalization of similar properties of the ``classical regularity" (i.e. regularity from the algebra of fields).

\begin{lemma}\label{lang413}
Assume that $E\subseteq A$ is $\mathcal{L}$-regular, $E\subseteq B$ and $B\ind_E^{\mathfrak{D}} A$,
then $B\subseteq BA$ is $\mathcal{L}$-regular.
\end{lemma}

\begin{proof}
Let $a\in\dcl_{\mathcal{L}}^{\mathfrak{D}}(BA)\cap\acl_{\mathcal{L}}^{\mathfrak{D}}(B)$ and let $f\in\aut_{\mathcal{L}}(\mathfrak{D}/B)$. Because of
$$\acl_{\mathcal{L}}^{\mathfrak{D}}(B)\ind_E^{\mathfrak{D}} A$$
and by Corollary \ref{regular.PAPA}, we extend $f:\acl_{\mathcal{L}}^{\mathfrak{D}}(B)\to\acl_{\mathcal{L}}^{\mathfrak{D}}(B)$ and $\id: A\to A$, to $h\in\aut_{\mathcal{L}}(\mathfrak{D}/BA)$. We see that $f(a)=h(a)$, which is equal to $a$ ($a$ is definable over $BA$). Because $f$ was arbitrary, $a\in\dcl_{\mathcal{L}}^{\mathfrak{D}}(B)$.
\end{proof}

\begin{cor}\label{cor.413}
Assume that $E\subseteq A$ and $E\subseteq B$ are $\mathcal{L}$-regular, and $B\ind_E^{\mathfrak{D}} A$,
then $E\subseteq BA$ is $\mathcal{L}$-regular.
\end{cor}

\begin{proof}
By Lemma \ref{lang413} and Remark \ref{regular.remark}.
\end{proof}

\begin{cor}\label{PAC.substructures}
If $P'$ is a small PAC $\mathcal{L}$-substructure of $\mathfrak{D}$ and $P\leqslant_1 P'$, then $P$ is PAC.
\end{cor}

\begin{proof}
Assume that $P\subseteq N$ is regular and $N\models\exists\;x\;\varphi(p,x)$ for some $p\subseteq P$ and quantifier free $\mathcal{L}$-formula. Let $N'\equiv_P N$ be such that $P'\ind^{\mathfrak{D}}_P N'$. By Lemma \ref{lang413}, $P'\subseteq\dcl_{\mathcal{L}}^{\mathfrak{D}}(P',N')$ is regular, but $\dcl_{\mathcal{L}}^{\mathfrak{D}}(P',N')\models\exists\;x\;\varphi(p,x)$ and $P'$ is PAC, thus $P'\models\exists\;x\;\varphi(p,x)$.
\end{proof}

\begin{remark}\label{rem.413}
Assume that $E\subseteq A$ and $E\subseteq B$ are normal extensions.
\begin{enumerate}
\item Obviously, the extension $E\subseteq\dcl_{\mathcal{L}}^{\mathfrak{D}}(AB)$ is normal.

\item
It $E\subseteq A$ and $E\subseteq B$ are Galois, then $E\subseteq\dcl_{\mathcal{L}}^{\mathfrak{D}}(AB)$ is also Galois.

\item
The map
$$\Phi:\aut_{\mathcal{L}}(\dcl_{\mathcal{L}}^{\mathfrak{D}}(AB)/E)\to\aut_{\mathcal{L}}(A/E)\times\aut_{\mathcal{L}}(B/E),$$
given by $f\mapsto (f|_A,f|_B)$, is a continuous embedding.

\item
If $A\ind_E^{\mathfrak{D}}B$ and $E\subseteq A$ is $\mathcal{L}$-regular, then $\Phi$ is onto, hence an isomorphism (by Corollary \ref{regular.PAPA}).
\end{enumerate}
\end{remark}

\subsection{Structures with $G$-action, boundedness and PAC}\label{subsec:PAC}
Now we come back to the $G$-action case and start with introducing the following general convention.

\begin{definition}
Let $(M,\bar{\sigma})$ be an $\mathcal{L}^G$-structure and let $M\subseteq N$ be an extension of $\mathcal{L}$-structures. For $f\in\aut(N)$ by $(M^f,\bar{\sigma}^f)$ we denote the $\mathcal{L}^G$-structure $(f(M),(f\circ\sigma_g\circ f^{-1})_{g\in G})$.
\end{definition}
\noindent
We keep assumptions from the beginning of Subsection \ref{subs:stable.case}
and fix the following:
\begin{itemize}
\item 
an $\mathcal{L}$-substructure $M$ of $\mathfrak{D}$, which is equipped with a $G$-action $(\sigma_g)_{g\in G}$ such that $(M,(\sigma_g)_{g\in G})$ is an existentially closed model of $(T'_{\forall})_G$,

\item $F:=M\cap\acl_{\mathcal{L}}^{\mathfrak{D}}(M^G)$,
\end{itemize}
and consider profinite groups
\begin{itemize}
\item $\mathcal{G}:=\aut_{\mathcal{L}}(\acl_{\mathcal{L}}^{\mathfrak{D}}(M^G)/M^G)$,
\item $\mathcal{N}:=\aut_{\mathcal{L}}(\acl_{\mathcal{L}}^{\mathfrak{D}}(M^G)/F)$.
\end{itemize}

\begin{prop}\label{prop.G.closed}
The negation of each from equivalent conditions from Lemma \ref{lemma37} is satisfied when we replace $(N,(\tau_g)_{g\in G})$ from the formulation of Lemma \ref{lemma37} with $(M,(\sigma_g)_{g\in G})$.
\end{prop}

\begin{proof}
Assume that there is an $\mathcal{L}$-substructure $N'\subseteq\acl_{\mathcal{L}}^{\mathfrak{D}}(M^G)$, $F\subseteq N'$, equipped with a $G$-action $(\tau'_g)_{g\in G}$ which extends $(\sigma_g|_F)_{g\in G}$. We will show that $(N')^G=M^G$.

We start with checking that $F\subseteq M$ is $\mathcal{L}$-regular. By Remark \ref{dcl_M} and Remark \ref{dcl_F} we obtain the following
\begin{IEEEeqnarray*}{rCl}
\dcl_{\mathcal{L}}^{\mathfrak{D}}(M)\cap\acl_{\mathcal{L}}^{\mathfrak{D}}(M\cap\acl_{\mathcal{L}}^{\mathfrak{D}}(M^G)) &=& M\cap\acl_{\mathcal{L}}^{\mathfrak{D}}(M\cap\acl_{\mathcal{L}}^{\mathfrak{D}}(M^G)) \\
&\subseteq & M\cap\acl_{\mathcal{L}}^{\mathfrak{D}}(M^G) = \dcl_{\mathcal{L}}^{\mathfrak{D}}(M\cap\acl_{\mathcal{L}}^{\mathfrak{D}}(M^G)),
\end{IEEEeqnarray*}
which implies $\mathcal{L}$-regularity. Now we ``move" $N'$ by some $f\in\aut_{\mathcal{L}}(\mathfrak{D}/F)$ to obtain
$$M\ind_F^{\mathfrak{D}} (N')^f.$$
By Corollary \ref{regular.PAPA}, we extend each pair $(\sigma_g,f\tau'_gf^{-1})$, where $g\in G$ to an $\mathcal{L}$-automorphism of $\mathfrak{D}$ and then note that
$\langle M,(N')^f\rangle_{\mathcal{L}}$ with these extensions of pairs of automorphisms is a model of $(T'_{\forall})_G$, let us denote it by $(\tilde{M},(\tilde{\sigma}_g)_{g\in G})$. Note that
$$(M,(\sigma_g)_{g\in G})\leqslant_1(\tilde{M},(\tilde{\sigma}_g)_{g\in G}).$$

Let $n\in N'\subseteq\acl_{\mathcal{L}}^{\mathfrak{D}}(M^G)$, let $\varphi(y,x)$ be a quantifier free $\mathcal{L}$-formula and let $m$ be a finite tuple from $M^G$ such that $\varphi(m,\mathfrak{D})$ is finite and contains $n$. We have
$\varphi(m,\tilde{M})\neq\emptyset$, so $\varphi(m,M)\neq\emptyset$. 
By Lemma \ref{F_Galois}, $M^G\subseteq F=M\cap\acl_{\mathcal{L}}^{\mathfrak{D}}(M^G)$ is Galois, hence - as a normal extension - it contains $n$.
Therefore $N'\subseteq F\subseteq M$, so $(N')^G=M^G$.
\end{proof}

\begin{definition}\label{frattini0}
Let $H,H'$ be profinite groups and $\pi: H\to H'$ be a continuous epimorphism.
The mapping $\pi$ is called a \emph{Frattini cover} if for each closed subgroup $H_0$ of $H$, the condition $\pi(H_0)=H'$ implies that $H_0=H$.
\end{definition}

\begin{cor}\label{frattini1}
The restriction map
$$\pi:\aut_{\mathcal{L}}(\acl_{\mathcal{L}}^{\mathfrak{D}}(M^G)/M^G)\to\aut_{\mathcal{L}}(M\cap\acl_{\mathcal{L}}^{\mathfrak{D}}(M^G)/M^G)$$
is a Frattini cover.
\end{cor}

\begin{proof}
By Proposition \ref{prop.G.closed}.
\end{proof}

\begin{remark}\label{frattini2}
The above corollary can be formulated in a, in some sense, stronger version
in the case of the theory of fields, i.e. Theorem 3.40 in \cite{nacfa} or Theorem 5. and 6. in \cite{sjogren}. Theorem 3.40 in \cite{nacfa} states that for $\mathfrak{D}\models\acf$, if $G$ is a finite group, then
$$\aut_{\mathcal{L}}(\acl_{\mathcal{L}}^{\mathfrak{D}}(M^G)/M^G)\to G$$
is the \emph{universal Frattini cover} (i.e. we require additionally that the profinite group 
$$\aut_{\mathcal{L}}(\acl_{\mathcal{L}}^{\mathfrak{D}}(M^G)/M^G)$$
 is projective, consult Proposition 22.6.1 in \cite{FrJa} or Theorem 3.38 in \cite{nacfa}). By Proposition \ref{prop.generators}, we see that 
 $$G=\aut_{\mathcal{L}}(M\cap\acl_{\mathcal{L}}^{\mathfrak{D}}(M^G)/M^G),$$
  so Corollary \ref{frattini1} states that the map 
$\aut_{\mathcal{L}}(\acl_{\mathcal{L}}^{\mathfrak{D}}(M^G)/M^G)\to G$ is a Frattini cover. We will see (Lemma \ref{inv.pac}) that the invariants for a finitely generated group are a PAC substructure, but the absolute Galois group of a PAC field always is projective (Theorem 11.6.2 in \cite{FrJa}), hence for fields we can easily add ``universal" before ``Frattini cover" from the thesis of \cite[Theorem 3.40]{nacfa}.
\end{remark}

\begin{definition}\label{Gal.bounded}
We say that a small subset $A$ of $\mathfrak{D}$ is \emph{Galois bounded} if the profinite group
$$\aut_{\mathcal{L}}(\acl_{\mathcal{L}}^{\mathfrak{D}}(A)/\dcl_{\mathcal{L}}^{\mathfrak{D}}(A))$$ 
is small, i.e.: if for every $n\in\mathbb{N}_{>0}$ it has only finitely many closed subgroups of index $n$.
\end{definition}

\begin{prop}\label{inv.bounded}
If $G$ is finitely generated, then $\mathcal{G}$ is finitely generated as a profinite group.
\end{prop}

\begin{proof}
The proof follows the proof of \cite[Theorem 3.22]{nacfa}. 
First we recall that group $\aut_{\mathcal{L}}(F/M^G)\cong\mathcal{G}/\mathcal{N}$
and is generated as a profinite group by $(\sigma_g|_F)_{g\in G}$ (by Proposition \ref{prop.generators}), hence
$$\mathcal{G}/\mathcal{N}=\cl\big(\langle \sigma_g|_F\;|\; g\in G\rangle\big)=\cl\big(\langle \sigma_{g_i}|_F\;|\; i\leqslant n\rangle\big),$$
where $g_1,\ldots,g_n$ generate $G$.

Choose finitely many generators $h_i\mathcal{N}$, $i\leqslant n$, of the topological group $\mathcal{G}/\mathcal{N}$. Let $\mathcal{G}_0$ be the topological subgroup of $\mathcal{G}$ generated by $\lbrace h_1,\ldots,h_n\rbrace$. Because $\mathcal{G}_0\mathcal{N}=\mathcal{G}$, by
Proposition \ref{prop.G.closed} it must be $\mathcal{G}_0=\mathcal{G}$.
\end{proof}

\begin{cor}\label{inv.bounded2}
If $G$ is finitely generated, then $M^G$ is Galois bounded.
\end{cor}

\begin{proof}
The above proposition and \cite[Proposition 2.5.1.a)]{ribzal}, which says that a finitely generated profinite group is Galois bounded, imply the corollary.
\end{proof}

\begin{prop}\label{inv.pac}
If $G$ is finitely generated then the $\mathcal{L}$-substructure $M^G$ is PAC.
\end{prop}

\begin{proof}
Assume that there is an $\mathcal{L}$-regular extension $M^G\subseteq N'$ and $N'\models\exists x\;\varphi(c,x)$ for some quantifier free $\mathcal{L}$-formula $\varphi(y,x)$ and some finite tuple $c\subseteq M^G$. We ``move" $N'$ by some $f\in\aut_{\mathcal{L}}(\mathfrak{D}/M^G)$ to $N'':=f(N')$ such that
$$N''\ind_{M^G}^{\mathfrak{D}} M.$$
Because $M^G\subseteq N''$ is $\mathcal{L}$-regular, like in previous proofs, we can extend $G$-action from $M$ to $\langle M\cup N''\rangle_{\mathcal{L}}$ (on $N''$ by the trivial action) and get a model of $(T'_{\forall})_G$, say $(\tilde{M},(\tilde{\sigma}_g)_{g\in G})$. 
If $n\in N'$ satisfies $\varphi(c,x)$, then $f(n)$ also satisfies 
$\varphi(c,x)$ in $N''$ and in $\tilde{M}$. 
We see that $f(n)\in\tilde{M}^G$ and $\varphi(y,x)$ is quantifier free, so 
$$\tilde{M}\models\exists x\;\varphi(c,x)\;\wedge\;\bigwedge\limits_{g\text{ is generator of }G}\sigma_g(x)=x.$$
We use that $(M,(\sigma_g)_{g\in G})$ is existentially closed in $(\tilde{M},(\tilde{\sigma}_g)_{g\in G})$ to obtain a solution of the formula
$$\varphi(c,x)\;\wedge\;\bigwedge\limits_{g\text{ is generator of }G}\sigma_g(x)=x$$
in $M$, and so in $M^G$.
\end{proof}

\begin{remark}
Consider $G$-actions on fields (i.e. $\mathcal{L}$ is the language of rings, models of $T$ are fields and $\mathfrak{D}\models\acf$), where $G$ is finitely generated. Then Remark \ref{inv.perfect}, Proposition \ref{inv.bounded} and Proposition \ref{inv.pac} combined with \cite[Fact 2.6.7]{kim1}, prove that the $\mathcal{L}$-theory of a subfield of invariants is supersimple.
\end{remark}

The interesting question is whether the phenomena from the above remark is valid in a more general situation? 
It is under the assumption of existence of $T_G^{\mc}$, and we will come back to this in Section \ref{sec:forking}, where the existence of $T_G^{\mc}$ is assumed. Before moving forward, we will observe that for finitely generated $G$, invariants are bounded in the sense of another approach to the ``boundedness".
Namely, \cite[Definition 2.3]{PilPol}:

\begin{definition}\label{PP.bounded}
Assume that $T'$ is an $\mathcal{L}'$-theory and $M'\models T'$. We say that a definably closed $\mathcal{L}'$-substructure $P'$ of $M'$ is \emph{bounded} if there is some cardinal $\kappa$ such that whenever $(M'',P'')$ is elementarily equivalent to $(M',P')$ (as an $\mathcal{L}'\cup\lbrace P'\rbrace$-structure), then 
$\aut_{\mathcal{L}'}(\acl_{\mathcal{L}'}^{M''}(P'')/P'')$ has cardinality at most $\kappa$.
\end{definition}

Note that by \cite[Remark 2.6.(1)]{PilPol} ``bounded" and ``Galois bounded" in the case of fields means the same.

\begin{lemma}\label{MG.bounded}
If $G$ is finitely generated then $M^G$ is bounded.
\end{lemma}

\begin{proof}
Because of Remark \ref{dcl_M}, $M^G$ is a definably closed $\mathcal{L}$-substructure of $\mathfrak{D}$.
By Proposition \ref{inv.bounded}, we can chose finitely many generators of the topological group $\mathcal{G}$, say $g_1,\ldots,g_n$. Observe that
$$M^G=\big(\acl_{\mathcal{L}}^{\mathfrak{D}}(M^G)\big)^{\mathcal{G}}=\big(\acl_{\mathcal{L}}^{\mathfrak{D}}(M^G)\big)^{\langle g_1,\ldots,g_n\rangle}
=\big(\acl_{\mathcal{L}}^{\mathfrak{D}}(M^G)\big)^{\lbrace g_1,\ldots,g_n\rbrace}.$$
Therefore, by \cite[Lemma 2.4]{PilPol}, $M^G$ is bounded.
\end{proof}

Now we can ask about PAC for $M$ instead of $M^G$. 
It turns out that to obtain the thesis we do not need to assume that $G$ is finitely generated (as in Proposition \ref{inv.pac}, where it was important for definability of $M^G$), but the proof is much more complicated and needs an auxiliary fact. 

\begin{fact}\label{total.independent}
Let $I$ be ordered set of size smaller than $\kappa_\mathfrak{D}$ (small in the sense of saturation of $\mathfrak{D}$). Let
$(p_i(x))_{i\in I}$ be a family of complete $\mathcal{L}$-types over a small subset $A$ of $\mathfrak{D}$ (where $x$ is a finite tuple of variables). There exists sequence $(a_i)_{i\in I}\subseteq\mathfrak{D}$ such that for each $i\in I$ we have $a_i\models p_i(x)$ and $a_i\ind^{\mathfrak{D}}_A (a_j)_{j\neq i}$.
\end{fact}

\begin{proof}
We recursively construct a sequence $(a_i)_{i\in I}$ of realisations of the family
$(p_i(x))_{i\in I}$, i.e. $a_i\models p_i(x)$ for $i\in I$,
such that for every $i\in I$ we have $a_i\ind^{\mathfrak{D}}_A a_{<i}$.
For each $i\in I$, it follows $a_i\ind^{\mathfrak{D}}_A (a_j)_{j\neq i}$.
\end{proof}

\begin{prop}\label{M.is.PAC}
The structure $M$ is PAC as an $\mathcal{L}$-substructure of $\mathfrak{D}$.
\end{prop}

\begin{proof}
The proof was made after a conversation with Piotr Kowalski about a similar fact in the case of the theory of fields (\cite[Theorem 3.]{sjogren}).

Assume	that $M\subseteq N$ is an $\mathcal{L}$-regular extension for some small $\mathcal{L}$-substructure $N$ of $\mathfrak{D}$. Moreover, let $N\models\exists x\;\varphi(m,x)$ for a tuple $m$ from $M$ and a quantifier free $\mathcal{L}$-formula $\varphi(y,x)$. We need to show that $M\models\exists x\;\varphi(m,x)$.

Choose an element $n\in N$ such that $\varphi^{N}(m,n)$. 
By Lemma \ref{fhat_lemma}, for each $g\in G$ the automorphism $\sigma_g$ extends $\hat{\sigma}_g\in\aut_{\mathcal{L}}(\mathfrak{D}/M^G)$.
By Fact \ref{total.independent} for the family $\big(\tp_{\mathcal{L}}^{\mathfrak{D}}(\hat{\sigma}_g(n)/M)\big)_{g\in G}$, there exists a sequence $(a_g)_{g\in G}\subseteq\mathfrak{D}$ and a sequence $(f_g)_{g\in G}\subseteq\aut_{\mathcal{L}}(\mathfrak{D}/M)$ such that
$$a_g=f_g\big(\hat{\sigma}_g(n)\big),\quad\text{and}\quad a_g\ind^{\mathfrak{D}}_M(a_h)_{h\neq g}$$
for all $g\in G$. We assume here that $G$ is ordered as a set, say $(a_g)_{g\in G}=(a_{h_i})_{i\in I}$, and use it to define a $G$-action on $(a_g)_{g\in G}$.
\
\\
\\
\textbf{Claim:}
There exists a collection of elements $(\tau_g)_{g\in G}\subseteq\aut_{\mathcal{L}}(\mathfrak{D})$ such that for each $g,g'\in G$ it follows
\begin{itemize}
\item $\big(\tau_g(a_{h_i})\big)_{i\in I}\subseteq\lbrace a_{h_i}\;|\;i\in I\rbrace$,
\item $\tau_g|_M=\sigma_g|_M$,
\item $\tau_g\circ\tau_{g'}=\tau_{gg'}$ on the set $\lbrace a_{h_i}\;|\;i\in I\rbrace$.
\end{itemize}
\
\\
\\
Proof of the claim:
Let $g\in G$, we will define $\tau_g$ recursively: we will construct a sequence $(\tau_{g,i})_{i\in I}\subseteq\aut_{\mathcal{L}}(\mathfrak{D})$ such that for all $i,j\in I$ it follows that
\begin{itemize}
\item if $j\leqslant i$ then $\tau_{g,i}(a_{h_j})=a_{gh_j}$,

\item $\tau_{g,i}|_M=\sigma_g|_M$.
\end{itemize}

We start with
$$\tau_{g,0}:=f_{gh_0}\circ\hat{\sigma}_{gh_0}\circ\hat{\sigma}_{h_0}^{-1}\circ f_{h_0}^{-1}$$
(to be in accordance with the order convention we should write ``$f_{h_i}$" instead of ``$f_{gh_0}$", where $i\in I$ is such that $h_i=gh_0$).
For the successor step: assume that we constructed $\tau_{g,\alpha}$ and want to define $\tau_{g,\alpha+1}$. Because
$$(a_{h_j})_{j\leqslant\alpha}\ind^{\mathfrak{D}}_M a_{h_{\alpha+1}},\quad
(a_{gh_j})_{j\leqslant\alpha}\ind^{\mathfrak{D}}_M a_{gh_{\alpha+1}},$$
$M\subseteq f_{h_{\alpha+1}}(N)\ni a_{h_{\alpha+1}}$ is an $\mathcal{L}$-regular extension, and because $\tau_{g,\alpha}|_M=\sigma_g|_M$, by Corollary \ref{regular.PAPA}, we can extend
$$\tau_{g,\alpha}\quad\text{and}\quad f_{gh_{\alpha+1}}\circ\hat{\sigma}_{gh_{\alpha+1}}\circ\hat{\sigma}_{h_{\alpha+1}}^{-1}\circ f_{h_{\alpha+1}}^{-1}$$
to $\theta\in\aut_{\mathcal{L}}(\mathfrak{D})$ such that
$$\theta|_{M(a_{h_j})_{j\leqslant\alpha}}=\tau_{g,\alpha}|_{M(a_{h_j})_{j\leqslant\alpha}},\qquad \theta|_{Ma_{h_{\alpha+1}}}=f_{gh_{\alpha+1}}\circ\hat{\sigma}_{gh_{\alpha+1}}\circ\hat{\sigma}_{h_{\alpha+1}}^{-1}\circ f_{h_{\alpha+1}}^{-1}|_{Ma_{h_{\alpha+1}}}.$$
We set $\tau_{g,\alpha+1}:=\theta$.
For $\alpha\in I$ which is a limit ordinal we proceed in the following way. For $i<\alpha$ we introduce 
$$\theta_{i}:=\tau_{g,i}|_{\dcl_{\mathcal{L}}^{\mathfrak{D}}(M,\,a_{h_{\leqslant i}})}$$
and we note that $\theta_j$ extends $\theta_i$ for $i\leqslant j<\alpha$. Hence we define $\theta_{\alpha}:\dcl_{\mathcal{L}}^{\mathfrak{D}}(M,\,a_{h_{<\alpha}})\to
\dcl_{\mathcal{L}}^{\mathfrak{D}}(M,\,a_{gh_{<\alpha}})$ as the limit of $(\theta_i)_{i<\alpha}$. It is an $\mathcal{L}$-isomorphism and so extends, by Lemma \ref{fhat_lemma}, to an element of $\aut_{\mathcal{L}}(\mathfrak{D})$,
say $\hat{\theta}_{\alpha}$. We only need to set $\tau_{g,\alpha}:=\hat{\theta}_{\alpha}$.

Assume that for $g\in G$ we constructed (as in the above process) a sequence $(\tau_{g,i})_{i\in I}\subseteq\aut_{\mathcal{L}}(\mathfrak{D})$. Again we need to use a limit argument to define $\tau_g\in\aut_{\mathcal{L}}(\mathfrak{D})$. 
Now for every $i\in I$, we introduce
$$\theta_{i}:=\tau_{g,i}|_{\dcl_{\mathcal{L}}^{\mathfrak{D}}(M,\,a_{h_{\leqslant i}})}$$
and set $\tau_g$ to be an extension (by Lemma \ref{fhat_lemma}) of the limit of the sequence $(\theta_i)_{i\in I}$.

The family $(\tau_g)_{g\in G}$ defined above satisfies for each $g\in G$ and each $i\in I$
\begin{itemize}
\item $\tau_{g}(a_{h_i})=a_{gh_i}$,

\item $\tau_{g}|_M=\sigma_g|_M$.
\end{itemize}
To reach our first goal (i.e. the thesis of the claim), we need only to check whether for every $g,g'\in G$ it follows
$\tau_g\circ\tau_{g'}=\tau_{gg'}$ on the set $\lbrace a_{h_i}\;|\;i\in I\rbrace$, which is straightforward. Here ends the proof of the claim.

Now, it is easy to define a $G$-action on $M':=\dcl_{\mathcal{L}}^{\mathfrak{D}}(M,(a_{h_i})_{i\in I})$ by the collection $(\tau_g)_{g\in G}$, thus $(M',(\tau_g)_{g\in G})\models (T'_{\forall})_G$ and $(M,(\sigma_g)_{g\in G})\leqslant_1(M',(\tau_g)_{g\in G})$. 

Because $\varphi(y,x)$ is quantifier free and by the choice of $(a_g)_{g\in G}$, we have $\mathfrak{D}\models\varphi(\sigma_g(m),a_g)$ and $M'\models \varphi(\sigma_g(m),a_g)$ for each $g\in G$. In particular for the neutral element, and so $M'\models\exists \;x\varphi(m,x)$. 
Since $(M,(\sigma_g)_{g\in G})\leqslant_1(M',(\tau_g)_{g\in G})$, it follows also $M\models\exists x\;\varphi(m,x)$. 
\end{proof}

\section{If model companion exists}\label{sec:forking}
\subsection{From stability to simplicity}\label{sec:stable.to.simple}
At last in this paper, we will assume that $T_G^{\mc}$ exists.
But before fixing the set-up, we will give very general definitions and recall a well known theorem which will be used afterwards.

\begin{definition}\label{ind_relat}
Suppose that $T$ is an $\mathcal{L}$-theory and $\mathfrak{C}\models T$ is a monster model. A ternary relation $\ind^{\circ}$ on small subsets of $\mathfrak{C}$ will be called \emph{independence relation} if the following are satisfied.
\begin{enumerate}
\item[(i)] (invariance) If $A\ind^{\circ}_E B$, $f\in\aut_{\mathcal{L}}(\mathfrak{C})$, then $f(A)\ind^{\circ}_{f(E)} f(B)$.

\item[(ii)] (local character) For every finite subset $A$ and every small subset $B$ there is a subset $E$ such that $|E|\leqslant|T|$ and $A\ind^{\circ}_E B$.

\item[(iii)] (finite character) $A\ind^{\circ}_E B$ if and only if for every finite tuple $\bar{b}\subseteq B$ it is $A\ind^{\circ}_E \bar{b}$.

\item[(iv)] (symmetry) $A\ind^{\circ}_E B$ if and only if $B\ind^{\circ}_E A$.

\item[(v)] (transitivity) Assume $E_0\subseteq E_1\subseteq E_2$. Then $A\ind^{\circ}_{E_0} E_2$ if and only if $A\ind^{\circ}_{E_0}E_1$ and $A\ind^{\circ}_{E_1} E_2$.

\item[(vi)] (existence) For any $A$, $B$, $E$ there is $f\in\aut_{\mathcal{L}}(\mathfrak{C}/E)$ such that $f(A)\ind^{\circ}_E B$.
\end{enumerate}
\end{definition}

\begin{definition}
Suppose that $T$ is an $\mathcal{L}$-theory and $\ind^{\circ}$ is an independence relation on a monster model $\mathfrak{C}\models T$. We say that $\ind^{\circ}$ satisfies \emph{the Independence Theorem over a model} if the following holds:
\begin{center}
For every small $M\preceq\mathfrak{C}$, small subsets $A,B\subseteq\mathfrak{C}$ and tuples $a,b\subseteq\mathfrak{C}$ \\ such that $A\ind^{\circ}_M B$, $a\ind^{\circ}_MA$, $b\ind^{\circ}_MB$ and $a\equiv_M b$, \\there exists
\\a tuple $c\subseteq\mathfrak{C}$ such that $c\equiv_{MA}a$, $c\equiv_{MB}b$ and $c\ind^{\circ}_MAB$.
\end{center}
\end{definition}

\begin{theorem}[Theorem 4.2 in \cite{KimPi}]\label{simple.thm}
If a theory $T$ admits an independence relation which satisfies the Independence Theorem over a model, then $T$ is simple and this independence relation coincides with the forking independence.
\end{theorem}

The reader may expect that now we will define some ternary relation and prove that it is an independence relation satisfying the Independence Theorem over a model. Indeed, but to do this we need to clarify what is the stage for our objects.
\textbf{For the rest of this section}, we fix the following set-up:
\begin{itemize}
\item $\kappa_{\mathfrak{C}}$ is a big enough cardinal number,
\item $T$ is an $\mathcal{L}$-theory,
\item $T^{\mc}$ exists and has quantifier elimination,
\item $T_G^{\mc}$ exists,
\item $(\mathfrak{C},(\sigma_g)_{g\in G})\models T_G^{\mc}$ is $\kappa_{\mathfrak{C}}$-saturated and $\kappa_{\mathfrak{C}}$-strongly homogeneous,
\item $\kappa_{\mathfrak{D}}>|\mathfrak{C}|$,
\item $\mathfrak{D}$ is a $\kappa_{\mathfrak{D}}$-saturated and $\kappa_{\mathfrak{D}}$-strongly homogeneous model of $T^{\mc}$, which is an $\mathcal{L}$-extension of $\mathfrak{C}$ 
(note that $\Cn\big(T^{\mc}\cup\Dgat_{\mathcal{L}}(\mathfrak{C})\big)$ is complete),

\item $T':=\theo_{\mathcal{L}}(\mathfrak{D})$ has elimination of imaginaries and is stable,

\item a \emph{small subset of $\mathfrak{C}$} is a subset of $\mathfrak{C}$ of cardinality strictly less than $\kappa_{\mathfrak{C}}$ (similarly for substructures and tuples),
\item a \emph{small subset of $\mathfrak{D}$} is a subset of $\mathfrak{D}$ of cardinality strictly less than $\kappa_{\mathfrak{D}}$ (similarly for substructures and tuples).
\end{itemize}

We will use results obtained in the previous section. Therefore we need to explain how we fulfil assumptions from Section \ref{sec:galois}, what we do right now. We have a big saturated $\mathcal{L}$-structure $\mathfrak{D}$, theory of $\mathfrak{D}$, denoted by $T'$, eliminates quantifiers and imaginaries, and is stable.
If $(M,(\sigma_g)_{g\in G})\preceq (\mathfrak{C},(\sigma_g)_{g\in G})$, then 
$(M,(\sigma_g)_{g\in G})\models T_G^{\mc}$ and $(M,(\sigma_g)_{g\in G})$ is existentially closed model of $T_G$. If we come back to the beginning of Subsection \ref{subs:galois}, then we note that only $T_{\forall}\subseteq T'$ needs to be verified. Observe that $T_{\forall}=T^{\mc}_{\forall}\subseteq T^{\mc}\subseteq T'$.

Now, we will describe how the algebraic closures look in the theory $T_G^{\mc}$ and how they correspond to the algebraic closures in the theory $T^{\mc}$, which is important for the proof of the main theorem (Theorem \ref{ind_thm_model}).

\begin{remark}\label{cap_acl}
Note that the quantifier elimination for $T^{\mc}$ implies that for any $A\subseteq\mathfrak{C}$ it follows $\acl_{\mathcal{L}}^{\mathfrak{D}}(A)\cap\mathfrak{C}\subseteq\acl_{\mathcal{L}}^{\mathfrak{C}}(A)$.
\end{remark}

\begin{lemma}\label{ind_G_extensions}
Assume that $(M,\bar{\sigma})\preceq(\mathfrak{C},\bar{\sigma})$. Let $M'$ be an $\mathcal{L}$-substructure of $\mathfrak{D}$ equipped with a $G$-action $(\sigma'_g)_{g\in G}$ and let $E=\acl_{\mathcal{L}^G}^{\mathfrak{C}}(E)\subseteq M\cap M'$ be such that $\sigma_g|_E=\sigma_g'|_E$ for each $g\in G$. If $M\ind^{\mathfrak{D}}_E M'$, then there exists a $G$-action $(\tau_g)_{g\in G}$ on $\langle M\cup M'\rangle_{\mathcal{L}}$ which extends both $(\sigma_g)_{g\in G}$ and $(\sigma'_g)_{g\in G}$.
\end{lemma}

\begin{proof}
If for every $g\in G$ there exists $\tau_g\in\aut_{\mathcal{L}}(\mathfrak{D})$ such that $\tau_g|_M=\sigma_g$ and $\tau_g|_{M'}=\sigma_g'$, then $(\tau_g)_{g\in G}$ acts on $\langle M\cup M'\rangle_{\mathcal{L}}$ as a $G$-action. Hence, it is enough to show that such $\tau_g$ exists for each $g\in G$.

Fix $g\in G$. By Lemma \ref{fhat_lemma}, $\sigma_g$ and $\sigma_g'$ can be viewed as elements of $\aut_{\mathcal{L}}(\mathfrak{D})$. We have $\sigma_g(E)=E$, $\sigma_g(M)=M$ and $\sigma_g'(M')=M'$, so $\sigma_g(M)\ind^{\mathfrak{D}}_{\sigma_g(E)} \sigma_g'(M')$. Before we can use Corollary \ref{regular.PAPA}, we need to verify the regularity assumption: $\dcl_{\mathcal{L}}^{\mathfrak{D}}(M)\cap\acl_{\mathcal{L}}^{\mathfrak{D}}(E)=\dcl_{\mathcal{L}}^{\mathfrak{D}}(E)$. Because $\dcl_{\mathcal{L}}^{\mathfrak{D}}(M)\subseteq\mathfrak{C}$, it follows, by Remark \ref{cap_acl}, that $\dcl_{\mathcal{L}}^{\mathfrak{D}}(M)=M$. Therefore
\begin{IEEEeqnarray*}{rCl}
\dcl_{\mathcal{L}}^{\mathfrak{D}}(M)\cap\acl_{\mathcal{L}}^{\mathfrak{D}}(E) &=& M\cap\acl_{\mathcal{L}}^{\mathfrak{D}}(E)\subseteq\mathfrak{C}\cap\acl_{\mathcal{L}}^{\mathfrak{D}}(E) \\
&\subseteq & \acl_{\mathcal{L}}^{\mathfrak{C}}(E)\subseteq \acl_{\mathcal{L}^G}^{\mathfrak{C}}(E) \\
&=& E\subseteq\dcl_{\mathcal{L}}^{\mathfrak{D}}(E),
\end{IEEEeqnarray*}
where we used, again, Remark \ref{cap_acl}.
\end{proof}

\noindent We give below another version of the above lemma, this time for relatively algebraically closed sets (i.e. $E$ is equal to $\acl_{\mathcal{L}}^{\mathfrak{D}}(G\cdot A)\cap\mathfrak{C}$ instead of being equal to $\acl_{\mathcal{L}^G}^{\mathfrak{C}}(A)$).

\begin{lemma}\label{ind_G_extensions2}
Assume that $(M,\bar{\sigma})\preceq(\mathfrak{C},\bar{\sigma})$. Let $M'$ be an $\mathcal{L}$-substructure of $\mathfrak{D}$ equipped with a $G$-action $(\sigma'_g)_{g\in G}$ and let $E\subseteq M\cap M'$ be such that $\sigma_g|_E=\sigma_g'|_E$ for each $g\in G$ and $E=\acl_{\mathcal{L}}^{\mathfrak{D}}(GA)\cap\mathfrak{C}$ for some $A\subseteq\mathfrak{C}$. If $M\ind^{\mathfrak{D}}_E M'$, then there exists a $G$-action 
$(\tau_g)_{g\in G}$ on $\langle M\cup M'\rangle_{\mathcal{L}}$ which extends $(\sigma_g)_{g\in G}$ and $(\sigma'_g)_{g\in G}$.
\end{lemma}

\begin{proof}
We can repeat the previous proof after checking two things.
First one, it is easy to see that $\sigma_g(E)\subseteq E$ for each $g\in G$.
Second one, we need to verify regularity of $E\subseteq M$:
$$\dcl_{\mathcal{L}}^{\mathfrak{D}}(M)\cap\acl_{\mathcal{L}}^{\mathfrak{D}}\big(\acl_{\mathcal{L}}^{\mathfrak{D}}(GA)\cap\mathfrak{C} \big)\subseteq \mathfrak{C}\cap\acl_{\mathcal{L}}^{\mathfrak{D}}(GA)\subseteq\dcl_{\mathcal{L}}^{\mathfrak{D}}\big(\acl_{\mathcal{L}}^{\mathfrak{D}}(GA)\cap\mathfrak{C} \big).$$
\end{proof}

\noindent
The following several facts generalize similar ones from \cite{acfa1}, \cite{ChaPil} and \cite{nacfa}.

\begin{fact}\label{equiv_fact}
Assume that $(M_1,\bar{\tau_1}),(M_2,\bar{\tau_2})\models T_G^{\mc}$, let $M_1,M_2\subseteq\mathfrak{D}$, $E\subseteq M_1\cap M_2$,
 and set $E_i:=\acl_{\mathcal{L}}^{\mathfrak{D}}(\bigcup_{g\in G}\tau_{i,g}(E))\cap M_i$.
If
\begin{enumerate}
\item[i)] $E_1\subseteq M_1\cap M_2$, and $\bar{\tau_1}$ and $\bar{\tau_2}$ agree on $E_1$,
\end{enumerate}
or
\begin{enumerate}
\item[ii)]$E_2\subseteq M_1\cap M_2$, and $\bar{\tau_1}$ and $\bar{\tau_2}$ agree on $E_2$,
\end{enumerate}
then $(M_1,\bar{\tau_1})\equiv_E (M_2,\bar{\tau_2})$.
\end{fact}

\begin{proof}
It is a standard argument and we will prove only item i).
Consider $(\mathfrak{C}_1,\bar{\tau}_1)\succeq (M_1,\bar{\tau}_1)$ which is a monster model for $T_G^{\mc}$. Without loss of generality, we assume that $\mathfrak{C}_1\subseteq\mathfrak{D}$. Note that for the $G$-action on $\mathfrak{C}_1$ it follows
\begin{IEEEeqnarray*}{rCl}
\acl_{\mathcal{L}}^{\mathfrak{D}}(GE)\cap M_1=E_1 &\subseteq & \acl_{\mathcal{L}}^{\mathfrak{D}}(GE_1)\cap\mathfrak{C}_1 \\
&\subseteq & \acl_{\mathcal{L}}^{\mathfrak{D}}\Big(G\big(\acl_{\mathcal{L}}^{\mathfrak{D}}(GE)\cap M_1\big)\Big)\cap\mathfrak{C}_1\\
&\subseteq & \acl_{\mathcal{L}}^{\mathfrak{D}}\Big(G\big(\acl_{\mathcal{L}}^{\mathfrak{D}}(GE)\big)\Big)\cap\mathfrak{C}_1\\
&\subseteq & \acl_{\mathcal{L}}^{\mathfrak{D}}\big(\acl_{\mathcal{L}}^{\mathfrak{D}}(GE)\big)\cap\mathfrak{C}_1\\
&\subseteq & \acl_{\mathcal{L}}^{\mathfrak{D}}(GE)\cap\mathfrak{C}_1
=\acl_{\mathcal{L}}^{\mathfrak{D}}(GE)\cap M_1=E_1.
\end{IEEEeqnarray*}
Let $f\in\aut_{\mathcal{L}}(\mathfrak{D}/E_1)$ be such that $M_1\ind^{\mathfrak{D}}_{E_1} f(M_2)$ and so also $(M_2^f, \bar{\tau}_2^f)\cong_{E_1}(M_2,\bar{\tau}_2)$. By Lemma \ref{ind_G_extensions2}, both $\mathcal{L}^G$-structures, $(M_1,\bar{\tau}_1)$ and $(M_2^f, \bar{\tau}_2^f)$, can be extended simultaneously and then embedded into a model of $T_G^{\mc}$, say $(N,\bar{\tau})$. Model completeness of $T_G^{\mc}$ implies that
$(M_1,\bar{\tau}_1)\preceq (N,\bar{\tau})$ and $(M_2^f,\bar{\tau}_2^f)\preceq (N,\bar{\tau})$, thus even
$(M_1,\bar{\tau}_1)\equiv_{E_1}(M_2^f,\bar{\tau}_2^f)\cong_{E_1}(M_2,\bar{\tau}_2)$.
\end{proof}

\begin{remark}\label{equiv_remark}
We can always set $E=\emptyset$ in the above fact, hence
the theory of a model of $T_G^{\mc}$ is determined by
the relative algebraic closure of the empty set and by the
action of $G$ on the relative algebraic closure of the empty set.
\end{remark}

\begin{fact}\label{fact_type_descr}
Let $c$, $c'$ be (small) tuples from $\mathfrak{C}$ and let $E\subseteq \mathfrak{C}$ be also small. 
Then
$$\tp_{\mathcal{L}^G}^{\mathfrak{C}}(c/E)=\tp_{\mathcal{L}^G}^{\mathfrak{C}}(c'/E)$$
if and only if there exist small $\mathcal{L}^G$-substructures $C,C'\subseteq\mathfrak{C}$ such that $E\subseteq C\cap C'$, $c\in C$, $c'\in C'$ and $C'=\acl_{\mathcal{L}^G}^{\mathfrak{C}}(C')$, and there exists an $\mathcal{L}^G_E$-isomorphism $f:(C,\bar{\sigma})\to(C',\bar{\sigma})$ sending $c$ to $c'$.
\end{fact}

\begin{proof}
Implication from left to right is obvious if we set $C:=\acl_{\mathcal{L}^G}^{\mathfrak{C}}(Ec)$ and $C':=\acl_{\mathcal{L}^G}^{\mathfrak{C}}(Ec')$, so we skip this part of the proof.

Consider a small elementary $\mathcal{L}^G$-substructure $(M,\bar{\sigma})$ of $(\mathfrak{C},\bar{\sigma})$ which contains $C\cup C'$. 
We assume that $f$ is equal to $\hat{f}$ from Lemma \ref{fhat_lemma}.
There exists $h\in\aut_{\mathcal{L}}(\mathfrak{D}/C')$ such that
$$M\ind_{C'}^{\mathfrak{D}} h(M^f).$$
Note that $\bar{\sigma}$ and $\bar{\sigma}^{hf}$ coincide on $C'$.
By Lemma \ref{ind_G_extensions}, $(\langle M\cup M^{hf}\rangle,\bar{\sigma}\cup\bar{\sigma}^{hf})\models (T_G)_{\forall}$. We embed $(M,\bar{\sigma})$ and $(M^{hf},\bar{\sigma}^{hf})$ into some $(N,\bar{\rho})\models T_G^{\mc}$, and because $(M,\bar{\sigma})\preceq(N,\bar{\rho})$, we can further embed $(N,\bar{\rho})$ in $(\mathfrak{C},\bar{\sigma})$ over $M$. Now, the image of $M^{hf}$ under the composition of these embeddings, say $i(M^{hf})$, is $\mathcal{L}^G$-isomorphic to $i(M)=M$ by the isomorphism $ihfi^{-1}$. However, $(M,\bar{\sigma}),(M^{hf},\bar{\sigma}^{hf})\models T_G^{\mc}$, so $ihfi^{-1}$ extends to an element of $\aut_{\mathcal{L}^G}(\mathfrak{C})$. We note that $ihfi^{-1}(c)=c'$ and $ihfi^{-1}|_E=\id_E$.
\end{proof}

\begin{lemma}\label{acl_lemma}
For each small $A\subseteq\mathfrak{C}$ it follows
$$\acl_{\mathcal{L}^G}^{\mathfrak{C}}(A)=\acl_{\mathcal{L}}^{\mathfrak{C}}(G\cdot A).$$
\end{lemma}

\begin{proof}
Naturally $\acl_{\mathcal{L}}^{\mathfrak{C}}(G\cdot A)\subseteq\acl_{\mathcal{L}^G}^{\mathfrak{C}}(A)$.
We will show that
$\acl_{\mathcal{L}}^{\mathfrak{C}}(G\cdot A)\supseteq\acl_{\mathcal{L}^G}^{\mathfrak{C}}(A)$.

Let $E:=\acl_{\mathcal{L}}^{\mathfrak{C}}(G\cdot A)$, clearly it is
$\acl_{\mathcal{L}^G}^{\mathfrak{C}}(A)=\acl_{\mathcal{L}^G}^{\mathfrak{C}}(E)$. Reductio ad absurdum, suppose that there exists $d\in\acl_{\mathcal{L}^G}^{\mathfrak{C}}(E)\setminus E$. Let $\varphi(x)$ be an $\mathcal{L}^G_E$-formula such that $\varphi^{\mathfrak{C}}(d)$ and $k:=|\varphi(\mathfrak{C})|<\infty$.

Consider a small $\mathcal{L}^G$-substructure $(M,\bar{\sigma})$ of $(\mathfrak{C},\bar{\sigma})$ which is a model of $T_G^{\mc}$ and contains $E$ and all $k$ realizations of $\varphi(x)$. 
There is $f\in\aut(\mathfrak{D}/E)$ such that $M\ind^{\mathfrak{D}}_E f(M)$.
Of course, $f$ is an $\mathcal{L}^G$-isomorphisms over $E$ between $(M,\bar{\sigma})$ and $(M^f,\bar{\sigma}^f)$, hence $(M^f,\bar{\sigma}^f)\models T_G^{\mc}$ and $\varphi^{M^f}(f(d))$.
By Lemma \ref{ind_G_extensions}, the
$G$-action $\bar{\sigma}$ on $M$ and the $G$-action $\bar{\sigma}^f$ on $M^f$ can be simultaneously extended to a $G$-action $\bar{\rho}$ on the $\mathcal{L}$-structure $\langle M\cup M^f\rangle$.

If $f(d)\in M$, then $f(d)\ind^{\mathfrak{D}}_E f(d)$ implies that $f(d)\in\acl_{\mathcal{L}}^{\mathfrak{D}}(E)$ and $d\in\acl_{\mathcal{L}}^{\mathfrak{D}}(E)$; but $d\in\mathfrak{C}$, so it would lead, by Remark \ref{cap_acl}, to $d\in\acl_{\mathcal{L}}^{\mathfrak{C}}(E)=E$, which is impossible.
Therefore $f(d)\not\in M$.

It follows that the $\mathcal{L}^G$-structure $(\langle M\cup M^f\rangle,\bar{\rho})$ is a model of $(T_G)_{\forall}$. Therefore we can embed (as an $\mathcal{L}^G$-substructure) $(\langle M\cup M'\rangle,\bar{\rho})$ into a model of $T_G$ and then into a model of $T_G^{\mc}$, say $(N,\bar{\rho})$.
The theory $T_G^{\mc}$ is model complete, so $(M,\bar{\sigma})\preceq(N,\bar{\rho})$ and
$(M,\bar{\sigma})\equiv(N,\bar{\rho})$. Now we can embed $(N,\bar{\rho})$ into $(\mathfrak{C},\bar{\sigma})$ over $M$ 
as an elementary $\mathcal{L}^G$-substructure.

Because $(M,\bar{\sigma})\preceq(N,\bar{\rho})$ and $(M^f,\bar{\sigma}^f)\preceq(N,\bar{\rho})$, there are at least $k+1$ realizations of the formula $\varphi$ in $N$, hence in $\mathfrak{C}$, which can not happen.
\end{proof}

\begin{lemma}\label{acl_lemma2}
For each small $A\subseteq\mathfrak{C}$ it follows
$$\acl_{\mathcal{L}}^{\mathfrak{C}}(G\cdot A)=
\acl_{\mathcal{L}}^{\mathfrak{D}}(G\cdot A)\cap\mathfrak{C}.$$
\end{lemma}

\begin{proof}
By Remark \ref{cap_acl}, we have $\acl_{\mathcal{L}}^{\mathfrak{D}}(G\cdot A)\cap\mathfrak{C}\subseteq \acl_{\mathcal{L}}^{\mathfrak{C}}(G\cdot A)$. For the proof of the second inclusion, suppose that there is $d\in\acl_{\mathcal{L}}^{\mathfrak{C}}(G\cdot A)\setminus E$, where $E:=\acl_{\mathcal{L}}^{\mathfrak{D}}(G\cdot A)\cap\mathfrak{C}$. By Lemma \ref{acl_lemma}, $d\in\acl_{\mathcal{L}^G}^{\mathfrak{C}}(A)\setminus E$.

We repeat the proof of Lemma \ref{acl_lemma}, for $E=\acl_{\mathcal{L}}^{\mathfrak{D}}(G\cdot A)\cap\mathfrak{C}$, but instead of using Lemma \ref{ind_G_extensions} we use Lemma \ref{ind_G_extensions2}, and, moreover, we do not need to use Remark \ref{cap_acl} anymore in the proof.
\end{proof}

\begin{cor}\label{acl_all}
For each small $A\subseteq\mathfrak{C}$ it follows
$$\acl_{\mathcal{L}^G}^{\mathfrak{C}}(A)=\acl_{\mathcal{L}}^{\mathfrak{C}}(G\cdot A)=
\acl_{\mathcal{L}}^{\mathfrak{D}}(G\cdot A)\cap\mathfrak{C}.$$
\end{cor}

\begin{proof}
By Lemma \ref{acl_lemma} and Lemma \ref{acl_lemma2}.
\end{proof}

\begin{remark}\label{almost.qe}
Now, it is not unexpected that $T_G^{\mc}$ allows quantifier elimination up to some existential formulas, similarly as ACFA. The proofs for the situation described in \cite[Paragraph 1.6]{acfa1} remain the same in our context, so we just provide those existential formulas.

Every $\mathcal{L}^G$-formula $\varphi(\bar{x})$ is equivalent modulo $T_G^{\mc}$ to an $\mathcal{L}^G$-formula $\exists\bar{y}\;\psi(\bar{y},\bar{x})$ such that $\psi(\bar{y},\bar{x})$ is quantifier free and $\psi(\bar{a},\bar{b})$ implies that
$\bar{a}\subseteq\acl_{\mathcal{L}}^{\mathfrak{C}}(G\cdot\bar{b})$.
\end{remark}

\noindent
For some time we expected that the possibility of defining a $G$-action on ``roots" of some algebraic $\mathcal{L}$-formula should not distinct these ``roots" (e.g. if $f(x)$ is a polynomial over a field $K$ with a $G$-action such that $f(x)$ has a root in $K$, we can ask whether we can extend the $G$-action on 
on a field extension of $K$ containing all the roots of $f$), but can not prove it. Subsection 3.2 in \cite{nacfa} describes algebraic extension in the case of $G-\tcf$ and do not give any clue about
extending the $G$-action on the remaining ``roots", which would led to normality of the relative algebraic closure for fields with an action of a finite group. The following Proposition provides positive answer in a greater (than only the fields case) generality, to the question: whether $A\subseteq\acl_{\mathcal{L}^G}^{\mathfrak{C}}(A)$ is a normal extension in the sense of $\mathcal{L}$-structure $\mathfrak{D}$?

\begin{prop}\label{A_Galois}
Let $A$ be a small subset of $\mathfrak{C}$ and let $A':=\dcl_{\mathcal{L}^G}^{\mathfrak{C}}(A)$. Then $A'\subseteq\acl_{\mathcal{L}}^{\mathfrak{D}}(A')\cap\mathfrak{C}$ is a Galois extension in the sense of the structure $\mathfrak{D}$ (Definition \ref{galois.ext.def}), in particular:
$$\aut_{\mathcal{L}}(\mathfrak{D}/A')\cdot\big(\acl_{\mathcal{L}}^{\mathfrak{D}}(A')\cap\mathfrak{C} \big)\subseteq \acl_{\mathcal{L}}^{\mathfrak{D}}(A')\cap\mathfrak{C}$$
and $dcl_{\mathcal{L}}^{\mathfrak{D}}(A')=A'$.
\end{prop}

\begin{proof}
Note that, by Remark \ref{dcl_M},
$$\dcl_{\mathcal{L}}^{\mathfrak{D}}(\acl_{\mathcal{L}}^{\mathfrak{D}}(A')\cap\mathfrak{C})\subseteq \acl_{\mathcal{L}}^{\mathfrak{D}}(A')\cap\mathfrak{C},$$
hence the set $\acl_{\mathcal{L}}^{\mathfrak{D}}(A')\cap\mathfrak{C}$
is definably closed in the sense of $\mathfrak{D}$. Moreover, by Corollary \ref{acl_all},
$\acl_{\mathcal{L}^G}^{\mathfrak{C}}(A')=\acl_{\mathcal{L}}^{\mathfrak{D}}(A')\cap\mathfrak{C}\subseteq\acl_{\mathcal{L}}^{\mathfrak{D}}(A')$.
Consider the following profinite group
$$H:=\aut_{\mathcal{L}^G}(\acl_{\mathcal{L}^G}^{\mathfrak{C}}(A')/A').$$
Since $A'$ is definably closed in the sense of $(\mathfrak{C},(\sigma_g)_{g\in G})$, we notice that $\acl_{\mathcal{L}^G}^{\mathfrak{C}}(A')^H=A'$
(if $a\in\acl_{\mathcal{L}^G}^{\mathfrak{C}}(A')^H$ and $f\in\aut_{\mathcal{L}^G}(\mathfrak{C}/A')$, then $f|_{\acl_{\mathcal{L}^G}^{\mathfrak{C}}(A')}\in H$ and so $f(a)=a$, thus $a\in\dcl_{\mathcal{L}^G}^{\mathfrak{C}}(A')=A'$).
By Lemma \ref{N_Galois} for $N=\acl_{\mathcal{L}^G}^{\mathfrak{C}}(A')\subseteq \acl_{\mathcal{L}}^{\mathfrak{D}}(A')$ and the action of the group $H$, it follows that
$$A'=N^H\subseteq N=\acl_{\mathcal{L}^G}^{\mathfrak{C}}(A')=\acl_{\mathcal{L}}^{\mathfrak{D}}(A')\cap\mathfrak{C}$$
is a Galois extension.
\end{proof}

We introduce now a ternary relation on all small subsets of $\mathfrak{C}$. 
Let $A,B,E$ be small subsets of $\mathfrak{C}$ (i.e. of cardinality strictly less than $\kappa_{\mathfrak{C}}$), we define the following
$$A\ind^{\circ}_E B\quad\iff\quad G\cdot A\ind^{\mathfrak{D}}_{G\cdot E}G\cdot B.$$
It only seemingly depends on the choice of $\mathfrak{D}$, i.e. if $\mathfrak{D}'\models T'$, $\mathfrak{C}\subseteq\mathfrak{D}'$, and $\mathfrak{D}'$ is $|\mathfrak{C}|^{+}$-saturated, then
$$G\cdot A\ind^{\mathfrak{D}}_{G\cdot E}G\cdot B\quad\iff\quad G\cdot A\ind^{\mathfrak{D}'}_{G\cdot E}G\cdot B.$$

\begin{remark}\label{ind_acl_rem}
It follows that
$$A\ind^{\circ}_E B\quad\iff\quad A\ind^{\circ}_{\acl_{\mathcal{L}^G}^{\mathfrak{C}}(E)} B\quad\iff\quad \acl_{\mathcal{L}^G}^{\mathfrak{C}}(A)\ind^{\circ}_{\acl_{\mathcal{L}^G}^{\mathfrak{C}}(E)}\acl_{\mathcal{L}^G}^{\mathfrak{C}}(B).$$
\end{remark}

\begin{proof}
We use standard properties of the forking independence in stable theories and Corollary \ref{acl_all}.
\end{proof}

\begin{prop}\label{6properties}
The relation $\ind^{\circ}$ is an independence relation.
\end{prop}

\begin{proof}
We need to check the items (i)-(vi) from Definition \ref{ind_relat}.
\vspace{4mm}
\\
\underline{(i) invariance}\hspace{2mm} It follows from Lemma \ref{fhat_lemma}. Let $A\ind^{\circ}_E B$ and $f\in\aut_{\mathcal{L}^G}(\mathfrak{C})$. Then there exists $\hat{f}\in\aut_{\mathcal{L}}(\mathfrak{D})$ such that $\hat{f}|_{\mathfrak{C}}=f$, hence
$$G\cdot A\ind^{\mathfrak{D}}_{G\cdot E} G\cdot B\;\Rightarrow\;\hat{f}(G\cdot A)\ind^{\mathfrak{D}}_{\hat{f}(G\cdot E)} \hat{f}(G\cdot B)\;\Rightarrow
f(G\cdot A)\ind^{\mathfrak{D}}_{f(G\cdot E)} f(G\cdot B)\;\Rightarrow$$
$$G\cdot f(A)\ind^{\mathfrak{D}}_{G\cdot f(E)} G\cdot f(B).$$
\vspace{4mm}
\\
\underline{(ii) local character}\hspace{2mm} Assume that $A$ is a finite subset of $\mathfrak{C}$ and $B$ is a small subset of $\mathfrak{C}$. There exists some $E'\subseteq G\cdot B$, $|E'|\leqslant|T^{\mc}|$, such that $G\cdot A\ind^{\mathfrak{D}}_{E'}G\cdot B$. Therefore we can choose $E\subseteq B$ satisfying $|E|\leqslant|E'|$ and $E'\subseteq G\cdot E$. By the transitivity of $\ind^{\mathfrak{D}}$, we obtain that
$$G\cdot A\ind^{\mathfrak{D}}_{E'}G\cdot B\;\Rightarrow G\cdot A\ind^{\mathfrak{D}}_{G\cdot E} G\cdot B.$$
Naturally $|E|\leqslant |E'|\leqslant |T^{\mc}|\leqslant |T_G^{\mc}|$.
(Note that, in fact, the size of $E$ does not depend on the size of $G$.)
\vspace{4mm}
\\
\underline{(iii)-(v)}\hspace{2mm} These items are easy to verify.
\vspace{4mm}
\\
\underline{(vi) existence}\hspace{2mm} Let $A$, $B$ and $E$ be small subsets of $\mathfrak{C}$. Our aim is to find $f_0\in\aut_{\mathcal{L}^G}(\mathfrak{C}/E)$ such that $f_0(A)\ind^{\circ}_E B$. Take a small $(M,\bar{\sigma})\preceq(\mathfrak{C},\bar{\sigma})$ which contains $A$, $B$ and $E$. We introduce $E':=\acl_{\mathcal{L}^G}^{\mathfrak{C}}(E)$.

There exists $f\in\aut_{\mathcal{L}}(\mathfrak{D}/E')$ such that $f(M)\ind^{\mathfrak{D}}_{E'} M$. By Lemma \ref{ind_G_extensions}, $\bar{\sigma}^f$ and $\bar{\sigma}$ extend simultaneously to $\langle M^f\cup M\rangle$. We have
$(\langle M^f\cup M\rangle,\bar{\sigma}^f\cup\bar{\sigma})\models (T_G)_{\forall}$, hence there exists $(N,\bar{\rho})\models T_G^{\mc}$ and 
we can embed $(N,\bar{\rho})$ into $(\mathfrak{C},\bar{\sigma})$ over $M$, say $i:(N,\bar{\rho})\to(\mathfrak{C},\bar{\sigma})$. 

The image of $f(M)\subseteq N$ under this embedding, $if(M)$, will be $\mathcal{L}^G$-isomorphic over $E'$ to $f(M)$, hence also $\mathcal{L}^G$-isomorphic over $E'$ to $M$. Since $T_G^{\mc}$ is model complete and
$(M^{if},\bar{\sigma}^{if})\cong (M,\bar{\sigma})\models T_G^{\mc}$, it follows $(M^{if},\bar{\sigma}^{if})\preceq(\mathfrak{C},\bar{\sigma})$. Therefore homogeneity of $(\mathfrak{C},\bar{\sigma})$ assure us that there is $\tilde{f}\in\aut_{\mathcal{L}^G}(\mathfrak{C}/E')$ such that $\tilde{f}|_M=if|_M$.
By Lemma \ref{fhat_lemma}, there exists $h\in\aut_{\mathcal{L}}(\mathfrak{D}/E')$ such that $h|_{\mathfrak{C}}=\tilde{f}$.

We claim that $h(M)\ind^{\mathfrak{D}}_{E'}M$.
Suppose that, contrary to our claim, for some quantifier free $\mathcal{L}_{E'}$-formula $\varphi(x,y)$, $m,m_0\in M$, a collection $\lbrace f_i\rbrace_{i<\omega}$ of elements of $\aut_{\mathcal{L}}(\mathfrak{D}/E')$ and some $k<\omega$ we have $\varphi^{\mathfrak{D}}(h(m),m_0)$ and the set
$$\lbrace \varphi(x,f_i(m_0))\;|\;i,\omega\rbrace$$
is $k$-inconsistent. We will prove the claim if we show that $\varphi^{\mathfrak{D}}(f(m),m_0)$ (which will contradict $f(M)\ind^{\mathfrak{D}}_{E'}M$). 
We use that $i|_M=\id_M$,
\begin{IEEEeqnarray*}{rCl}
\varphi^{\mathfrak{D}}(h(m),m_0) &\iff & \varphi^{\mathfrak{D}}(if(m),m_0) \\
&\iff & \varphi^{\mathfrak{D}}(if(m),i(m_0)) \\
&\iff & \varphi^{N}(f(m),m_0) \\
&\iff & \varphi^{\langle f(M),M\rangle_{\mathcal{L}}}(f(m),m_0) \\
&\iff & \varphi^{\mathfrak{D}}(f(m),m_0)
\end{IEEEeqnarray*}

Since $\tilde{f}$ is an element of $\aut_{\mathcal{L}^G}(\mathfrak{C}/E)$
and $G\cdot\acl_{\mathcal{L}^G}^{\mathfrak{C}}(E)=\acl_{\mathcal{L}^G}^{\mathfrak{C}}(E)=E'$, we get 
$$G\cdot \tilde{f}(M)\ind^{\mathfrak{D}}_{G\cdot \acl_{\mathcal{L}^G}^{\mathfrak{C}}(E)} G\cdot M.$$
Therefore, by Remark \ref{ind_acl_rem}, $\tilde{f}(M)\ind^{\circ}_E M$ and so, by the previously proved finite character of $\ind^{\circ}$, it follows $\tilde{f}(A)\ind^{\circ}_E B$.

\end{proof}

Before proving the Independence Theorem, we need to introduce two more notions about heirs/coheirs and about algebraic closures.

\begin{definition}
Let $A\subseteq B\subseteq\mathfrak{D}$, $p\in S_{\mathcal{L}}^{\mathfrak{D}}(A)$,
$q\in S_{\mathcal{L}}^{\mathfrak{D}}(B)$, $p\subseteq q$.
\begin{enumerate}
\item We say that $q$ is \emph{heir} of $p$ if for any $\mathcal{L}(A)$-formula $\varphi(x,y)$, existence of $b\in B$ such that $\varphi(x,b)\in q$ implies existence of $a\in A$
such that $\varphi(x,a)\in p$.

\item We say that $q$ is \emph{coheir} of $p$ if for any $\mathcal{L}(A)$-formula $\varphi(x,y)$, existence of $b\in B$ such that $\varphi(x,b)\in q$ implies existence of $a\in A$
such that $\mathfrak{D}\models\varphi(a,b)$.
\end{enumerate}
\end{definition}
\noindent
The above definition brings back common definitions of heir and coheir (e.g. \cite[Definition 8.1.1]{tentzieg}), but we make it over arbitrary sets instead of models. The following lemma is different than the usual one, because $M$ is not necessary a model of our stable theory $T^{\mc}$.

\begin{lemma}\label{dnf_heir}
Assume that $(M,\bar{\sigma})\preceq(\mathfrak{C},\bar{\sigma})$.
If $M\subseteq B\subseteq\mathfrak{C}$ and $p\in S_{\mathcal{L}}^{\mathfrak{D}}(B)$ does not fork over $M$, then $p$ is the heir of $p|M$. 
\end{lemma}

\begin{proof}
We start with an $\mathcal{L}$-formula $\varphi(x,y)$ and $b\in B$ such that $\varphi(x,b)\in p$.
Our aim is to prove existence of an element $b'\in M$ such that $\varphi(x,b')\in p|_M$.
We want to repeat a part of the proof of \cite[Lemma 8.3.5.(2)]{tentzieg}, more precisely: the last paragraph of it.

To do this, we need only to show that $\tp_{\mathcal{L}}^{\mathfrak{D}}(b/M)$ is finitely satisfiable in $M$, which will imply existence of a global coheir extension of $\tp_{\mathcal{L}}^{\mathfrak{D}}(b/M)$ 
(a standard reasoning, see e.g. the argument in the proof of  \cite[Lemma 8.1.3]{tentzieg}) - the rest of our proof 
goes in the same way as in the last paragraph of the proof of \cite[Lemma 8.3.5.(2)]{tentzieg}.

Let
$\psi(y,m)\in\tp_{\mathcal{L}}^{\mathfrak{D}}(b/M)$, i.e.
 $\mathfrak{D}\models\psi(b,m)$. There is a quantifier free $\mathcal{L}$-formula $\psi_0$ equivalent to $\psi$ modulo $T^{\mc}$. It follows that $\mathfrak{D}\models\psi_0(b,m)$, and because $b,m\in\mathfrak{C}$, we have $\mathfrak{C}\models\psi_0(b,m)$. Therefore 
$$\mathfrak{C}\models (\exists y)\big(\psi_0(y,m)\big),$$
which, by $(M,\bar{\sigma})\preceq(\mathfrak{C},\bar{\sigma})$ implies
$$M\models (\exists y)\big(\psi_0(y,m)\big).$$
Take $b_0\in M$ such that $M\models\psi_0(b_0,m)$, it follows that $\mathfrak{D}\models\psi_0(b_0,m)$, hence also $\mathfrak{D}\models\psi(b_0,m)$.
\end{proof}

\begin{definition}\label{alg.cls.def}
Assume that $N$ is an $\mathcal{L}$-substructure of an $\mathcal{L}$-structure $N'$.
We say that \emph{algebraic closures in $N'$ split over $N$}, if for every $M\preceq N$ and every $A\subseteq N$ which contains $M$ it follows
$$\acl_{\mathcal{L}}^{N'}(A)=\dcl_{\mathcal{L}}^{N'}(\acl_{\mathcal{L}}^{N'}(A)\cap N,\acl_{\mathcal{L}}^{N'}(M)).$$
\end{definition}

We will be interested in the case when $N'=\mathfrak{D}$ and $N=\mathfrak{C}$, so we will ask whether
for every $(M,\bar{\sigma})\preceq (\mathfrak{C},\bar{\sigma})$ and every $A\subseteq\mathfrak{C}$ which contains $M$ it follows
$$\acl_{\mathcal{L}}^{\mathfrak{D}}(A)=\dcl_{\mathcal{L}}^{\mathfrak{D}}(\acl_{\mathcal{L}}^{\mathfrak{D}}(A)\cap\mathfrak{C},\acl_{\mathcal{L}}^{\mathfrak{D}}(M)).$$
In a special case when $A=\acl_{\mathcal{L}^G}^{\mathfrak{C}}(A)$, the last line turns out to be simpler:
$$\acl_{\mathcal{L}}^{\mathfrak{D}}(A)=\dcl_{\mathcal{L}}^{\mathfrak{D}}(A,\acl_{\mathcal{L}}^{\mathfrak{D}}(M)).$$

\begin{remark}
\begin{enumerate}
\item
If $\mathfrak{C}=\mathfrak{D}$, then algebraic closures split. Such situation occurs if we can extend the $G$-action from $\mathfrak{C}$ to $\mathfrak{D}$, which is true for free groups (ACFA, \cite{sjogren}, \cite{ChaPil}) and in some other cases (like $G=\mathbb{Q}$ and the theory $\mathbb{Q}$ACFA \cite{qacfa}).

\item
The case of finite groups is rather opposite to the cases considered in the item 1. above.
However (which is rather unexpected), if $G$ is finite, then algebraic closures splits as well (Corollary \ref{G_finite_alg.cls}). An example 
of a corresponding theory is given in Example \ref{gtcf1}.

\item A reader interested in this concept, which is related to the boundedness, should consult Lemma 3.8 in \cite{Polkowska}, Proposition 2.5.(ii) in \cite{PilPol} and our Proposition \ref{bounded.split}.
\end{enumerate}
\end{remark}

Now, after all the preparations above, we can prove the main theorem of the thesis.

\begin{theorem}[The Independence Theorem over a model]\label{ind_thm_model}
Assume that algebraic closures in $\mathfrak{D}$ split over $\mathfrak{C}$.
Let $(M,\bar{\sigma})\preceq(\mathfrak{C},\bar{\sigma})$ and let $p_1(x_1)$, $p_2(x_2)$, $p_3(x_3)$, $p_{12}(x_1,x_2)$, $p_{23}(x_2,x_3)$ and $p_{13}(x_1,x_3)$ be complete $\mathcal{L}^G$-types over $M$ which satisfy $p_i(x_i),p_j(x_j)\subseteq p_{ij}(x_i,x_j)$ and if $a_ia_j\models p_{ij}(x_i,x_j)$ then
$$a_j\ind_M^{\circ} a_i.$$
There exists a complete $\mathcal{L}^G$-type $p_{123}(x_1,x_2,x_3)$ which extends each $p_{ij}(x_i,x_j)$ and such that if $a_1a_2a_3\models p_{123}(x_1,x_2,x_3)$ then
$$a_3\ind_M^{\circ}a_1a_2.$$
\end{theorem}

\begin{proof}
Our proof mimics proofs of similar facts from \cite{acfa1}, \cite{ChaPil} and primarily \cite{markernotes}. For some reasons, we needed to glue together different arguments from all of those proofs and add a new reasoning and new ideas which deal with phenomena coming from the fact that $\mathfrak{C}$ may be not a model of $T^{\mc}$.

Choose elements $a,b,c_1,c_2$ such that
$ab\models p_{12}$, $ac_1\models p_{13}$ and $bc_2\models p_{23}$.
There exists $f_c\in\aut_{\mathcal{L}^G}(\mathfrak{C}/Ma)$ such that for $c_1':=f(c_1)$ it follows that $c_1'\ind_{Ma}^{\circ}b$.
Using $c_1\ind_M^{\circ} a$, invariance and transitivity of $\ind^{\circ}$, we conclude
$c_1'\ind_M^{\circ}ab$.
Because $ac_1\equiv_M^{\mathcal{L}^G}ac_1'$, we can assume without loss of generality that we already have $c_1\ind_M^{\circ}ab$.

Introduce the following $\mathcal{L}^G$-substructures of $\mathfrak{C}$
$$A:=\acl_{\mathcal{L}^G}^{\mathfrak{C}}(Ma),\;\;
B:=\acl_{\mathcal{L}^G}^{\mathfrak{C}}(Mb),\;\; C_1:=\acl_{\mathcal{L}^G}^{\mathfrak{C}}(Mc_1),$$
$$C_2:=\acl_{\mathcal{L}^G}^{\mathfrak{C}}(Mc_2),\;\;
D:=\langle\acl_{\mathcal{L}^G}^{\mathfrak{C}}(Mab),\acl_{\mathcal{L}^G}^{\mathfrak{C}}(Mac_1)\rangle.$$
Because $c_1,c_2\models p_3$, there exists $f_0:(C_1,\bar{\sigma})\to(C_2,\bar{\sigma})$, an $\mathcal{L}^G$-isomorphism over $M$ such that $f_0(c_1)=c_2$. Moreover, $B\ind_M^{\mathfrak{D}}C_1$ and $B\ind_M^{\mathfrak{D}}C_2$, and $B$ is $\mathcal{L}$-regular over $M$ (even $\mathfrak{C}$ is $\mathcal{L}$-regular over $M$), so by Corollary \ref{regular.PAPA} and Lemma \ref{fhat_lemma}, $\id:B\to B$ and $f_0:C_1\to C_2$ extend to $f\in\aut_{\mathcal{L}}(\mathfrak{D})$. We note that
$$f:(\langle B,C_1\rangle,\bar{\sigma})\to (\langle B,C_2\rangle,\bar{\sigma})$$
is an $\mathcal{L}^G$-isomorphisms over $B$.
Let $C:=f^{-1}(\acl_{\mathcal{L}^G}^{\mathfrak{C}}(BC_2))$, by Corollary \ref{acl_all} it follows that
$\langle B,C_1\rangle\subseteq C\subseteq \acl_{\mathcal{L}}^{\mathfrak{D}}(BC_1)$ and
$$f:(C,\bar{\sigma}^{f^{-1}})\to (\acl_{\mathcal{L}^G}^{\mathfrak{C}}(Mbc_2),\bar{\sigma})$$
is an $\mathcal{L}^G$-isomorphisms over $B$.
Note that $\bar{\sigma}^{f^{-1}}$ coincides with $\bar{\sigma}$ on $\langle B, C_1\rangle$. We are going now to prove the following.
\
\\
\\
\textbf{Claim:} $D$ is $\mathcal{L}$-regular over $\langle B, C_1\rangle$.
\\
Proof of the claim: We want to show that
$$\dcl_{\mathcal{L}}^{\mathfrak{D}}\big(\acl_{\mathcal{L}^G}^{\mathfrak{C}}(AB),\acl_{\mathcal{L}^G}^{\mathfrak{C}}(AC_1)\big)\,\cap\,\acl_{\mathcal{L}}^{\mathfrak{D}}(BC_1)\subseteq \dcl_{\mathcal{L}}^{\mathfrak{D}}(BC_1).$$
We will show the above inclusion following the main steps of the proof of the claim from the proof of \cite[Theorem 3.7]{ChaPil}. Let us take an element 
$$\alpha\in\dcl_{\mathcal{L}}^{\mathfrak{D}}\big(\acl_{\mathcal{L}}^{\mathfrak{C}}(AB),\acl_{\mathcal{L}}^{\mathfrak{C}}(AC_1)\big)\,\cap\,\acl_{\mathcal{L}}^{\mathfrak{D}}(BC_1)$$
(recall that, by Corollary \ref{acl_all}, $\acl_{\mathcal{L}^G}^{\mathfrak{C}}(F)=\acl_{\mathcal{L}}^{\mathfrak{C}}(F)$ for any $G$-invariant set $F$)
and consider elements $\beta\in\acl_{\mathcal{L}}^{\mathfrak{C}}(AB)$, $\gamma\in
\acl_{\mathcal{L}}^{\mathfrak{C}}(AC_1)$ and an 
$\mathcal{L}$-formula $\psi_{\alpha}$ such that $\psi_{\alpha}(x,\beta,\gamma)$ defines $\lbrace\alpha\rbrace$ in $\mathfrak{D}$.

Now, we use Corollary \ref{acl_all} and obtain that $\beta\in\acl_{\mathcal{L}}^{\mathfrak{D}}(AB)$, $\gamma\in
\acl_{\mathcal{L}}^{\mathfrak{D}}(AC_1)$, so we can choose  
$\mathcal{L}$-formulas $\psi_{\beta}$ and $\psi_{\gamma}$ and elements 
$\tilde{a}\in A$, $\tilde{b}\in B$ and $\tilde{c}\in C_1$ such that
$\tp_{\mathcal{L}}^{\mathfrak{D}}(\beta/\tilde{a}\tilde{b})$ and
$\tp_{\mathcal{L}}^{\mathfrak{D}}(\gamma/\tilde{a}\tilde{c})$ are algebraic,
$\psi_{\beta}(y,\tilde{a},\tilde{b})$ isolates $\tp_{\mathcal{L}}^{\mathfrak{D}}(\beta/\tilde{a}\tilde{b})$ and $\psi_{\gamma}(z,\tilde{a},\tilde{c})$ isolates
$\tp_{\mathcal{L}}^{\mathfrak{D}}(\gamma/\tilde{a}\tilde{c})$.

We introduce new $\mathcal{L}$-formulas which will code definability and algebraicity:
\begin{IEEEeqnarray*}{rCl}
\psi_{\alpha}'(x,y,z) &\text{ given by }& \;\psi_{\alpha}(x,y,z)\wedge(\forall x_1,x_2)(\bigwedge\limits_{i\leqslant 2}\psi_{\alpha}(x_i,y,z)\rightarrow x_1=x_2),\\
\psi_{\beta}'(y,v,w) &\text{ given by }& \;\psi_{\beta}(y,v,w)\wedge(\forall y_1,\ldots,y_{n_{\beta}})(\bigwedge\limits_{i\leqslant n_{\beta}}\psi_{\beta}(y_i,v,w)\rightarrow \bigvee\limits_{i\neq j}y_i=y_j),\\
\psi_{\gamma}'(z,v,w') &\text{ given by }& \;\psi_{\gamma}(z,v,w')\wedge(\forall z_1,\ldots,z_{n_{\gamma}})(\bigwedge\limits_{i\leqslant n_{\gamma}}\psi_{\gamma}(z_i,v,w')\rightarrow \bigvee\limits_{i\neq j}z_i=z_j),
\end{IEEEeqnarray*}
where $n_{\beta}:=|\psi_{\beta}(\mathfrak{D},\tilde{a},\tilde{b})|$ and $n_{\gamma}:=|\psi_{\gamma}(\mathfrak{D},\tilde{a},\tilde{c})|$. Obviously
$$\mathfrak{D}\models \psi_{\beta}'(\beta,\tilde{a},\tilde{b})\,\wedge\,\psi_{\gamma}'(\gamma,\tilde{a},\tilde{c})\,\wedge\,\psi_{\alpha}'(\alpha,\beta,\gamma),$$
hence we get
$$\mathfrak{D}\models(\exists y,z)\big(\psi_{\beta}'(y,\tilde{a},\tilde{b})\,\wedge\,\psi_{\gamma}'(z,\tilde{a},\tilde{c})\,\wedge\,\psi_{\alpha}'(\alpha,y,z)\big).$$
We are going now to show that $\tp_{\mathcal{L}}^{\mathfrak{D}}(\tilde{b}\tilde{c}\alpha/M\tilde{a})$ is a heir of $\tp_{\mathcal{L}}^{\mathfrak{D}}(\tilde{b}\tilde{c}\alpha/M)$, so we will be able to replace $\tilde{a}$ with some tuple from $M$.

Because of  $C_1\ind_M^{\mathfrak{D}}BA$, it follows $C_1\ind_B^{\mathfrak{D}}A$, hence also $A\ind_B^{\mathfrak{D}}BC_1$. By $A\ind_M^{\mathfrak{D}}B$, we get $A\ind_M^{\mathfrak{D}}BC_1$ and so $BC_1\ind_M^{\mathfrak{D}}A$.
Note that $\tilde{b}\tilde{c}\alpha\subseteq\acl_{\mathcal{L}}^{\mathfrak{D}}(BC_1)$, therefore $\tilde{b}\tilde{c}\alpha\ind_M^{\mathfrak{D}}\tilde{a}$ and, by Lemma \ref{dnf_heir}, $\tp_{\mathcal{L}}^{\mathfrak{D}}(\tilde{b}\tilde{c}\alpha/M\tilde{a})$ is a heir of $\tp_{\mathcal{L}}^{\mathfrak{D}}(\tilde{b}\tilde{c}\alpha/M)$.

There exists $\tilde{m}\in M$ such that 
$$\mathfrak{D}\models(\exists y,z)\big(\psi_{\beta}'(y,\tilde{m},\tilde{b})\,\wedge\,\psi_{\gamma}'(z,\tilde{m},\tilde{c})\,\wedge\,\psi_{\alpha}'(\alpha,y,z)\big).$$
It means that there are $\beta',\gamma'\in\mathfrak{D}$ such that
$$\mathfrak{D}\models \psi_{\beta}'(\beta',\tilde{m},\tilde{b})\,\wedge\,\psi_{\gamma}'(\gamma',\tilde{m},\tilde{c})\,\wedge\,\psi_{\alpha}'(\alpha,\beta',\gamma'),$$
so, by the definitions of $\psi_{\beta}'$ and $\psi_{\gamma}'$, 
$$\beta'\in\acl_{\mathcal{L}}^{\mathfrak{D}}(\tilde{m},\tilde{b}),\quad\gamma'\in\acl_{\mathcal{L}}^{\mathfrak{D}}(\tilde{m},\tilde{c}),\quad\alpha\in\dcl_{\mathcal{L}}^{\mathfrak{D}}(\beta',\gamma').$$

Let us summarize the proof of the claim up to this point. We obtained the following:
$$\alpha\in\dcl_{\mathcal{L}}^{\mathfrak{D}}\big(\acl_{\mathcal{L}}^{\mathfrak{D}}(B),\acl_{\mathcal{L}}^{\mathfrak{D}}(C_1)\big)\cap\acl_{\mathcal{L}}^{\mathfrak{C}}(BC_1),$$
and we wish to show that $\alpha\in\dcl_{\mathcal{L}}^{\mathfrak{D}}(BC_1)$.
Let $f\in\aut_{\mathcal{L}}(\mathfrak{D}/BC_1)$ and note that, that the assumption that algebraic closures split, implies
\begin{IEEEeqnarray*}{rCl}
\alpha\in\dcl_{\mathcal{L}}^{\mathfrak{D}}\big(\acl_{\mathcal{L}}^{\mathfrak{D}}(B),\acl_{\mathcal{L}}^{\mathfrak{D}}(C_1)\big) &=& \dcl_{\mathcal{L}}^{\mathfrak{D}}\Big(
\dcl_{\mathcal{L}}^{\mathfrak{D}}\big(B,\acl_{\mathcal{L}}^{\mathfrak{D}}(M)\big),\dcl_{\mathcal{L}}^{\mathfrak{D}}\big(C_1,\acl_{\mathcal{L}}^{\mathfrak{D}}(M)\big)\Big)\\
&=& \dcl_{\mathcal{L}}^{\mathfrak{D}}\big(BC_1,\acl_{\mathcal{L}}^{\mathfrak{D}}(M)\big).
\end{IEEEeqnarray*}
We see that $M\subseteq\acl_{\mathcal{L}}^{\mathfrak{C}}(BC_1)$ is regular, hence,
by Fact \ref{regular}, there exists $h\in\aut_{\mathcal{L}}(\mathfrak{D}/\acl_{\mathcal{L}}^{\mathfrak{D}}(M))$ such that
$h|_{\acl_{\mathcal{L}}^{\mathfrak{C}}(BC_1)}=f|_{\acl_{\mathcal{L}}^{\mathfrak{C}}(BC_1)}$. In fact, because $f|_{BC_1}=\id_{BC_1}$, we have $h\in\aut_{\mathcal{L}}(\mathfrak{D}/BC_1\acl_{\mathcal{L}}^{\mathfrak{D}}(M))$.
Because $\alpha\in\acl_{\mathcal{L}}^{\mathfrak{C}}(BC_1)$, we can compute
$$f(\alpha)=h(\alpha)=\alpha.$$
Here ends the proof of the claim.

By Lemma \ref{fhat_lemma}, each automorphism $f^{-1}\circ\sigma_g\circ f$ of $C$, where $g\in G$, extends to an automorphism of $\acl_{\mathcal{L}}^{\mathfrak{D}}(BC_1)$, say $h_g$.
We can use the claim and Fact \ref{regular} to extend the $G$-action $\bar{\sigma}$ on $D$ and the collection $(h_g)_{g\in G}$ of automorphisms of $\acl_{\mathcal{L}}^{\mathfrak{D}}(BC_1)$. For each $g\in G$, we denote by $\tau_g\in\aut_{\mathcal{L}}(\mathfrak{D})$ an automorphism such that $\tau_g|_D=\sigma_g$ and $\tau_g|_{\acl_{\mathcal{L}}^{\mathfrak{D}}(BC_1)}=h_g$.

Note that $\bar{\tau}:=(\tau_g)_{g\in G}$ defines a $G$-action on $\langle D, C\rangle$ which extends $G$-actions of the structures $(D,\bar{\sigma})$ and $(C,\bar{\sigma}^{f^{-1}})$, and therefore the $\mathcal{L}^G$-structure $(\langle D,C\rangle,\bar{\tau})$ is a model of $(T_G)_{\forall}$ and embeds into a model of $T_G^{\mc}$, say $(N,\bar{\rho})$. Obviously $(M,\bar{\sigma})\preceq (N,\bar{\rho})$ and so we can embed $(N,\bar{\rho})$ into $(\mathfrak{C},\bar{\sigma})$ over $M$ as an elementary $\mathcal{L}^G$-substructure.

Let
$$p_{123}(x_1,x_2,x_3):=\tp_{\mathcal{L}^G}^{\mathfrak{C}}(a'b'c_1'/M),$$
where tuple $a'b'c_1'$ is the image of the tuple $abc_1$ under the above embedding.

The image (under the above embedding) of $\acl_{\mathcal{L}^G}^{\mathfrak{C}}(Mab)$ is $\mathcal{L}^G_M$-isomorphic to $\acl_{\mathcal{L}^G}^{\mathfrak{C}}(Mab)$, and the image (under the above embedding) of $\acl_{\mathcal{L}^G}^{\mathfrak{C}}(Mac_1)$ is $\mathcal{L}^G_M$-isomorphic to 
$\acl_{\mathcal{L}^G}^{\mathfrak{C}}(Mac_1)$. Moreover, $f:(C,\bar{\sigma}^{f^{-1}})\to (\acl_{\mathcal{L}^G}^{\mathfrak{C}}(Mbc_2),\bar{\sigma})$ induces an $\mathcal{L}^G_M$-isomorphism between the image (under the above embedding) of $C$ and $\acl_{\mathcal{L}^G}^{\mathfrak{C}}(Mbc_2)$.
Therefore, a multiple use of Fact \ref{fact_type_descr} proves
$$\tp_{\mathcal{L}^G}^{\mathfrak{C}}(a'b'/M)=
\tp_{\mathcal{L}^G}^{\mathfrak{C}}(ab/M)=p_{12}(x_1,x_2),$$
$$\tp_{\mathcal{L}^G}^{\mathfrak{C}}(a'c_1'/M)=
\tp_{\mathcal{L}^G}^{\mathfrak{C}}(ac_1/M)=p_{13}(x_1,x_3),$$
$$\tp_{\mathcal{L}^G}^{\mathfrak{C}}(b'c_1'/M)=
\tp_{\mathcal{L}^G}^{\mathfrak{C}}(bc_2/M)=p_{23}(x_2,x_3),$$
and hence $p_{123}$ extends $p_{12}$, $p_{13}$ and $p_{23}$.

By $D'$ we denote the image of $D$ by the discussed above embedding. 
Let $i:D\to D'$ be the induced $\mathcal{L}^G_M$-isomorphism.
Because $D'$ and $D$ are $\mathcal{L}$-isomorphic (over $M$) $\mathcal{L}$-substructures of $\mathfrak{D}$, by Lemma \ref{fhat_lemma} there exists $\hat{i}\in\aut_{\mathcal{L}}(\mathfrak{D}/M)$ which extends $i$. We have the following 
$$c_1\ind_M^{\mathfrak{C}}ab\;\iff\;Gc_1\ind_M^{\mathfrak{D}}Ga Gb\;\Rightarrow\;
\hat{i}(Gc_1)\ind_M^{\mathfrak{D}}\hat{i}(GaGb)\;\Rightarrow$$
$$\Rightarrow\;i(Gc_1)\ind_M^{\mathfrak{D}}i(GaGb)\;\Rightarrow\;Gi(c_1)\ind_M^{\mathfrak{D}}Gi(a)Gi(b)\;\iff\;c_1'\ind_M^{\mathfrak{C}}a'b',$$
which finishes the proof of the theorem.
\end{proof}

\begin{theorem}\label{TG.simple}
If algebraic closures in $\mathfrak{D}$ split over $\mathfrak{C}$, then $T_G^{\mc}$ is simple and $\ind^{\circ}$ is the forking independence.
\end{theorem}

\begin{proof}
By Proposition \ref{6properties}, Theorem \ref{ind_thm_model} and Theorem \ref{simple.thm}.
\end{proof}

\begin{remark}
The theory $G-\tcf$ from \cite{nacfa} and the theory ACFA are supersimple but not stable, hence we should not expect that $T_G^{\mc}$ is usually stable. Moreover, the theory $\mathbb{Q}$ACFA from \cite{qacfa} is not supersimple (and the theory ACF is superstable), hence a general implication ``superstable $\Rightarrow$ supersimple" does not hold. The reader may consult Corollary \ref{simple.final}, where we argue for the aforementioned implication in the case of a finite group $G$ ("superstable + finite $G$ $\Rightarrow$ supersimple").
\end{remark}

\begin{remark}\label{alg.cls.split.important}
One could ask whether the assumption about algebraic closures in the statement of Theorem \ref{TG.simple} is necessary. If we consider the theory of existentially closed fields equipped with actions of the infinite dihedral group, $D_{\infty}-\tcf$, (as in Section 8. in \cite{sjogren}, which was axiomatized in \cite{ozlempiotr}) it turns out that these fields are PAC fields, but they are not bounded. Hence the theory $D_{\infty}-\tcf$ is not simple (Fact 2.6.7 in \cite{kim1}). Therefore, it is not possible to state Theorem \ref{TG.simple} without any assumptions related to boundedness and requiring that algebraic closures split is a property which is not stronger than being bounded (see Proposition \ref{bounded.split}).
\end{remark}

\begin{comment}
\begin{example}\label{example.empty2}
\textbf{NO EI for this set-up!}

If we turn back to Example \ref{example.empty}, we see that $T^{\mc}$, for the ``empty" theory $T$, is stable, eliminates quantifiers and imaginaries. It is not hard to see that algebraic closures split in the case of the ``empty" theory. Therefore we can say what is the forking independence in $T_G^{\mc}$ in this case (by Theorem \ref{TG.simple}):
$$a\ind^{\circ}_A B\quad\iff\quad Ga\ind^{\mathfrak{D}}_{GA}GB\quad\iff\quad Ga\cap GB\subseteq GA,$$
where $\ind^{\mathfrak{D}}$ is the forking independence in $T^{\mc}$ if $T$ is the ``empty" theory (the last equivalence can be shown by Theorem \ref{simple.thm}).

We already knew that $T_G^{\mc}$ is superstable (the last lines of Example \ref{example.empty}), but we can now count the SU-rank. Assume that $a\nind^{\circ}_A B$, which means that $Ga\cap GB\not\subseteq GA$. Hence, there exists $g,h\in G$ and $b\in B$ such that
$$ga=hb\not\in GA,$$
$$a=g^{-1}hb\not\in GA.$$
It follows $a\not\in GA$ and $a\in GB$, so $Ga\subseteq GB$. Therefore for any $B'\supseteq B$ we have
$$Ga\cap GB'\subseteq GB\cap GB'\subseteq GB$$
and so $a\ind^{\circ}_B B'$, which means that the SU-rank of $T_G^{\mc}$ is equal to $1$.
\end{example}
\end{comment}

\subsection{When algebraic closures split?}\label{G_finite}
The main result of this subsection (Proposition \ref{bounded.split}) shows that boundedness implies that algebraic closures split (i.e. requiring that algebraic closures split is a property not stronger than being by $\mathfrak{C}$ bounded). Moreover if $G$ is finite, then $\mathfrak{C}$ is bounded (in $\mathfrak{D}$).

\begin{prop}\label{finite.bounded}
If $G$ is finite and $(M,\bar{\sigma})\models T_G^{\mc}$, then $M$ is bounded (Definition \ref{PP.bounded}).
\end{prop}

\begin{proof}
Let $(M,\bar{\sigma})\preceq(\mathfrak{C},\bar{\sigma})$ (not necessarily small) and $\mathcal{H}:=\aut_{\mathcal{L}}(\acl_{\mathcal{L}}^{\mathfrak{D}}(M^G)/M)$.
By Lemma \ref{N_Galois}, $M\subseteq\acl_{\mathcal{L}}^{\mathfrak{D}}(M^G)$ and $M^G\subseteq M$ is Galois. We have the following short exact sequence
$$\xymatrix{1\ar[r] &\mathcal{H}
\ar[r] & \aut_{\mathcal{L}}(\acl_{\mathcal{L}}^{\mathfrak{D}}(M^G)/M^G)\ar[r]& \aut_{\mathcal{L}}(M/M^G)\ar[r] &1}.$$
Because $G$ is finite, by Proposition \ref{prop.generators}, it follows
$$\aut_{\mathcal{L}}(M/M^G)=\cl\big((\sigma_g)_{g\in G}\big)=(\sigma_g)_{g\in G},$$ so $\aut_{\mathcal{L}}(M/M^G)$ is finite.
By Proposition \ref{inv.bounded}, the profinite group $\aut_{\mathcal{L}}(\acl_{\mathcal{L}}^{\mathfrak{D}}(M^G)/M^G)$ is finitely generated. Therefore, $\mathcal{H}$ 
is a finite index subgroup of a finitely generated profinite group.
Since $\mathcal{H}$ is closed (by Fact \ref{krull.topology}), its complement is also closed as the union of finitely many copies of $\mathcal{H}$. Therefore the subgroup $\mathcal{H}$ is open and by \cite[Proposition 2.5.5]{ribzal}, it is also
finitely generated.
In a similar manner as in the proof of Lemma \ref{MG.bounded}, we can show that $M$ is bounded.
\end{proof}

\begin{prop}\label{bounded.split}
If the structure $\mathfrak{C}$ is bounded,
then algebraic closures in $\mathfrak{D}$ split over $\mathfrak{C}$. 
\end{prop}

\begin{proof}
Assume that $(M,\bar{\sigma})\preceq(\mathfrak{C},\bar{\sigma})$ is small.
\
\\
\\
\textbf{Claim} There exists an $\mathcal{L}$-substructure $D$ of $\mathfrak{D}$ such that
$(D,M)\preceq(\mathfrak{D},\mathfrak{C})$.
\
\\
Proof of the claim. 
We will construct $D$ 
as the union of a tower of $\mathcal{L}$-substructures $(D_n)_{n<\omega}$.
We choose one realization (in $\mathfrak{D}$) of each $\mathcal{L}\cup\lbrace\mathfrak{C}\rbrace$-formula over $M$ (here ``$\mathfrak{C}$" plays the role of a predicate added to the language) and let $D_1'$ be the $\mathcal{L}$-definable closure of 
the set consisting of all
these realizations. We find $D_1\equiv_M D_1'$ such that $D_1\ind_M^{\mathfrak{D}}\mathfrak{C}$.
We take $n\geqslant 1$ and assume
that we have chosen $D_i$ for each $i\leqslant n$. Now,	 let us choose one realization
of each $\mathcal{L}\cup\lbrace\mathfrak{C}\rbrace$-formula over $D_n$. By $D_{n+1}'$ we denote the $\mathcal{L}$-definable closure of the set of these realizations and by $D_{n+1}$ an $\mathcal{L}$-substructure of $\mathfrak{D}$ such that $D_{n+1}\equiv_{D_n}D_{n+1}'$ and $D_{n+1}\ind_{D_n}^{\mathfrak{D}}\mathfrak{C}$.

Let $D:=\bigcup\limits_{i<\omega}D_i$, which is an $\mathcal{L}$-substructure of $\mathfrak{D}$. We see that $D\ind_M^{\mathfrak{D}}\mathfrak{C}$, so
$D\cap\mathfrak{C}\subseteq\acl_{\mathcal{L}}^{\mathfrak{D}}(M)$, and it follows that
$$D\cap\mathfrak{C}\subseteq\acl_{\mathcal{L}}^{\mathfrak{D}}(M)\cap\mathfrak{C}=M,$$
hence $D\cap\mathfrak{C}=M$ and $(D,M)$ is an $\mathcal{L}\cup\lbrace\mathfrak{C}\rbrace$-substructure of $(\mathfrak{D},\mathfrak{C})$. By the construction, Tarski-Vaught test is satisfied and $(D,M)\preceq(\mathfrak{D},\mathfrak{C})$, which finishes the proof of the claim.

Let $M\subseteq B\subseteq\mathfrak{C}$, we want to show
$$\acl_{\mathcal{L}}^{\mathfrak{D}}(B)=\dcl_{\mathcal{L}}^{\mathfrak{D}}(\acl_{\mathcal{L}}^{\mathfrak{D}}(B)\cap\mathfrak{C},\acl_{\mathcal{L}}^{\mathfrak{D}}(M)).$$

Because $\mathfrak{C}$ is bounded (and $\dcl_{\mathcal{L}}^{\mathfrak{D}}(\mathfrak{C})=\mathfrak{C}$), by \cite[Proposition 2.5.(iv)]{Polkowska},
where for $(M,P)=(M_1,P_1)$ we substitute
$(\mathfrak{D},\mathfrak{C})$ and for $(M_0,P_0)$ we substitute $(D,M)$, we obtain
the last line on the page 178. in the proof of \cite[Lemma 3.8]{Polkowska}, i.e.
$$\acl_{\mathcal{L}}^{\mathfrak{D}}(\mathfrak{C})\subseteq\dcl_{\mathcal{L}}^{\mathfrak{D}}(\mathfrak{C}\acl_{\mathcal{L}}^{\mathfrak{D}}(M)).$$

Now we use the rest of the proof of \cite[Lemma 3.8]{Polkowska} for 
$(\bar{M},\bar{P})=(\mathfrak{D},\mathfrak{C})$, $(M_0,P_0)=(D,M)$ (size if $M_0$ or $P_0$ does not matter in the proof of \cite[Lemma 3.8]{Polkowska}) and 
$$A=\acl_{\mathcal{L}}^{\mathfrak{D}}(B)=\acl_{\mathcal{L}}^{\mathfrak{D}}(\acl_{\mathcal{L}}^{\mathfrak{D}}(B)\cap\mathfrak{C})$$
 and obtain
$$\acl_{\mathcal{L}}^{\mathfrak{D}}(B)\subseteq\dcl_{\mathcal{L}}^{\mathfrak{D}}(\acl_{\mathcal{L}}^{\mathfrak{D}}(B)\cap\mathfrak{C},\acl_{\mathcal{L}}^{\mathfrak{D}}(M)).$$
\end{proof}

\begin{cor}\label{G_finite_alg.cls}
If $G$ is finite, then algebraic closures in $\mathfrak{D}$ split over $\mathfrak{C}$.
\end{cor}

\begin{proof}
By Proposition \ref{finite.bounded} for $(M,\bar{\sigma})=(\mathfrak{C},\bar{\sigma})$ and by Proposition \ref{bounded.split}.
\end{proof}

\begin{cor}\label{simple.final}
\begin{enumerate}
\item
If $\mathfrak{C}$ is bounded (in particular if $G$ is finite), then $T_G^{\mc}$ is simple and the ternary relation $\ind^{\circ}$ (defined before Remark \ref{ind_acl_rem}) is the forking independence.

\item
If $T^{\mc}$ is superstable and $G$ is finite, then $T_G^{\mc}$ is supersimple.
\end{enumerate}
\end{cor}

\begin{proof}
We only need to prove the second item.
By the last sentence of the proof of the item (ii) in Proposition \ref{6properties}, we can choose $E$ to be finite.
\end{proof}

\begin{remark}
It was proved in \cite{nacfa} that the theory $G-\tcf$ is simple (for a finite group $G$). The proof proceeded by counting the SU-rank of invariants and showing that the structure of a field $K$ equipped with an action of a finite group $G$ can be interpreted in the (pure field) structure of the invariants. The method from \cite{nacfa} is more or less an incarnation of Corollary \ref{interpretable.inv}.
However, it
does not provide any description of the forking independence.
\end{remark}

\begin{example}\label{forking.formula}
The description of forking independence in $G-\tcf$ (Example \ref{gtcf1}) was not made in \cite{nacfa}, but now, after the main results, we can provide it.
Assume that $(\mathfrak{C},(\sigma_g)_{g\in G})$ is a monster model of the theory $G-\tcf$. Let $B,C\subseteq\mathfrak{C}$ be small subsets and let $a$ be a finite tuple from $\mathfrak{C}$.
By Corollary \ref{simple.final} it follows 
$$a\ind^{G-\tcf}_C B\;\;\iff\;\;a\ind^{\circ}_C B\;\;\iff\;\;Ga\ind^{\acf}_{GC}GB\;\;\iff$$
$$\iff\;\;\td(Ga/\langle GB,GC\rangle)=\td(Ga/\langle GC\rangle),$$
where $\langle GB,GC\rangle$ and $\langle GC\rangle$ denote subfields generated by $GB\cup GC$ and $GC$ respectively.
\end{example}

The following corollary was pointed out to us by Piotr Kowalski and summarizes the known results from the model theory of fields with group actions. In this corollary, $G$ is not necessarily finite. However we did not define $G-\tcf$ for infinite groups, it should be clear that if $G$ is not finite and the model companion of fields with actions of $G$ exists, then it is denoted by $G-\tcf$. To avoid confusion, we state the following fact using
the general ``$T_G^{\mc}$"-notation.

\begin{cor}\label{fields.cor}
Let $T$ be the theory of fields, $\mathcal{L}$ be the language of rings
and let $G$ be a group (not necessarily finite).
Assume that $T_G^{\mc}$ (= $G-\tcf$) exists and let $(K,(\sigma_g)_{g\in G})\models T_G^{\mc}$. The following are equivalent.
\begin{enumerate}
\item
The theory of $K$ in the language $\mathcal{L}$ is simple.

\item
The field $K$ is a bounded field.

\item
The theory of $(K,(\sigma_g)_{g\in G})$ in the language $\mathcal{L}^G$ is simple.
\end{enumerate}
\end{cor}

\begin{proof}
Since $(K,(\sigma_g)_{g\in G})\models T_G^{\mc}$, it follows that $K$ is a PAC field (see \cite[Theorem 3.]{sjogren} or our Proposition \ref{M.is.PAC}).
Implication from (1) to (2) is a well known fact for PAC fields, i.e.\cite[Fact 2.6.7]{kim1}. The passage from (2) to (3) is an instance of Corollary \ref{simple.final} (we just need to note that a monster model of $\theo(K,(\sigma_g)_{g\in G})$ is also bounded, e.g. \cite[Proposition 2.5.(v)]{PilPol}). The theory of $K$ in the language $\mathcal{L}$ is interpretable in $\theo(K,(\sigma_g)_{g\in G})$, hence (3) implies (1).
\end{proof}

\subsection{About elimination of imaginaries}
The following proof of the geometric elimination of imaginaries (i.e. every imaginary element is interalgebraic with a finite real tuple) is based on similar proofs from \cite[Paragraph 1.10]{acfa1}, \cite[Paragraph 2.9]{ChaPil} and an observation from \cite{hils0}. However, there are some tricky steps in our adaptation, therefore we include a whole proof.
We need to evoke one more fact (\cite[Lemma 1.4]{evahru}).

\begin{fact}[von Neuman's Lemma]\label{vNlemma}
Let $M$ be a large saturated structure. Let $X$ be $\emptyset$-definable, $e\in M$, $E=:\acl^M(e)\cap X$ and $\bar{a}\in X$. Then there is $\bar{b}$ such that $\tp(\bar{b}/Ee)=\tp(\bar{a}/Ee)$ and
$$\acl^M(E\bar{a})\cap\acl^M(E\bar{b})\cap X=	E.$$
\end{fact}

\noindent
For the proof of the above formulation of von Neuman's Lemma, the reader may consult Lemma 4.1 in notes to a seminar which were made by David Marker (\cite{markernotes}).

In the proof of the following theorem, we will use the notion of the ``fundamental order", which was described in \cite{lascarstab} and which will be denoted by $\leqslant_{\fo}$. 
In a nutshell: the fundamental order is an ordering on types which has maximal elements (\cite[Proposition 2.7]{lascarstab}).
To ensure the reader that the fundamental order can be applied in our situation we note a few facts and then describe how our situation corresponds to the content of \cite{lascarstab}.

\begin{remark}
Assume that $A\subseteq B$ are small
subsets of $\mathfrak{D}$.
Let $p\in S(A)$ and let $q\in S(B)$ be an extension of $p$. It follows $p\geqslant_{\fo} q$.
\end{remark}

\begin{fact}\label{fo_fact1}
Assume that $A\subseteq B$ are small, algebraically closed, subsets of $\mathfrak{D}$.
Let $p\in S(A)$ and let $q\in S(B)$ be an extension of $p$. The following are equivalent.
\begin{enumerate}
\item
Type $q$ is the non-forking extension of $p$.

\item
For every small $M\preceq N\preceq\mathfrak{D}$, such that $A\subseteq M$ and $B\subseteq N$, there exists $p_1\in S(M)$ such that the heir of $p_1$ on $N$ extends $q$.
\end{enumerate}
\end{fact}

\begin{proof}
We start with proving implication from (1) to (2). Assume that $q=\tp_{\mathcal{L}}^{\mathfrak{D}}(a/B)$ for some small $a\subset\mathfrak{D}$. Let $M\preceq N\preceq\mathfrak{D}$ be such that $A\subseteq M$ and $B\subseteq N$.
For $p_1$ we set $\tp_{\mathcal{L}}^{\mathfrak{D}}(a'/M)$ such that $a'\equiv_A a$ and $a'\ind_A^{\mathfrak{D}}M$ (the non-forking extension of $p$).
Heir of $p_1$ on $N$ is its non-forking extension, say $p_1|_N=\tp_{\mathcal{L}}^{\mathfrak{D}}(a^{\prime\prime}/N)$ 
satisfying $a^{\prime\prime}\equiv_M a'$ and $a^{\prime\prime}\ind_M^{\mathfrak{D}} N$. Our goal is to show that $a^{\prime\prime}\equiv_B a$.

Because $A\subseteq M$, $a^{\prime\prime}\equiv_M a'$ and $a'\ind_A^{\mathfrak{D}} M$, it follows that $a^{\prime\prime}\ind_A^{\mathfrak{D}}M$. By transitivity of $\ind$ we obtain that $a^{\prime\prime}\ind_A^{\mathfrak{D}} N$ and so $a^{\prime\prime}\ind_A^{\mathfrak{D}} B$. We know that $a^{\prime\prime}\equiv_M a'\equiv_A a$ and $a\ind_A^{\mathfrak{D}} B$. Therefore we can apply Corollary \ref{regular.PAPA} ($A\subseteq B$ is obviously regular).

Now we will prove implication from (2) to (1). Assume that for some $\varphi(x,y)\in\mathcal{L}(A)$ and some $b\subseteq B$ it is $\varphi(x,b)\in q$. We need show that $\varphi(x,b)$ does not divide over $A$. To see this we use equivalent condition from \cite[Corollary 8.4]{casasimpl}, i.e. $\varphi(x,b)\in q$ does not fork over any model $M$ containing $A$. Let $A\subseteq M$ be arbitrary and let $M\preceq N$ be such that $B\subseteq N$. By condition (2) there exists type $p_1\in S(M)$ such that its heir on $N$ contains $\varphi(x,b)\in q$. Therefore $\varphi(x,b)\in q$ does not fork over $M$.
\end{proof}

\begin{cor}
Assume that $A\subseteq B$ are small, algebraically closed, subsets of $\mathfrak{D}$.
Let $p\in S(A)$ and let $q\in S(B)$ be an extension of $p$. The following are equivalent.
\begin{enumerate}
\item
Type $q$ is the non-forking extension of $p$.

\item
It follows $p\sim_{\fo} q$ (i.e. $p\leqslant_{\fo} q$ and $q\leqslant_{\fo} p$).
\end{enumerate}
\end{cor}

\begin{proof}
By Fact \ref{fo_fact1} and Proposition 3.8 in \cite{lascarstab}.
\end{proof}

\begin{theorem}\label{imaginaries}
Let us assume that the ternary relation $\ind^{\circ}$ coincides with the forking independence in the theory $T_G^{\mc}$ (which implies that $T_G^{\mc}$ is simple).
Then the theory $T_G^{\mc}$ allows geometric elimination of imaginary elements.
\end{theorem}

\begin{proof}
Let $e$ be an element of the structure $(\mathfrak{C},\bar{\sigma})^{\eq}$ given by a $\emptyset$-definable function $f$ and a finite tuple $a\subseteq\mathfrak{C}$, i.e. $f(a)=e$. Let $C:=\acl_{\mathcal{L}^G}^{\mathfrak{C}^{\eq}}(e)\cap\mathfrak{C}$ and
$Q:=\lbrace a'\subseteq\mathfrak{C}\;|\;a'\models\tp_{\mathcal{L}^G}^{\mathfrak{C}}(a/C)\rbrace$.
\
\\
\textbf{Claim}
There is $c\in Q$ such that $a\ind_C^{\circ} c$ and $f(c)=e$.
\
\\
Proof of Claim. By Fact \ref{vNlemma} there exists $b_0\models\tp_{\mathcal{L}^G}^{\mathfrak{C}^{\eq}}(a/Ce)$ such that
$$C\subseteq\acl_{\mathcal{L}^G}^{\mathfrak{C}}(Ca)\cap\acl_{\mathcal{L}^G}^{\mathfrak{C}}(Cb_0)\subseteq
\acl_{\mathcal{L}^G}^{\mathfrak{C}^{\eq}}(Ca)\cap\acl_{\mathcal{L}^G}^{\mathfrak{C}^{\eq}}(Cb_0)\cap\mathfrak{C}\subseteq C.$$
We see that $f(b_0)=e$ and $b_0\in Q$. Choose, by \cite[Proposition 2.7]{lascarstab}, $b\in Q$ such that
\begin{itemize}
\item
$\acl_{\mathcal{L}^G}^{\mathfrak{C}}(Ca)\cap\acl_{\mathcal{L}^G}^{\mathfrak{C}}(Cb)=C$,

\item
$f(b)=e$,

\item
there is no $b'\in Q$ satisfying $f(b')=e$ and $\acl_{\mathcal{L}^G}^{\mathfrak{C}}(Ca)\cap\acl_{\mathcal{L}^G}^{\mathfrak{C}}(Cb')=C$ such that
$$\tp_{\mathcal{L}}^{\mathfrak{D}}\Big(Gb'/\acl_{\mathcal{L}}^{\mathfrak{D}}\big(G\cdot(Ca)\big)\Big)>_{\fo}\tp_{\mathcal{L}}^{\mathfrak{D}}\Big(Gb/\acl_{\mathcal{L}}^{\mathfrak{D}}\big(G\cdot(Ca)\big)\Big).$$

\end{itemize}
Let $c\models\tp_{\mathcal{L}^G}^{\mathfrak{C}}(b/\acl_{\mathcal{L}^G}^{\mathfrak{C}}(Ca))$ be such that $c\ind_{Ca}^{\circ} b$. It follows that $f(c)=e$ and
$$\acl_{\mathcal{L}^G}^{\mathfrak{C}}(Cc)\cap\acl_{\mathcal{L}^G}^{\mathfrak{C}}(Cab)\subseteq\acl_{\mathcal{L}^G}^{\mathfrak{C}}(Cc)\cap\acl_{\mathcal{L}^G}^{\mathfrak{C}}(Ca)=C.$$
Let $d\models\tp_{\mathcal{L}^G}^{\mathfrak{C}}(c/\acl_{\mathcal{L}^G}^{\mathfrak{C}}(Cb))$ be such that $d\ind_{Cb}^{\circ} a$. It follows that $f(d)=e$ and
$$\acl_{\mathcal{L}^G}^{\mathfrak{C}}(Cd)\cap\acl_{\mathcal{L}^G}^{\mathfrak{C}}(Cab)\subseteq\acl_{\mathcal{L}^G}^{\mathfrak{C}}(Cd)\cap\acl_{\mathcal{L}^G}^{\mathfrak{C}}(Cb)=C.$$
Because $\acl_{\mathcal{L}}^{\mathfrak{D}}(G(Cab))\supseteq\acl_{\mathcal{L}}^{\mathfrak{D}}(G(Cb))$, we have
$$\tp_{\mathcal{L}}^{\mathfrak{D}}\Big(Gc/\acl_{\mathcal{L}}^{\mathfrak{D}}\big(G(Cab)\big)\Big)\leqslant_{\fo}\tp_{\mathcal{L}}^{\mathfrak{D}}\Big(Gc/\acl_{\mathcal{L}}^{\mathfrak{D}}\big(G(Cb)\big)\Big).$$
There is $h_1\in\aut_{\mathcal{L}^G}(\mathfrak{C}/\acl_{\mathcal{L}^G}^{\mathfrak{C}}(Cb))$ sending $c$ to $d$, and hence sending $Gc$ to $Gd$. Note that $\acl_{\mathcal{L}^G}^{\mathfrak{C}}(Cb)\subseteq\mathfrak{C}$ is regular. By Fact \ref{regular} we can extend $h_1$ and $\id:\acl_{\mathcal{L}}^{\mathfrak{D}}(G(Cb))\to \acl_{\mathcal{L}}^{\mathfrak{D}}(G(Cb))$ 
to $\hat{h_1}\in\aut_{\mathcal{L}}(\mathfrak{D}/\acl_{\mathcal{L}}^{\mathfrak{D}}(G(Cb)))$ sending $Gc$ to $Gd$, which implies
$$\tp_{\mathcal{L}}^{\mathfrak{D}}\Big(Gc/\acl_{\mathcal{L}}^{\mathfrak{D}}\big(G(Cb)\big)\Big)=\tp_{\mathcal{L}}^{\mathfrak{D}}\Big(Gd/\acl_{\mathcal{L}}^{\mathfrak{D}}\big(G(Cb)\big)\Big).$$
Because $d\ind_{Cb}^{\circ}a$ and so $Gd\ind_{\acl_{\mathcal{L}}^{\mathfrak{D}}(G(Cb))}^{\mathfrak{D}}\acl_{\mathcal{L}}^{\mathfrak{D}}(G(Cab))$, it follows
$$\tp_{\mathcal{L}}^{\mathfrak{D}}\Big(Gd/\acl_{\mathcal{L}}^{\mathfrak{D}}\big(G(Cb)\big)\Big)\sim_{\fo}\tp_{\mathcal{L}}^{\mathfrak{D}}\Big(Gd/\acl_{\mathcal{L}}^{\mathfrak{D}}\big(G(Cab)\big)\Big).$$
From the last tree exposed equations, we see that
$$\tp_{\mathcal{L}}^{\mathfrak{D}}\Big(Gc/\acl_{\mathcal{L}}^{\mathfrak{D}}\big(G(Cab)\big)\Big)\leqslant_{\fo}\tp_{\mathcal{L}}^{\mathfrak{D}}\Big(Gd/\acl_{\mathcal{L}}^{\mathfrak{D}}\big(G(Cab)\big)\Big).$$
Because $c\ind_{Ca}^{\circ}b$, it follows that
$$\tp_{\mathcal{L}}^{\mathfrak{D}}\Big(Gc/\acl_{\mathcal{L}}^{\mathfrak{D}}\big(G(Cab)\big)\Big)\sim_{\fo}\tp_{\mathcal{L}}^{\mathfrak{D}}\Big(Gc/\acl_{\mathcal{L}}^{\mathfrak{D}}\big(G(Ca)\big)\Big).$$
Again, there is $h_2\in\aut_{\mathcal{L}^G}(\mathfrak{C}/\acl_{\mathcal{L}^G}^{\mathfrak{C}}(Ca))$ sending $Gb$ to $Gc$, hence we can extend it by Fact \ref{regular} and obtain
$$\tp_{\mathcal{L}}^{\mathfrak{D}}\Big(Gc/\acl_{\mathcal{L}}^{\mathfrak{D}}\big(G(Ca)\big)\Big)=\tp_{\mathcal{L}}^{\mathfrak{D}}\Big(Gb/\acl_{\mathcal{L}}^{\mathfrak{D}}\big(G(Ca)\big)\Big).$$
Because $\acl_{\mathcal{L}}^{\mathfrak{D}}(G(Cab))\supseteq\acl_{\mathcal{L}}^{\mathfrak{D}}(G(Ca))$, we have
$$\tp_{\mathcal{L}}^{\mathfrak{D}}\Big(Gd/\acl_{\mathcal{L}}^{\mathfrak{D}}\big(G(Cab)\big)\Big)\leqslant_{\fo} \tp_{\mathcal{L}}^{\mathfrak{D}}\Big(Gd/\acl_{\mathcal{L}}^{\mathfrak{D}}\big(G(Ca)\big)\Big).$$
We have all ingredients, we use a simplified notation to describe the situation:
$$b/Ca=c/Ca\sim_{\fo}c/Cab\leqslant_{\fo}d/Cab\leqslant_{\fo}d/Ca,$$
hence, by the choice of $b$, inequalities in the above line are in fact ``$\sim_{\fo}$". Therefore
$$\tp_{\mathcal{L}}^{\mathfrak{D}}\Big(Gc/\acl_{\mathcal{L}}^{\mathfrak{D}}\big(G(Cab)\big)\Big)\sim_{\fo}\tp_{\mathcal{L}}^{\mathfrak{D}}\Big(Gd/\acl_{\mathcal{L}}^{\mathfrak{D}}\big(G(Cab)\big)\Big),$$
which implies
$$\tp_{\mathcal{L}}^{\mathfrak{D}}\Big(Gc/\acl_{\mathcal{L}}^{\mathfrak{D}}\big(G(Cab)\big)\Big)\sim_{\fo}\tp_{\mathcal{L}}^{\mathfrak{D}}\Big(Gc/\acl_{\mathcal{L}}^{\mathfrak{D}}\big(G(Cb)\big)\Big).$$
The last thing occurs only if $Gc\ind_{G(Cb)}^{\mathfrak{D}}G(Cab)$.

Introduce $p:=\tp_{\mathcal{L}}^{\mathfrak{D}}(Gc/\acl_{\mathcal{L}}^{\mathfrak{D}}(G(Cab))\cap\mathfrak{C})$, which is, by Corollary \ref{reg.stationary}, a stationary type. It follows $\cb(p)\subseteq\dcl_{\mathcal{L}}^{\mathfrak{D}}(\acl_{\mathcal{L}}^{\mathfrak{D}}(G(Cab))\cap\mathfrak{C})\subseteq\mathfrak{C}$. By \cite[Remark 2.26]{anandgeometric}, because $p$ does not fork over $G(Cb)$, we have $\cb(p)\subseteq\acl_{\mathcal{L}}^{\mathfrak{D}}(G(Cb))$ and because $p$ does not fork over $G(Ca)$, we have $\cb(p)\subseteq\acl_{\mathcal{L}}^{\mathfrak{D}}(G(Ca))$. Therefore
$$\cb(p)\subseteq\acl_{\mathcal{L}}^{\mathfrak{D}}(G(Ca))\cap\acl_{\mathcal{L}}^{\mathfrak{D}}(G(Cb))\cap\mathfrak{C}=\acl_{\mathcal{L}^G}^{\mathfrak{C}}(Ca)\cap\acl_{\mathcal{L}^G}^{\mathfrak{C}}(Cb)=C.$$
Hence $p$ does not fork over $C$, $Gc\ind_C^{\mathfrak{D}}G(Cab)$ and therefore $c\ind_C^{\circ}ab$. Finally, we obtained $c\in Q$ such that $a\ind_C^{\circ} c$ and $f(c)=e$. This is the end of the proof of Claim.

Note that $e\in\dcl_{\mathcal{L}^G}^{\mathfrak{C}^{\eq}}(a)\cap\dcl_{\mathcal{L}^G}^{\mathfrak{C}^{\eq}}(c)$, but $a\ind_C^{\circ} c$ implies
that
$\acl_{\mathcal{L}^G}^{\mathfrak{C}^{eq}}(a)\cap\acl_{\mathcal{L}^G}^{\mathfrak{C}^{eq}}(c)\subseteq\acl_{\mathcal{L}^G}^{\mathfrak{C}^{\eq}}(C)$. Therefore $e\in\acl_{\mathcal{L}^G}^{\mathfrak{C}^{\eq}}(C)=\acl_{\mathcal{L}^G}^{\mathfrak{C}^{\eq}}(\acl_{\mathcal{L}^G}^{\mathfrak{C}^{\eq}}(e)\cap\mathfrak{C})$.
\end{proof}

For the final remark of this subsection, we need the following easy observation.

\begin{lemma}\label{finite.codes}
Theory $T_G^{\mc}$ codes finite sets.
\end{lemma}

\begin{proof}
Let $A\subseteq\mathfrak{C}$ be finite. As a finite subset of $\mathfrak{D}$ it has a code $b_A\subseteq\mathfrak{D}$. Because $A$ is $A$-definable, it follows that $b_A\subseteq\dcl_{\mathcal{L}}^{\mathfrak{D}}(A)\subseteq\dcl_{\mathcal{L}}^{\mathfrak{D}}(\mathfrak{C})=\mathfrak{C}$. Assume that $f\in\aut_{\mathcal{L}^G}(\mathfrak{C})$. We extend $f$ to $\hat{f}\in\aut_{\mathcal{L}}(\mathfrak{D})$ (by Lemma \ref{fhat_lemma})
and obtain
$$f(A)=A\;\;\iff\;\;\hat{f}(A)=A\;\;\iff\;\;\hat{f}(b_A)=b_A\;\;\iff\;\;f(b_A)=b_A,$$
where the middle part follows from definition of being a code in $\mathfrak{D}$.
\end{proof}

\begin{remark}\label{no.full.EI}
\begin{enumerate}
\item
One could ask about the weak elimination of imaginaries in $T_G^{\mc}$. The theory CCMA from \cite{hils0} is an example of $T_G^{\mc}$ which fits into our assumptions, but does not eliminate imaginaries (Corollary 3.6 in \cite{hils0}), hence does not allow the weak elimination of imaginaries - otherwise by \cite[Proposition 1.7]{casfar} and Lemma \ref{finite.codes}, we would obtain a contradiction.

\item
By Theorem \ref{imaginaries} and Lemma \ref{finite.codes}, we know that $G-\tcf$ has the geometric elimination of imaginaries and codes finite tuples, but we do not know whether is has the weak elimination of imaginaries. It turns out that is enough to add a finite tuple to the language to obtain the full elimination of imaginaries (see Theorem 4.10 in \cite{nacfa}).
\end{enumerate}
\end{remark}

\subsection{Simplicty of invariants}
In the first versions of this paper the following part was placed at the end of Subsection \ref{subsec:PAC}, because previously we assumed existence of $T_G^{\mc}$ in the whole paper and not only in Section \ref{sec:forking}. It seems that, unlike other results of Subsection \ref{subsec:PAC}, simplicity of invariants needs existence of some model complete theory with a $G$-action. Therefore we moved statements about simplicity of invariants here, at the end of Section \ref{sec:forking}, where existence of $T_G^{\mc}$ is assumed.

In the following subsection we assume also that $G$ is finitely generated, thus  invariants are a definable subset of $\mathfrak{C}$. We will give two different arguments for simplicity of invariants: one using simplicity of $T_G^{\mc}$ (if algebraic closures split), the other one using results about invariants obtained in Subsection \ref{subsec:PAC}.

Now, we briefly describe how to pass between two worlds: $\mathcal{L}$-structure $M^G$ and $\mathcal{L}^G$-structure $(M,(\sigma_g)_{g\in G})$, where 
$(M,(\sigma_g)_{g\in G})\preceq (\mathfrak{C},(\sigma_g)_{g\in G})$ (not necessarily small).
Let $\varphi(x)$ be a $\mathcal{L}$-formula, $c\in M^G$.
We claim that there is an $\mathcal{L}^G$-formula $\varphi^G(x)$ such that
\begin{equation}\label{translation1}
M^G\models\varphi(c)\;\iff\; M\models\varphi^G(c).
\end{equation}
Definition of $\varphi^G(x)$ is given recursively.
If $\varphi(x)$ is atomic, set $\varphi^G(x)$ to be
$$\varphi(x)\,\wedge\,\bigwedge\limits_{g\text{ is generator of }G}\sigma_g(x)=x.$$
We deal with adding quantifiers to a $\mathcal{L}$-formula $\psi(x,y)$ in the following manner
$$(\forall y)\,(\,\psi(x,y)\,) \text{ turns into } (\forall y)\,\big(\bigwedge\limits_{g\text{ is generator of }G}\sigma_g(y)=y\rightarrow \psi^G(x,y)\big),$$
$$(\exists y)\,(\,\psi(x,y)\,) \text{ turns into } (\exists y)\,\big(\bigwedge\limits_{g\text{ is generator of }G}\sigma_g(y)=y\,\wedge\, \psi^G(x,y)\big).$$
Now, if $\varphi(x)$ and $\psi(y)$ are $\mathcal{L}$-formulas, we set
\begin{IEEEeqnarray*}{rCl}
(\varphi\,\wedge\,\psi)^G(x,y) & \text{ to be } & \varphi^G(x)\,\wedge\,\psi^G(y), \\
(\varphi\,\vee\,\psi)^G(x,y) & \text{ to be } & \varphi^G(x)\,\vee\,\psi^G(y),\\
(\neg \varphi)^G(x) & \text{ to be } & \neg\varphi^G(x).
\end{IEEEeqnarray*}
In other words we restrict domain of every variable (free or not) to the definable subset  of invariants of the $G$-action, which is also an $\mathcal{L}$-substructure (see Remark \ref{inv.perfect}).

\begin{cor}\label{cor.MG.simpe}
If algebraic closures in $\mathfrak{D}$ split over $\mathfrak{C}$,
then the theory $\theo_{\mathcal{L}}(M^G)$ is simple.
\end{cor}

\begin{proof}
Using translation from \ref{translation1}, we see that
$(M,(\sigma_g)_{g\in G})\preceq (\mathfrak{C},(\sigma_g)_{g\in G})$ implies
$M^G\preceq\mathfrak{C}^G$. Therefore it is enough to show simplicity of $\theo_{\mathcal{L}}(\mathfrak{C}^G)$, which is straightforward, since $\mathfrak{C}^G$ is interpretable in $(\mathfrak{C},(\sigma_g)_{g\in G})$, which is simple by Theorem \ref{TG.simple}.
\end{proof}

Someone could ask whether, in the above corollary, it is necessary to assume that algebraic closures split. The following theorem, which uses results of Subsection \ref{subsec:PAC}, does not require that algebraic closures split and still states that theory of invariants is simple.

\begin{theorem}\label{MG.simple}
The theory of $M^G$ in the language $\mathcal{L}$ is simple. If moreover $T'$ is superstable, then $\theo(M^G)$ is supersimple.
\end{theorem}

\begin{proof}
By Lemma \ref{MG.bounded}, the $\mathcal{L}$-structure $M^G$ is bounded. By Proposition \ref{inv.pac}, $M^G$ is a PAC structure. We need to fix some monster model $\bar{P}$ of $\theo(M^G)$, which will be a bounded PAC $\mathcal{L}$-substructure of $\mathfrak{D}$.
After doing this, Proposition \ref{PACvsPACPP}.(1) will assure us that we are dealing with bounded PAC substructure in the sense of \cite{Polkowska}.
Having a monster model $\bar{P}$, which is a bounded PAC substructure of $\mathfrak{D}$, it is enough to use results of Section 3. in \cite{Polkowska}, page 177., which lead to simplicity/supersimplicity (we do not assume here that ``PAC property is first order" - as in \cite{Polkowska} - but instead of it we construct a proper monster model, consult the beginning of Section 3. in \cite{Polkowska}).

If we consider $(\bar{M},\bar{P}):=(\mathfrak{D},\bar{P})$ then we match the requirements made at the beginning of Section 3. in \cite{Polkowska}, page 177.
Therefore, we can use \cite[Corollary 3.22]{Polkowska} and obtain the thesis of the theorem.
We only need to obtain $\bar{P}$.

Take $\bar{P}\succeq M^G$ which is $\kappa$-saturated and $\kappa$-strongly homogeneous. Without loss of generality we can assume that $|\bar{P}|<\kappa_{\mathfrak{C}}$. By the translation from \ref{translation1}, we see that $M^G\preceq\mathfrak{C}^G$.
Moreover, it can be shown, using the translation \ref{translation1}, that $\mathfrak{C}^G$ is $\kappa_{\mathfrak{C}}$-saturated. If $\mathfrak{C}^G$ is $\kappa_{\mathfrak{C}}$-saturated, then it is also $\kappa_{\mathfrak{C}}$-universal, hence $\bar{P}$ can be elementarily embedded in $\mathfrak{C}^G$.
By $\bar{P}\preceq\mathfrak{C}^G$ and $\dcl_{\mathcal{L}}^{\mathfrak{D}}(\mathfrak{C}^G)=\mathfrak{C}^G$, it follows that $\dcl_{\mathcal{L}}^{\mathfrak{D}}(\bar{P})=\bar{P}$.
Our goal is to show that $\bar{P}$ is a bounded PAC $\mathcal{L}$-substructure of $\mathfrak{D}$.
\
\\
\\
\textbf{Claim 1:} $\bar{P}$ is bounded.
\\
Proof of the claim: 
Note that $\bar{P}\subseteq\mathfrak{C}^G$ is regular (Lemma \ref{lang410}).
By Proposition \ref{inv.bounded}, $\aut_{\mathcal{L}}(\acl_{\mathcal{L}}^{\mathfrak{D}}(\mathfrak{C}^G)/\mathfrak{C}^G)$ is finitely generated as a profinite group. Consider the following restriction map
$$\xymatrixcolsep{3.5pc}\xymatrix{\aut_{\mathcal{L}}(\acl_{\mathcal{L}}^{\mathfrak{D}}(\mathfrak{C}^G)/\mathfrak{C}^G) \ar[r]^(.55){|_{\acl_{\mathcal{L}}^{\mathfrak{D}}(\bar{P})}} & \aut_{\mathcal{L}}(\acl_{\mathcal{L}}^{\mathfrak{D}}(\bar{P})/\bar{P}) }.$$
If $f\in\aut_{\mathcal{L}}(\acl_{\mathcal{L}}^{\mathfrak{D}}(\bar{P})/\bar{P})$, then Fact \ref{regular} allows us to construct $h\in\aut_{\mathcal{L}}(\mathfrak{D})$ such that $h|_{\acl_{\mathcal{L}}^{\mathfrak{D}}(\bar{P})}=f$ and $h|_{\mathfrak{C}^G}=\id_{\mathfrak{C}^G}$. Therefore, $f$ is the restriction of $h|_{\acl_{\mathcal{L}}^{\mathfrak{D}}(\mathfrak{C}^G)}\in \aut_{\mathcal{L}}(\acl_{\mathcal{L}}^{\mathfrak{D}}(\mathfrak{C}^G)/\mathfrak{C}^G)$ and the aforementioned restriction map is an epimorphism of profinite groups. We see that also $\aut_{\mathcal{L}}(\acl_{\mathcal{L}}^{\mathfrak{D}}(\bar{P})/\bar{P})$ is finitely generated as a profinite group.
Similarly as in the proof of Lemma \ref{MG.bounded}, we show that $\bar{P}$ is bounded.
\
\\
\\
\textbf{Claim 2:} $\bar{P}$ is PAC.
\\
Proof of the claim: 
This fact with proof was pointed to us by Martin Hils.
Assume that for some small $N\subseteq\mathfrak{D}$, the extension $\bar{P}\subseteq N$ is regular. Let $\varphi(x,y)$ be a quantifier free $\mathcal{L}$-formula, $a\in\bar{P}^{|x|}$, and let $N\models\varphi(a,n)$ for some $n\in N^{|y|}$. There is $f\in\aut_{\mathcal{L}}(\mathfrak{D}/\bar{P})$
such that for  $N':=f(N)$ it follows
$$N'\ind_{\bar{P}}^{\mathfrak{D}} \mathfrak{C}^G.$$
The extension $\bar{P}\subseteq N'$ is regular, hence by Lemma \ref{lang413}, also $\mathfrak{C}^G\subseteq \mathfrak{C}^G N'$ is regular.  Thus 
$\mathfrak{C}^G\subseteq \dcl_{\mathcal{L}}^{\mathfrak{D}}(N',\mathfrak{C}^G)$ is a regular extension. Because $\mathfrak{C}^G$ is a PAC substructure of $\mathfrak{D}$ (Proposition \ref{inv.pac}), it follows $\mathfrak{C}^G\leqslant_1 \dcl_{\mathcal{L}}^{\mathfrak{D}}(N',\mathfrak{C}^G)$.
Since $\bar{P}\preceq\mathfrak{C}^G$, we have that $\bar{P}\leqslant_1 \dcl_{\mathcal{L}}^{\mathfrak{D}}(N',\mathfrak{C}^G)$.
Note that $\dcl_{\mathcal{L}}^{\mathfrak{D}}(N',\mathfrak{C}^G)\models \varphi(a,f(n))$, therefore $\bar{P}\models \exists x\;\varphi(a,x)$.
\end{proof}

\begin{cor}\label{interpretable.inv}
If $G$ is finitely generated and the theory $\theo(M,(\sigma_g)_{g\in G})$ is interpretable in $\theo(M^G)$, then the theory $\theo(M,(\sigma_g)_{g\in G})$ is simple. If moreover $T^{\mc}$ is superstable, then $\theo(M,(\sigma_g)_{g\in G})$ is supersimple.
\end{cor}

\section{Postlude}\label{postlude}
We recall Question \ref{q220} from the beginning of this paper.

\begin{question}
What assumptions about $T$ and $G$ assert existence of $T_G^{\mc}$?
\end{question} 

After all we can formulate a conjecture.

\begin{conj}
If $T^{\mc}$ exists and $G$ is finite, then $T_G^{\mc}$ exists.
\end{conj}

\noindent
The strategy to proof the above conjecture could involve the notion of the Kaiser hull. Mainly by \cite[Satz 12.]{kaiser1} we have an explicit description of the Kaiser hull for inductive theories. By \cite{kaiser1}, we know that if a model companion exists, then it is equal to Kaiser hull, hence problem reduces to check whether the Kaiser hull, given in \cite[Satz 12.]{kaiser1}, is model complete. Our intuition says that it should be, because for finite $G$ we need to consider only finitely many additional variables, hence situation does not change a lot.

Other interesting question is whether $T_G^{\mc}$ is a ``generic" simple theory. We see that the most studied simple theory is ACFA, which is of the shape $T_G^{\mc}$ for $T=$``theory of fields" and $G=\mathbb{Z}$.

\begin{question}
How ``many" simple theories appear as $T_G^{\mc}$ for some stable theory $T$ and some group $G$?
\end{question}

Next aim which should be considered as priority is answer to the following question.

\begin{question}
What properties of $T_G^{\mc}$ say about $T$ (or $T^{\mc}$)? Are there some invariants or canonical reductions of the logical structure of models of $T$ which can be observed in the structure of the theory $T_G^{\mc}$?
\end{question}

Now the reader may wish to choose some theory $T$, which has a stable model companion $T^{\mc}$, which allows to eliminate quantifiers and has elimination of imaginaries. Prove existence of $T_G^{\mc}$ and use developed by us tool to study his favourite stable theory with added dynamics.
We end this paper with the final sentence: Have fun!

\bibliographystyle{plain}
\bibliography{1nacfa2}

\end{document}